\documentclass[10pt,reqno]{amsart}

\usepackage[dvipsnames]{xcolor}
\usepackage{amsmath,amssymb,amsthm,amsfonts,enumerate,tikz,bm}
\usepackage{mathrsfs}
\usetikzlibrary{matrix,arrows,positioning,calc}
\usepackage[all]{xy}
\usepackage[bookmarksnumbered,colorlinks]{hyperref}
\usepackage{url}
\numberwithin{equation}{section}
\usepackage{graphicx}
\usepackage[T1]{fontenc}
\usepackage{textcomp}
\usepackage{palatino,helvet}

\usepackage{footmisc}

\newtheorem{definition}{Definition}[section]
\newtheorem{remark}[definition]{Remark}
\newtheorem{example}[definition]{Example}
\newtheorem{theorem}[definition]{Theorem}
\newtheorem{proposition}[definition]{Proposition}
\newtheorem{lemma}[definition]{Lemma}
\newtheorem{corollary}[definition]{Corollary}
\theoremstyle{remark}

\usepackage{paralist}

\usepackage{scrextend}
\deffootnote{2em}{0em}{\thefootnotemark\quad}

\newcommand{\Ext}{\mathrm{Ext}}

\newcommand{\Hom}{\mathrm{Hom}}

\newcommand{\A}{\mathcal{A}}

\newcommand{\Le}{\mathcal{L}}

\newcommand{\X}{\mathcal{X}}
\newcommand{\Y}{\mathcal{Y}}

\newcommand{\pd}{\mathrm{pd}}

\newcommand{\id}{\mathrm{id}}
\newcommand{\Gid}{\mathrm{Gid}}
\newcommand{\resdim}{\mathrm{resdim}}

\newcommand{\coresdim}{\mathrm{coresdim}}

\newcommand{\Ker}{\mathrm{Ker}}

\newcommand{\GF}{\mathcal{GF}}

\newcommand{\Gfd}{\mathrm{Gfd}}

\newcommand{\Tor}{\mathrm{Tor}}

\newcommand{\GI}{\mathcal{GI}}
\newcommand{\Coker}{\mathrm{CoKer}}

\newcommand{\Ch}{\mathsf{Ch}}

\usepackage{latexsym,amssymb,amscd} 
\usepackage{amsmath}
\usepackage[all]{xypic}
\usepackage{amsthm}

\numberwithin{equation}{section}

\begin{document}
\title[Relative Gorenstein flat modules]{Homological and homotopical aspects of Gorenstein flat modules and complexes relative to duality pairs}
%\thanks{}
\author{V\'ictor Becerril}
\address[V. Becerril]{Centro de Ciencias Matem\'aticas. Universidad Nacional Aut\'onoma de M\'exico. 
 CP58089. Morelia, Michoac\'an, M\'EXICO}
\email{victorbecerril@matmor.unam.mx}

\author{Marco A. P\'erez}
\address[M. A. P\'erez]{Instituto de Matem\'atica y Estad\'istica ``Prof. Ing. Rafael Laguardia''. Facultad de Ingenier\'ia. Universidad de la Rep\'ublica. CP11300. Montevideo, URUGUAY}
\email{mperez@fing.edu.uy}
%\date{}
\thanks{2020 MSC: 18G10, 18G20, 18G25.}
\thanks{Key Words: Relative Gorenstein flat modules, relative Gorenstein flat dimensions, duality pairs, cotorsion pairs, model category structures, recollements.}

%%%%%%%%%%%%%%%%%%%%%%%%%%%%%%%%%%%%
%%%%%%%%%%%%%%%%%%%%%%%%%%%%%%%%%%%%
%%%%%%%%%%%%%%%%%%%%%%%%%%%%%%%%%%%%
%%%%%%%%%%%%%%%%%%%%%%%%%%%%%%%%%%%%

\begin{abstract} 
We study homological and homotopical aspects of Gorenstein flat modules over a ring with respect to a duality pair $(\mathcal{L,A})$. These modules are defined as cycles of exact chain complexes with components in $\mathcal{L}$ which remain exact after tensoring by objects in 
\[
\mathcal{A} \cap \Big( \bigcap_{i \in \mathbb{Z}_{> 0}} {\rm Ker}({\rm Ext}^i_{R^{\rm o}}(-,\mathcal{A})) \Big).
\] 
In the case where $(\mathcal{A,L})$ is also a duality pair and $\mathcal{L}$ is closed under extensions, (co)products, $R \in \mathcal{L}$, and $\mathcal{A}$ is the right half of a hereditary complete cotorsion pair, we prove that these relative Gorenstein flat modules are closed under extensions, and that the corresponding Gorenstein flat dimension is well behaved in the sense that it recovers many of the properties and characterizations of its (absolute) Gorenstein flat counterpart. The latter in turn is a consequence of a Pontryagin duality relation that we show between these relative Gorenstein flat modules and certain Gorenstein injective modules relative to $\mathcal{A}$. We also find several hereditary and cofibrantly generated abelian model structures from these Gorenstein flat modules and complexes relative to $(\mathcal{L,A})$. At the level of chain complexes, we find three recollements between the homotopy categories of these model structures, along with several derived adjunctions connecting these recollements. 
\end{abstract}  
\maketitle
%\centerline{}

%%%%%%%%%%%%%%%%%%%%%%%%%%%%%%%%%%%%%
%%%%%%%%%%%%%%%%%%%%%%%%%%%%%%%%%%%%%
%%%%%%%%%%%%%%%%%%%%%%%%%%%%%%%%%%%%%
%%%%%%%%%%%%%%%%%%%%%%%%%%%%%%%%%%%%%

\section{Introduction}

Let $R$ be an associative ring with identity. Gorenstein flat $R$-modules were introduced by Enochs, Jenda and Torrecillas in \cite{EJT93}. These are defined as cycles of exact chain complexes of flat left $R$-modules, which remains exact after tensoring by injective right $R$-modules. These modules and some of its properties were studied by Holm in \cite{Holm04}. In particular, he proved that over a right coherent ring, Gorenstein flat left $R$-modules are closed under extensions. The validity of this closure property over an arbitrary ring was an open problem until 2020, when it was settled in the positive by \v{S}aroch and \v{S}\v{t}ov\'{\i}\v{c}ek in \cite{SS}. Furthermore, concerning homotopical aspects, these authors also proved that the Gorenstein flat left $R$-modules form the class of cofibrant objects in a cofibrantly generated abelian model structure on the category of modules over an arbitrary ring, generalizing thus the cases for Iwanaga-Gorenstein rings (see Gillespie and Hovey \cite[Coroll. 3.13]{GillespieHovey10}) and for coherent rings (see Gillespie \cite[Thm. 3.3]{GillespieStable17}). 

Gorenstein flat modules have then turned out to be key in Gorenstein homological algebra, for being a good analog for flat modules, but also for their rich interactions with Gorenstein projective and Gorenstein injective modules. Recently, there has been an increasing interest for generalizations of Gorenstein flat modules, such as the Gorenstein AC-flat modules presented by Bravo, Estrada and Iacob in \cite{BEI}. The latter, as one should expect, have appealing interactions with the Gorenstein AC-projective modules, defined by Bravo, Gillespie and Hovey in \cite{BGH14}. Further generalizations have also been defined and studied in \cite{EIP20} and \cite{WangYangZhu19}. In the former reference, Estrada, Iacob and the second named author of the present article study the Gorenstein $\mathcal{B}$-flat modules, that is, cycles of complexes of flat left $R$-modules which remain exact after tensoring by modules in a class $\mathcal{B}$ of right $R$-modules. Under certain conditions for $\mathcal{B}$ (namely, that $\mathcal{B}$ is closed under products and contains an elementary cogenerator of its definable closure), these modules are closed under extensions, and also have many interesting homological properties such as the construction of Gorenstein $\mathcal{B}$-flat covers. In the latter reference, Wang, Yang and Zhu investigate Gorenstein flat modules relative to a duality pair $(\mathcal{L,A})$. The required conditions that $(\mathcal{L,A})$ needs to fulfill so that these relative Gorenstein flat modules are closed under extensions are that $(\mathcal{L,A})$ is a perfect duality, $\mathcal{L}$ is closed under epikernels, $\mathcal{A}$ is closed under products, and  $\Tor^{R}_i(A,L) = 0$ for every $A \in \mathcal{A}$, $L \in \mathcal{L}$ and $i \in \mathbb{Z}_{> 0}$. 

It is then natural to think that if one is looking for a relativization for Gorenstein flat modules with the aim of showing results within the context of homotopical algebra (such as the construction of relative versions of the flat stable module category \cite[Coroll. 3.4]{GillespieStable17}), then it is important to look for sufficient conditions under which these modules are closed under extensions, since this is a minimal requirement if one wants to obtain abelian model category structures. Indeed, the present article is in part motivated by the following related question, posted by Estrada and Iacob during a research stay Luminy, France, while working on \cite{EIP20} with the second named author: if we consider the class of cycles of exact complexes of level modules which remain exact after tensoring by absolutely clean modules (see Bravo, Gillespie and Hovey \cite[Def. 2.6]{BGH14}), is this class closed under extensions? We shall give a partial answer to this (see Example \ref{ex:FPinfty} and Corollary \ref{resolvente}), by using the fact that (level modules, absolutely clean modules) is a duality pair of a certain type. 

In this article, we aim to continue with the study of Gorenstein flat modules relative to duality pairs. Specifically, we shall consider duality pairs $(\mathcal{L,A})$ such that $\mathcal{L}$ is closed under extensions, (co)products, $R \in \mathcal{L}$, $(\mathcal{A,L})$ is also a duality pair, and $\mathcal{A}$ is the right half of a hereditary complete cotorsion pair. We shall refer to such duality pairs as \emph{product closed bicomplete}. As an example, we have the duality pair mentioned in the previous paragraph. On the other hand, our relative version of Gorenstein flat modules is defined as follows: cycles of exact chain complexes with components in $\mathcal{L}$ which remain exact after tensoring by objects in $$\mathcal{A} \cap {}^\perp\mathcal{A} = \mathcal{A} \cap \Big( \bigcap_{i \in \mathbb{Z}_{> 0}} {\rm Ker}({\rm Ext}^i_{R^{\rm o}}(-,\mathcal{A})) \Big).$$ Under these assumptions, we prove that these relative Gorenstein flat modules are closed under extensions, and that the corresponding Gorenstein flat dimension is well behaved in the sense that it recovers many of the properties and characterizations of its (absolute) Gorenstein flat counterpart (for instance, it can be described in terms of torsion functors). The previous will be a consequence of a Pontryagin duality relation that we show between these relative Gorenstein flat modules and certain Gorenstein injective modules relative to $\mathcal{A}$ (in the sense of \cite{BMS}). The condition $\Tor^{R}_i(A,L) = 0$ for every $A \in \mathcal{A}$, $L \in \mathcal{L}$ and $i \in \mathbb{Z}_{> 0}$, will not be required in our approach. We also define Gorenstein flat complexes relative to product closed bicomplete duality pairs, and study homotopical aspects of these complexes and their module counterparts. Specifically, we obtain several hereditary and cofibrantly generated abelian model structures and compute their homotopy categories: one structure on the category of modules over $R$, and eight on the category of chain complexes. At the level complexes, we find three recollements between the homotopy categories of these model structures, along with several derived adjunctions connecting these recollements.

%%%%%%%%%%%%%%%%%%%%%%%%%%%%%%%%%%%%%%%%%%%%%%%
%%%%%%%%%%%%%%%%%%%%%%%%%%%%%%%%%%%%%%%%%%%%%%%

\subsection*{Organization}

In Section \ref{sec:prelim} we recall the necessary background on relative homological algebra, such as the definitions and notations for approximations, homological dimensions, cotorsion pairs, duality pairs, purity and model structures. 

Section \ref{sec:relative_G-flat} is devoted to investigate the class of cycles of exact chain complexes with components in a class $\mathcal{L}$ of left $R$-modules, which remain exact after tensoring by objects in a class $\mathcal{A}$ of right $R$-modules which are also left orthogonal to $\mathcal{A}$. The latter class will be denoted by $\nu$, and the cycles of such complexes will be called Gorenstein $(\mathcal{L},\nu)$-flat $R$-modules, which form a class that we denote by $\mathcal{GF}_{(\mathcal{L},\nu)}$. Many interesting properties are obtained in the case where $(\mathcal{L,A})$ is a product closed bicomplete duality pair. The first important result about $\mathcal{GF}_{(\mathcal{L},\nu)}$ is a Pontryagin duality relation with the class $\mathcal{GI}_{(\nu,\mathcal{A})}$ of $(\nu,\mathcal{A})$-Gorenstein injective right $R$-modules (defined as cycles of exact chain complexes with components in $\mathcal{A}$ which remain exact after applying the functor $\Hom_{R^{\rm o}}(\nu,-)$). Specifically, we show in Theorem \ref{dualidad} that $M \in \mathcal{GF}_{(\mathcal{L},\nu)}$ if, and only if, its Pontryagin dual $M^+ := \Hom_{\mathbb{Z}}(M,\mathbb{Q / Z})$ belongs to $\mathcal{GI}_{(\nu,\mathcal{A})}$. Several remarkable outcomes of this result are comprised in Corollary \ref{resolvente}, where we show that $\GF_{(\Le, \nu)}$ is a Kaplansky class which is also resolving, closed under direct summands and direct limits, and that $(\GF_{(\Le, \nu)},\GI_{(\nu, \A)})$ is a perfect duality pair. In particular, $\GF_{(\Le, \nu)}$ is closed under extensions, and every left $R$-module has a Gorenstein $(\mathcal{L},\nu)$-flat cover. 

In Section \ref{sec:relative_G-flat_dimensions} we define and study homological dimensions relative to $\GF_{(\Le, \nu)}$. Another consequence of Theorem \ref{dualidad}, in the case where $(\mathcal{L,A})$ is a product closed bicomplete duality pair, is a duality relation between the Gorenstein $(\mathcal{L},\nu)$-flat and the $(\nu,\mathcal{A})$-Gorenstein injective dimensions, namely, that the equality $\Gfd_{(\mathcal{L},\nu)}(M) = \Gid_{(\nu,\mathcal{A})}(M^+)$ holds for every left $R$-module $M$, as proved in Proposition \ref{iguales}. On the other hand, we characterize in Theorem \ref{theo:finiteness_GFdim} the Gorenstein $(\mathcal{L},\nu)$-flat dimension of $M$ in terms of the vanishing of the torsion functors $\Tor^R_i(-,M)$ at the class $\nu$. Moreover, we also study global relative Gorenstein flat dimensions of the ground ring $R$. Specifically, we introduce the left Gorenstein $(\mathcal{L},\nu)$-weak global dimension and the left Gorenstein $(\mathcal{L},\nu)$-weak finitistic dimension of $R$, denoted by ${\rm l.Gwgdim}_{(\mathcal{L},\nu)}(R)$ and $l.\mathrm{GF}_{(\mathcal{L},\nu)}\mbox{-}\mathrm{findim}(R)$, respectively. In the case where $(\mathcal{L,A})$ is a product closed bicomplete duality pair, we characterize in Proposition \ref{prop:finiteness_global_dim} the finiteness of ${\rm l.Gwgdim}_{(\mathcal{L},\nu)}(R)$ in terms of the flat dimension of modules in $\nu$. Furthermore, under the same conditions we show in Proposition \ref{prop:finiteness_relative_Gflat_findim} that $l.\mathrm{GF}_{(\mathcal{L},\nu)}\mbox{-}\mathrm{findim}(R)$ coincides with the global $\mathcal{L}$-resolution dimension of $R$.

Homotopical aspects of Gorenstein $(\mathcal{L},\nu)$-flat modules are studied in Section \ref{sec:relative_G-flat_models}. For every $n \in \mathbb{Z}_{\geq 0}$, if we let $\mathcal{GF}_{(\mathcal{L},\nu)}{}^\wedge_n$ denote the class of $R$-modules $M$ with $\Gfd_{(\mathcal{L},\nu)}(M) \leq n$, we show in Theorem \ref{theo:GF_model_structure} that $\mathcal{GF}_{(\mathcal{L},\nu)}{}^\wedge_n$  is the class of cofibrant objects of a hereditary and cofibrantly generated abelian model structure, provided that $(\mathcal{L,A})$ is a product closed bicomplete duality pair. 

We close our article in Section \ref{appx:complexes}, where we show that the notions and properties of Gorenstein $(\mathcal{L},\nu)$-flat and $(\nu,\mathcal{A})$-Gorenstein injective modules carry over to the category of chain complexes. We characterize a particular family of these complexes in terms of $\GF_{(\Le, \nu)}$ and $\GI_{(\nu, \A)}$, so that the duality relation in the context of modules in inherited by the chain complex counterparts of $\GF_{(\Le, \nu)}$ and $\GI_{(\nu, \A)}$. As a consequence, the closure properties, the existence of relative approximations and the construction of homological dimensions relative to $\GF_{(\Le, \nu)}$ are also valid for relative Gorenstein flat complexes. The model structure from Theorem \ref{theo:GF_model_structure} has also a chain complex counterpart in Theorem \ref{theo:GF_model_structure_Ch}. If we have a product closed bicomplete duality pair $(\mathcal{L,A})$ of modules, and set the duality pair $(\mathcal{L}\text{-complexes},\mathcal{A}\text{-complexes})$ in this theorem, we obtain one hereditary and cofibrantly generated abelian model structure where the cofibrant objects are given by the complexes with components in $\mathcal{GF}_{(\mathcal{L},\nu)}$ (see Corollary \ref{coro:last_model}). Other seven model structures related to $\mathcal{GF}_{(\mathcal{L},\nu)}$ are constructed using several methods of Gillespie \cite{GillespieCh,GillespieHereditary,GillespieMock} (see Corollaries \ref{coro:derived}, \ref{coro:rel_derived}, \ref{coro:last_model} and \ref{coro:seriously_this_is_the_last}). Moreover, in Corollary \ref{coro:reco1} we construct three different recollements between the homotopy categories of these structures, which in turn are related via derived adjunctions found in Proposition \ref{prop:derived_ad}.

%%%%%%%%%%%%%%%%%%%%%%%%%%%%%%%%%%%%
%%%%%%%%%%%%%%%%%%%%%%%%%%%%%%%%%%%%
%%%%%%%%%%%%%%%%%%%%%%%%%%%%%%%%%%%%
%%%%%%%%%%%%%%%%%%%%%%%%%%%%%%%%%%%%

\section{Preliminaries}\label{sec:prelim}

In what follows, we shall work with categories of modules and chain complexes over an associative ring $R$ with identity. Most of the concepts recalled below are presented within the framework of Grothendieck categories. 

%%%%%%%%%%%%%%%%%%%%%%%%%%%%%%%%%%%%
%%%%%%%%%%%%%%%%%%%%%%%%%%%%%%%%%%%%

\subsection*{Notations}

We denote by $\mathsf{Mod}(R)$ and $\mathsf{Mod}(R^{\rm o})$ the categories of left and right $R$-modules. For simplicity, we shall refer to these two classes of modules as $R$-modules and $R^{\rm o}$-modules, respectively.

Given a Grothendieck category $\mathcal{G}$, by $\Ch(\mathcal{G})$ we shall denote the category of complexes of objects in $\mathcal{G}$. If $\mathcal{G} = \mathsf{Mod}(R)$ or $\mathcal{G} = \mathsf{Mod}(R^{\rm o})$, then the categories of complexes of $R$-modules and $R^{\rm o}$-modules will be denoted by $\Ch(R)$ and $\Ch(R^{\rm o})$, for simplicity.  Objects in $\Ch(\mathcal{G})$ are sequences 
\[
X_\bullet = \cdots \to X_{m+1} \xrightarrow{\partial^{X_\bullet}_{m+1}} X_{m} \xrightarrow{\partial^{X_\bullet}_m} X_{m-1} \to \cdots
\] 
such that $\partial^{X_\bullet}_m \circ \partial^{X_\bullet}_{m+1} = 0$ for every integer $m \in \mathbb{Z}$. The cycles of $X_\bullet$, denoted $Z_m(X_\bullet)$, are defined as the kernel of $\partial^{X_\bullet}_m$. 

If $\mathcal{X}$ is a class of objects in $\mathcal{G}$, denoted $\mathcal{X} \subseteq \mathcal{G}$, then $\Ch(\mathcal{X})$ denotes the class of chain complexes of objects in $\mathcal{X}$, that is, $X_\bullet \in \Ch(\mathcal{X})$ if $X_m \in \mathcal{X}$ for every $m \in \mathbb{Z}$. 

Among the most important subcategories of $\mathsf{Mod}(R)$, we mainly consider the classes of projective, injective and flat $R$-modules, which will be denoted by $\mathcal{P}(R)$, $\mathcal{I}(R)$ and $\mathcal{F}(R)$, respectively.  

Concerning functors defined on modules, we let 
\[
\Ext^i_R(-,-) \colon \mathsf{Mod}(R) \times \mathsf{Mod}(R) \longrightarrow \mathsf{Mod}(\mathbb{Z})
\] 
denote the right $i$-th derived functor of 
\[
\Hom_R(-,-) \colon \mathsf{Mod}(R) \times \mathsf{Mod}(R) \longrightarrow \mathsf{Mod}(\mathbb{Z}),
\] 
where $\mathsf{Mod}(\mathbb{Z})$ denotes the category of abelian groups. In the more general setting of a Grothendieck category $\mathcal{G}$ (with the possible absence of enough projective objects), extension functors $\Ext^i_{\mathcal{G}}(-,-)$ may be defined in the sense of Yoneda. 

If $M \in \mathsf{Mod}(R^{\rm o})$ and $N \in \mathsf{Mod}(R)$, then $M \otimes_R N$ denotes the tensor product of $M$ and $N$. Recall that the construction of this tensor product defines a bifunctor 
\[
- \otimes_R - \colon \mathsf{Mod}(R^{\rm o}) \times \mathsf{Mod}(R) \longrightarrow \mathsf{Mod}(\mathbb{Z}).
\] 
The left derived functors of $- \otimes_R -$ are denoted by 
\[
\Tor^R_i(-,-) \colon \mathsf{Mod}(R^{\rm o}) \times \mathsf{Mod}(R) \longrightarrow \mathsf{Mod}(\mathbb{Z}).
\]

%%%%%%%%%%%%%%%%%%%%%%%%%%%%%%%%%%%%
%%%%%%%%%%%%%%%%%%%%%%%%%%%%%%%%%%%%

\subsection*{Orthogonality}

In what follows, we let $\mathbb{Z}_{> 0}$ and $\mathbb{Z}_{\geq 0}$ denote the sets of positive and nonnegative integers, respectively. For a Grothendieck category $\mathcal{G}$, $\mathcal{X}, \mathcal{Y} \subseteq \mathcal{G}$ and $i \in \mathbb{Z}_{> 0}$, the notation 
\[
\Ext^i_{\mathcal{G}}(\mathcal{X,Y}) = 0
\] 
means that $\Ext^i_{\mathcal{G}}(X,Y) = 0$ for every $X \in \mathcal{X}$ and $Y \in \mathcal{Y}$. In the case where $\mathcal{Y}$ is the singleton $\{ Y \}$ for some $Y \in \mathcal{G}$, we simply write $\Ext^i_{\mathcal{G}}(\mathcal{X},Y) = 0$. The notation $\Ext^i_{\mathcal{G}}(X,\mathcal{Y}) = 0$ for $X \in \mathcal{G}$ has a similar meaning. Moreover, by 
\[
\Ext^{\geq 1}_{\mathcal{G}}(\mathcal{X,Y}) = 0
\] 
we shall mean that $\Ext^i_{\mathcal{G}}(\mathcal{X,Y}) = 0$ for every $i \in \mathbb{Z}_{> 0}$. One also has similar meanings for $\Ext^{\geq 1}_{\mathcal{G}}(\mathcal{X},N) = 0$, $\Ext^{\geq 1}_{\mathcal{G}}(N,\mathcal{Y}) = 0$ and $\Ext^{\geq 1}_{\mathcal{G}}(X,Y) = 0$. We can also replace $\Ext$ by $\Tor$ in order to obtain similar notations for $\Tor$-orthogonality. 

The right Ext-orthogonal complements of $\mathcal{X}$ will be denoted by
\begin{align*}
\mathcal{X}^{\perp_i} & = \{ M \in \mathcal{G} {\rm \ : \ } \Ext^i_{\mathcal{G}}(\mathcal{X},M) = 0 \} & \text{and} & & \mathcal{X}^{\perp} & = \bigcap_{i \in \mathbb{Z}_{> 0}} \mathcal{X}^{\perp_i}.
\end{align*}
The left orthogonal complements, on the other hand, are defined similarly. If $\mathcal{X} \subseteq \mathsf{Mod}(R^{\rm o})$, the right Tor-orthogonal complements of $\mathcal{X}$ will be denoted by
\begin{align*}
\mathcal{X}^{\top_i} & = \{ M \in \mathsf{Mod}(R) {\rm \ : \ } \Tor_i^R(\mathcal{X},M) = 0 \} & \text{and} & & \mathcal{X}^{\top} & = \bigcap_{i \in \mathbb{Z}_{> 0}} \mathcal{X}^{\top_i}.
\end{align*}
Left Tor-orthogonal complements ${}^{\top_i}\mathcal{Y}$ and ${}^\top\mathcal{Y}$ are defined for classes $\mathcal{Y} \subseteq \mathsf{Mod}(R)$ of $R$-modules in a similar fashion.

%%%%%%%%%%%%%%%%%%%%%%%%%%%%%%%%%%%%%
%%%%%%%%%%%%%%%%%%%%%%%%%%%%%%%%%%%%% 

\subsection*{Relative homological dimensions} 

There are homological dimensions defined in terms of extension functors. Let $M$ be an object in a Grothendieck category $\mathcal{G}$, and $\mathcal{X}, \mathcal{Y} \subseteq \mathcal{G}$. The \emph{injective dimensions of $M$\footnote{Where $\id_{\mathcal{X}}(M) = \infty$ if $ \{ m \in \mathbb{Z}_{\geq 0} \text{ : } \Ext^{\geq m+1}_{\mathcal{G}}(\mathcal{X},M) = 0  \} = \emptyset$.} and $\mathcal{Y}$ relative to $\mathcal{X}$} are defined by
\begin{align*}
\id_{\mathcal{X}}(M) & := \inf \{ m \in \mathbb{Z}_{\geq 0} \text{ : } \Ext^{\geq m+1}_{\mathcal{G}}(\mathcal{X},M) = 0  \}, \\ 
\id_{\mathcal{X}}(\mathcal{Y}) & := \sup \{ \id_{\mathcal{X}}(Y) \text{ : } Y \in \mathcal{Y} \}.
\end{align*}
%The \emph{injective dimensions of $M$ and $\mathcal{X}$ relative to $\mathcal{Y}$}, denoted by $\id_{\mathcal{Y}}(M)$ and $\id_{\mathcal{Y}}(\mathcal{X})$, are defined dually. 
In the case where $\mathcal{X} = \mathcal{G}$, we write 
\begin{align*}
\id_{\mathcal{G}}(M) & = \id(M) & & \text{and} & \id_{\mathcal{G}}(\mathcal{Y}) & = \id(\mathcal{Y})
\end{align*} 
for the (absolute) injective dimensions of $M$ and $\mathcal{Y}$. 

By an \emph{$\mathcal{X}$-resolution of $M$} we mean an exact complex 
\[
\cdots \to X_m \to X_{m-1} \to \cdots \to X_1 \to X_0 \twoheadrightarrow M
\]
with $X_k \in \mathcal{X}$ for every $k \in \mathbb{Z}_{\geq 0}$. For simplicity, any $k$-th syzygy in a $\mathcal{X}$-resolution of $M$ will be denoted by $\Omega^{\mathcal{X}}_k(M) = {\rm Ker}(X_{k-1} \to X_{k-2})$ for $k > 1$, and $\Omega^{\mathcal{X}}_1(M) = {\rm Ker}(X_0 \to M)$. If $X_k = 0$ for $k > m$, we say that the previous resolution has \emph{length} $m$. The \emph{resolution dimension relative to $\mathcal{X}$} (or the \emph{$\mathcal{X}$-resolution dimension}) of $M$\footnote{Where $\resdim_{\mathcal{X}}(M) = \infty$ if $\{ m \in \mathbb{Z}_{\geq 0} \ \mbox{ : } \ \text{there exists an $\mathcal{X}$-resolution of $C$ of length $m$} \} = \emptyset$.} is defined as the value
\[
\resdim_{\mathcal{X}}(M) := \min \{ m \in \mathbb{Z}_{\geq 0} \ \mbox{ : } \ \text{there exists an $\mathcal{X}$-resolution of $M$ of length $m$} \}.
\]
Moreover, if $\mathcal{Y} \subseteq \mathcal{G}$ then
\[
\resdim_{\mathcal{X}}(\mathcal{Y}) := \sup \{ \resdim_{\mathcal{X}}(Y) \ \mbox{ : } \ \text{$Y \in \mathcal{Y}$} \}
\]
defines the \emph{resolution dimension of $\mathcal{Y}$ relative to $\mathcal{X}$}. The classes of objects with bounded (by some $n \in \mathbb{Z}_{\geq 0}$) and finite $\mathcal{X}$-resolution dimensions will be denoted by
\begin{align*}
\mathcal{X}^\wedge_n & := \{ M \in \mathcal{G} \text{ : } \resdim_{\mathcal{X}}(M) \leq n \} & \text{and} & & \mathcal{X}^\wedge & := \bigcup_{n \in \mathbb{Z}_{\geq 0}} \mathcal{X}^\wedge_n.
\end{align*}
Dually, we can define \emph{$\mathcal{X}$-coresolutions} and the \emph{coresolution dimension of $M$ and $\mathcal{Y}$ relative to $\mathcal{X}$} (denoted $\coresdim_{\mathcal{X}}(M)$ and $\coresdim_{\mathcal{X}}(\mathcal{Y})$). We also have the dual notations $\mathcal{X}^\vee_n$ and $\mathcal{X}^\vee$ for the classes of objects with bounded and finite $\mathcal{X}$-coresolution dimension. 

The resolution and coresolution dimensions define functions
\begin{align*}
\resdim_{\mathcal{X}}(-) & \colon \mathcal{G} \to \mathbb{Z}_{\geq 0} \cup \{ \infty \} & & \text{and} & \coresdim_{\mathcal{X}}(-) & \colon \mathcal{G}\to \mathbb{Z}_{\geq 0} \cup \{ \infty \},
\end{align*}
for which one can consider the following notion of stability.

\begin{definition}\label{def:stable}
We say that $\resdim_{\mathcal{X}}(-)$ is \textbf{stable} if for every $M \in \mathcal{G}$ and $n \in \mathbb{Z}_{\geq 0}$, the following two assertions are equivalent:
\begin{enumerate}
\item[(a)] $\resdim_{\mathcal{X}}(M) \leq n$.

\item[(b)] Any $(n\mbox{-}1)$-th $\mathcal{X}$-syzygy of $M$ belongs to $\mathcal{X}$. 
\end{enumerate}
The stable description for $\coresdim_{\mathcal{X}}(-)$ is defined dually. 
\end{definition}

We can note the following property concerning stability.

\begin{proposition}\label{prop:stability_in_ses}
Let $\mathcal{X} \subseteq \mathcal{G}$ such that $\resdim_{\mathcal{X}}(-)$ is stable. For every short exact sequence $A \rightarrowtail X \twoheadrightarrow C$ with $X \in \mathcal{X}$ and $A, C \in \mathcal{X}^\wedge$ and every $n \in \mathbb{Z}_{> 0}$, one has that $\resdim_{\mathcal{X}}(C) = n$ if, and only if, $\resdim_{\mathcal{X}}(A) = n-1$.
\end{proposition}

Some of the sufficient conditions that a class $\mathcal{X} \subseteq \mathcal{G}$ needs to fulfill so that $\resdim_{\mathcal{X}}(-)$ is stable are comprised in the concept of left Frobenius pair. If one lets $\omega \subseteq \mathcal{X}$, one says that $(\mathcal{X},\omega)$ is a \emph{left Frobenius pair} if:
\begin{enumerate}
\item $\mathcal{X}$ is left thick, that is, it is closed under extensions, under taking kernels of epimorphisms between its objects (epikernels, for short) and direct summands. 

\item $\id_{\mathcal{X}}(\omega) = 0$.

\item $\omega$ is a \emph{relative cogenerator} in $\mathcal{X}$, that is, for every object $X \in \mathcal{X}$ there is an embedding into an object of $\omega$ with cokernel in $\mathcal{X}$.
\end{enumerate}
%The dual concept is called \emph{right Frobenius pair}.
The following result can be found in \cite[Prop. 2.14]{BMS}, and was originally proved in \cite[Prop. 3.3]{AB89}.

\begin{proposition}\label{prop:Frobenius_stable}
If $(\mathcal{X},\omega)$ is a left Frobenius pair, then $\resdim_{\mathcal{X}}(-)$ is stable. 
\end{proposition}

%%%%%%%%%%%%%%%%%%%%%%%%%%%%%%%%%%%%
%%%%%%%%%%%%%%%%%%%%%%%%%%%%%%%%%%%%

\subsection*{Approximations}

Given an object $M$ in a Grothendieck category $\mathcal{G}$, and $\mathcal{X} \subseteq \mathcal{G}$, recall that a morphism $\varphi \colon X \to M$ with $X \in \mathcal{X}$ is an \emph{$\mathcal{X}$-precover} (or \emph{right approximation by $\mathcal{X}$}) \emph{of $M$} if for every morphism $\varphi' \colon X' \to M$ with $X' \in \mathcal{X}$, there exists a morphism $h \colon X' \to X$ such that $\varphi' = \varphi \circ h$. If in addition, in the case where $X' = X$ and $\varphi' = \varphi$, the equality $\varphi = \varphi \circ h$ can only be completed by automorphisms $h$ of $X$, then one says that $\varphi$ is an \emph{$\mathcal{X}$-cover of $M$}. An $\mathcal{X}$-precover is \emph{special} if it is epic and its kernel belongs to $\mathcal{X}^{\perp_1}$. We shall say that $\mathcal{X}$ is \emph{precovering} if every object in $\mathcal{G}$ has an $\mathcal{X}$-precover. Similarly, one has the notions of (special) (\emph{pre})\emph{envelopes} and \emph{special precovering}, \emph{covering}, (\emph{special}) \emph{preenveloping} and \emph{enveloping} classes.

%%%%%%%%%%%%%%%%%%%%%%%%%%%%%%%%%%%%
%%%%%%%%%%%%%%%%%%%%%%%%%%%%%%%%%%%%

\subsection*{Cotorsion pairs}

Two classes $\mathcal{X}$ and $\mathcal{Y}$ of objects in a Grothendieck category $\mathcal{G}$ form a \emph{cotorsion pair} $(\mathcal{X,Y})$ if $\mathcal{X} = {}^{\perp_1}\mathcal{Y}$ and $\mathcal{Y} = \mathcal{X}^{\perp_1}$. We shall refer to the intersection $\mathcal{X} \cap \mathcal{Y}$ as the \emph{core} of $(\mathcal{X,Y})$. The cotorsion pair $(\mathcal{X,Y})$ is said to be \emph{complete} if $\mathcal{X}$ is special precovering and $\mathcal{Y}$ is special preenveloping. In the case where $\mathcal{G}$ has enough projective objects, the latter two conditions are equivalent. If $\mathcal{X}$ is covering and $\mathcal{Y}$ is enveloping, one says that $(\mathcal{X,Y})$ is a \emph{perfect} cotorsion pair. 

A class $\mathcal{X} \subseteq \mathcal{G}$ is \emph{resolving} if every projective object of $\mathcal{G}$ belongs to $\mathcal{X}$ and $\mathcal{X}$ is closed under extensions and epikernels (that is, given a short exact sequence $A \rightarrowtail B \twoheadrightarrow C$ with $C \in \mathcal{X}$, then $A \in \mathcal{X}$ if and only if $B \in \mathcal{X}$). \emph{Coresolving} classes are defined dually. A cotorsion pair $(\mathcal{X,Y})$ is \emph{hereditary} if $\mathcal{X}$ is resolving and $\mathcal{Y}$ is coresolving. The latter two conditions are equivalent if $\mathcal{G}$ has enough projective objects, and they are in turn equivalent to saying that $\Ext^{\geq 1}_{\mathcal{G}}(\mathcal{X,Y}) = 0$. Note that $\mathcal{X} = {}^\perp\mathcal{Y}$ and $\mathcal{Y} = \mathcal{X}^\perp$ if $(\mathcal{X,Y})$ is a hereditary cotorsion pair. 

One way to show that a cotorsion pair $(\mathcal{X,Y})$ in a Grothendieck category with enough projective objects is complete by means of finding a \emph{cogenerating set}, that is, a set $\mathcal{S} \subseteq \mathcal{X}$ such that $\mathcal{Y} = \mathcal{S}^{\perp_1}$. By a result of Eklof and Trlifaj \cite[Thm. 10]{EklofTrlifaj}, every cotorsion pair of $R$-modules cogenerated by a set is complete. This result is known to be true in Grothendieck categories with enough projective objects (see for instance Aldrich et al. \cite[Corolls. 2.11 and 2.12]{Aldrich}). 

In order to find cotorsion pairs cogenerated by a set or to construct approximations, a useful concept is that of a Kaplansky class. For instance, any Kaplansky class in a Grothendieck category closed under direct limits and direct products is preenveloping. Recall from Gillespie’s \cite[Def. 5.1]{GillespieDW} that a class $\mathcal{X}$ of objects in $\mathcal{G}$ is a \emph{Kaplansky class} if there is a regular cardinal $\kappa$ for which the following condition is satisfied: For any subobject $S \subseteq M$, with $M \in \mathcal{X}$ nonzero and $S$ $\kappa$-generated, there exists a $\kappa$-presentable object $N \neq 0$ such that $S \subseteq N \subseteq M$ and $N, M / N \in \mathcal{X}$. In the case $\mathcal{G} = \mathsf{Mod}(R)$, this can be restated as follows: $\mathcal{X}$ is a Kaplansky class if there exists a regular cardinal $\kappa$ such that for every $M \in \mathcal{X}$ and for any subset $S \subseteq M$ with ${\rm Card}(S) \leq \kappa$, there exists a submodule $N \subseteq M$ such that $S \subseteq N$, ${\rm Card}(N) \leq \kappa$, $N, M / N \in \mathcal{X}$. Kaplansly classes in $\Ch(R)$ can also be defined in a similar way. Under certain conditions, if we are given a cotorsion pair $(\mathcal{X,Y})$ where $\mathcal{X}$ is a Kaplansky class, it is possible to find a cogenerating set for $(\mathcal{X,Y})$. Moreover, $(\mathcal{X,Y})$ is complete or, in some cases, perfect. We summarize this fact in the following result.

\begin{proposition}\label{prop:cotorsion_Kaplansky}
Let $\mathcal{X}$ be a Kaplansky class in a Grothendieck category $\mathcal{G}$. If $\mathcal{X}$ is closed under extensions, direct limits and direct summands, and contains a generator for $\mathcal{G}$, then $(\mathcal{X},\mathcal{X}^{\perp_1})$ is a cotorsion pair cogenerated by a set. Moreover, if $\mathcal{G}$ has enough projective objects, then $(\mathcal{X},\mathcal{X}^{\perp_1})$ is perfect. 
\end{proposition}

\begin{proof}
Follows by \cite[Corolls. 2.11, 2.12 \& 2.13]{Aldrich} and \cite[Prop. 4.8]{GillespieKaplansky}.
\end{proof}

%%%%%%%%%%%%%%%%%%%%%%%%%%%%%%%%%%%%
%%%%%%%%%%%%%%%%%%%%%%%%%%%%%%%%%%%%

\subsection*{Purity}

A short exact sequence $\varepsilon \colon A \rightarrowtail B \twoheadrightarrow C$ of $R$-modules is called \emph{pure} if the induced sequence $\Hom_R(F,\varepsilon)$ of abelian groups is exact for every finitely presented $R$-module $F$. In this situation, one says that $A$ is a pure submodule of $B$, and that $C$ is a pure quotient of $B$. A class $\mathcal{X}$ of $R$-modules is \emph{closed under pure submodules} (resp., \emph{quotients}) if for every $B \in \mathcal{X}$ and every pure exact sequence $\varepsilon$ as before, one has that $A \in \mathcal{X}$ (resp., $C \in \mathcal{X}$). All of these concepts have their counterpart in $\Ch(R)$, and are defined in a similar way. 

One important fact about purity is that it provides a quick way to determine that a class of modules or complexes is a Kaplansky class, by only checking closure under pure subobjects and pure quotients. More specifically, we have the following result, which was first proved by Holm and J{\o}rgensen in \cite[Prop. 3.2]{HJ08} for the module setting, and later generalized by Estrada and Gillespie in \cite[Lem. 3.3 \& proof of Prop. 3.4]{EstradaGillespie19} for chain complexes.

\begin{proposition}\label{prop:purity_and_Kaplansky}
Let $\mathcal{X}$ be a class of $R$-modules (resp., chain complexes in $\Ch(R)$). If $\mathcal{X}$ is closed under pure submodules (resp., pure subcomplexes) and pure quotients, then $\mathcal{X}$ is a Kaplansky class.
\end{proposition}

%%%%%%%%%%%%%%%%%%%%%%%%%%%%%%%%%%%%
%%%%%%%%%%%%%%%%%%%%%%%%%%%%%%%%%%%%

\subsection*{Gorenstein injective modules relative to admissible pairs}

Given two classes $\mathcal{X}$ and $\mathcal{Y}$ of objects in a Grothendieck category $\mathcal{G}$, an object $C \in \mathcal{G}$ is \emph{$(\mathcal{X,Y})$-Gorenstein injective} if there exists an exact and $\Hom_{\mathcal{G}}(\mathcal{X},-)$-acyclic complex 
\[
Y_\bullet = \cdots \to Y_m \xrightarrow{\partial^{Y_\bullet}_m} Y_{m-1} \to \cdots
\] 
in $\Ch(\mathcal{Y})$ such that $M \simeq Z_0(Y_\bullet)$. By \emph{$\Hom_{\mathcal{G}}(\mathcal{X},-)$-acyclic} we mean that 
\[
\Hom_{\mathcal{G}}(X,Y_\bullet) := \cdots \to \Hom_{\mathcal{G}}(X,Y_m) \xrightarrow{\Hom_{\mathcal{G}}(X,\partial^{Y_\bullet}_{m})} \Hom_{\mathcal{G}}(X,Y_{m-1}) \to \cdots
\] 
is an exact complex of abelian groups for every $X \in \mathcal{X}$. 

The class of $(\mathcal{X,Y})$-Gorenstein injective objects is denoted by $\mathcal{GI}_{(\mathcal{X,Y})}$. In order to have nice homological properties for $\mathcal{GI}_{(\mathcal{X,Y})}$ (for example, to be a coresolving class closed under direct summands) one needs a minimal set of conditions for $\mathcal{X}$ and $\mathcal{Y}$. These conditions are comprised in the concept of \emph{GI-admissible pair}, that is, pairs $(\mathcal{X,Y})$ of classes of objects in $\mathcal{G}$ such that:
\begin{enumerate}
\item $\Ext^{\geq 1}_{\mathcal{G}}(\mathcal{X,Y}) = 0$.

\item $\mathcal{X}$ and $\mathcal{Y}$ are closed under finite coproducts.

\item $\mathcal{Y}$ is closed under extensions.

\item $\mathcal{X} \cap \mathcal{Y}$ is a relative generator in $\mathcal{Y}$. 

\item $\mathcal{Y}$ is a relative cogenerator in $\mathcal{G}$.
\end{enumerate}

%%%%%%%%%%%%%%%%%%%%%%%%%%%%%%%%%%%%%
%%%%%%%%%%%%%%%%%%%%%%%%%%%%%%%%%%%%% 

\subsection*{Model structures on abelian categories} 

Following \cite{HoveyModel}, an \emph{abelian model structure} on a Grothendieck category $\mathcal{G}$ is formed by three classes of morphisms, called fibrations, cofibrations and weak equivalences, such that: a morphism $f$ is a cofibration (resp., fibration) if, and only if, it is a monomorphism (resp., an epimorphism) such that ${\rm CoKer}(f)$ is a cofibrant object (resp., ${\rm Ker}(f)$ is a fibrant object). Recall that an object $C$ is (co)fibrant if $C \to 0$ (resp., $0 \to C$) is a (co)fibration. 

There is an appealing one-to-one correspondence between abelian model structures and cotorsion pairs, proved by Hovey in \cite[Thm. 2.2]{HoveyModel}. We shall only be interested in the implication that asserts that if $\mathcal{Q}$, $\mathcal{R}$ and $\mathcal{W}$ are three classes of objects in $\mathcal{G}$ such that $\mathcal{W}$ is \emph{thick} (that is, left thick and right thick simultaneosly), and $(\mathcal{Q} \cap \mathcal{W},\mathcal{R})$ and $(\mathcal{Q},\mathcal{R} \cap \mathcal{W})$ are complete cotorsion pairs, then there exists a unique abelian model structure on $\mathcal{G}$ such that $\mathcal{Q}$, $\mathcal{R}$ and $\mathcal{W}$ are the classes of cofibrant, fibrant and trivial objects, respectively. So we shall denote abelian model structures as the triples 
\[
\mathcal{M} = (\mathcal{Q,W,R})
\]
formed by their classes of cofibrant, trivial and fibrant objects. Note that the cotorsion pairs $(\mathcal{Q},\mathcal{R} \cap \mathcal{W})$ and $(\mathcal{Q} \cap \mathcal{W}, \mathcal{R})$ have the same core, and it will be called the \emph{core} of the model structure $\mathcal{M}$. In \cite[Thm. 1.2]{GillespieHowTo}, Gillespie proved the following method to construct abelian model structures from two hereditary complete cotorsion pairs.

\begin{proposition}\label{prop:HoveyTriple}
If $(\mathcal{Q},\mathcal{R}')$ and $(\mathcal{Q}',\mathcal{R})$ are complete and hereditary cotorsion pairs in $\mathcal{G}$ such that $\mathcal{Q}' \subseteq \mathcal{Q}$, $\mathcal{R}' \subseteq \mathcal{R}$ and $\mathcal{Q} \cap \mathcal{R}' = \mathcal{Q}' \cap \mathcal{R}$ (in other words, they are \textbf{compatible}), then there exists a thick class $\mathcal{W}$ and a unique abelian model structure on $\mathcal{G}$ such that $\mathcal{Q}$, $\mathcal{R}$ and $\mathcal{W}$ are the classes of cofibrant, fibrant and trivial objects, respectively. Moreover, $\mathcal{W}$ has the following two descriptions:
\begin{align*}
\mathcal{W} & = \{ W \in \mathcal{G} {\rm \ : \ } \text{there is a short exact sequence } R \rightarrowtail Q \twoheadrightarrow W \text{ with } R \in \mathcal{R}' \text{ and } Q \in \mathcal{Q}' \} \\
& = \{ W \in \mathcal{G} {\rm \ : \ } \text{there is a short exact sequence } W \rightarrowtail R' \twoheadrightarrow Q' \text{ with } R \in \mathcal{R}' \text{ and } Q \in \mathcal{Q}' \}.
\end{align*}
\end{proposition}

%%%%%%%%%%%%%%%%%%%%%%%%%%%%%%%%%%%%
%%%%%%%%%%%%%%%%%%%%%%%%%%%%%%%%%%%%

\subsection*{Duality pairs}

Consider the Pontryagin duality functor
\[
(-)^+ := \Hom_{\mathbb{Z}}(-,\mathbb{Q / Z}) \colon \mathsf{Mod}(R) \to \mathsf{Mod}(R^{\rm o}),
\]
which is exact since $\mathbb{Q / Z}$ is an injective $\mathbb{Z}$-module. The notion of duality pair was introduced by Holm and J{\o}rgensen in \cite{HJ09}, in the following way: two classes $\mathcal{L} \subseteq \mathsf{Mod}(R)$ and $\mathcal{A} \subseteq \mathsf{Mod}(R^{\rm o})$ form a \emph{duality pair} $(\mathcal{L,A})$ if:
\begin{enumerate}
\item $L \in \mathcal{L}$ if, and only if, $L^+ \in \mathcal{A}$. 

\item $\mathcal{A}$ is closed under direct summands and finite direct sums. 
\end{enumerate}
A duality pair $(\mathcal{L,A})$ is called:
\begin{itemize}
\item \emph{(co)product-closed} if $\mathcal{L}$ is closed under (co)products. 

\item \emph{perfect} if it is coproduct closed, $\mathcal{L}$ is closed under extensions and contains $R$ (regarded as an $R$-module). 
\end{itemize}

\begin{remark}\label{rmk:duality_pair}
One can also consider the Pontryagin duality functor 
\[
(-)^+ := \Hom_{\mathbb{Z}}(-,\mathbb{Q / Z}) \colon \mathsf{Mod}(R^{\rm o}) \to \mathsf{Mod}(R),
\]
mapping $R^{\rm o}$-modules to $R$-modules, and thus one gets a similar notion of duality pair $(\mathcal{L,A})$ in the case where $\mathcal{L} \subseteq \mathsf{Mod}(R^{\rm o})$ and $\mathcal{A} \subseteq \mathsf{Mod}(R)$.
\end{remark}

The previous two types of duality pairs are important since they induce approximations by the classes $\mathcal{L}$ and $\mathcal{A}$. Indeed, it is proved in \cite[Thm. 3.1]{HJ09} by Holm and J{\o}rgensen, and in \cite[Prop. 2.3]{GillespieDuality} by Gillespie that:
\begin{itemize}
\item $\mathcal{L}$ is closed under pure submodules and pure quotients. In particular, $\mathcal{L}$ is a Kaplansky class by Proposition \ref{prop:purity_and_Kaplansky}.

\item If $(\mathcal{L,A})$ is coproduct closed, then $\mathcal{L}$ is covering.

\item If $(\mathcal{L,A})$ is product closed, then $\mathcal{L}$ is preenveloping.

\item If $(\mathcal{L,A})$ is perfect, then $(\mathcal{L},\mathcal{L}^{\perp_1})$ is a perfect cotorsion pair and $\mathcal{L}$ is closed under direct limits.
\end{itemize}
In \cite[Prop. 2.3]{GillespieDuality} it is proved that if $(\mathcal{L,A})$ is a perfect duality pair, then $\mathcal{P}(R) \subseteq \mathcal{L}$ and $\mathcal{I}(R^{\rm o}) \subseteq \mathcal{A}$. Moreover, $\mathcal{L}$ is closed under direct limits and so $\mathcal{F}(R) \subseteq \mathcal{L}$ by Lazard-Govorov's Theorem (see for instance Lam's \cite[Thm. II.4.34]{Lam99}).

Duality pairs were also studied by Bravo, Gillespie and Hovey in \cite{BGH14}. In \cite[Appx. A]{BGH14} and \cite[Def. 2.4]{GillespieDuality}, a duality pair $(\mathcal{L,A})$ is called \emph{symmetric} if $(\mathcal{A,L})$ is also a duality pair (in the sense of Remark \ref{rmk:duality_pair}). A \emph{complete} duality pair is a symmetric duality pair $(\mathcal{L,A})$ such that $(\mathcal{L,A})$ is a perfect duality pair.

%%%%%%%%%%%%%%%%%%%%%%%%%%%%%%%%%%%%
%%%%%%%%%%%%%%%%%%%%%%%%%%%%%%%%%%%%
%%%%%%%%%%%%%%%%%%%%%%%%%%%%%%%%%%%%
%%%%%%%%%%%%%%%%%%%%%%%%%%%%%%%%%%%%

\section{Relative Gorenstein flat modules}\label{sec:relative_G-flat} 

In what follows, we shall consider classes 
\begin{align*}
\mathcal{L} & \subseteq \mathsf{Mod}(R) & & \text{and} & \mathcal{A} & \subseteq \mathsf{Mod}(R^{\rm o}).
\end{align*} 
The following definition of Gorenstein flat $R$-modules relative to $(\mathcal{L,A})$ is due to Wang, Yang and Zhu  \cite[Def. 2.1]{WangYangZhu19}.

\begin{definition}\label{def:relativeGF}
An $R$-module $M$ is \textbf{Gorenstein $\bm{(\mathcal{L,A})}$-flat} if there exists an exact and $(\mathcal{A} \otimes _R -)$-acyclic complex $L_\bullet \in \Ch(\mathcal{L})$ such that $M \simeq Z_0(L_\bullet)$. By \textbf{$\bm{(\mathcal{A} \otimes _R -)}$-acyclic} we mean that $A \otimes _R L_\bullet$ is an exact complex of abelian groups for every $A \in \mathcal{A}$. The class of Gorenstein $(\mathcal{L,A})$-flat $R$-modules will be denoted by $\mathcal{GF}_{(\mathcal{L,A})}$. 

We shall denote the orthogonal complement $\mathcal{GF}_{(\mathcal{L,A})}^{\perp_1}$ by $\mathcal{GC}_{(\mathcal{L,A})}$, and its elements will be referred to as \textbf{Gorenstein $\bm{(\mathcal{L,A})}$-cotorsion modules}. 
\end{definition}

\begin{remark}
{} \
\begin{enumerate}
\item The containment $\mathcal{L} \subseteq \mathcal{GF} _{(\mathcal{L,A})}$ is clear. 

\item The Gorenstein $(\mathcal{F}(R),\mathcal{I}(R^{\rm o}))$-flat $R$-modules are precisely the Gorenstein flat $R$-modules. 

\item If $\mathcal{F}(R) \subseteq \mathcal{L}$ and $\mathcal{A} \subseteq \mathcal{I}(R^{\rm o})$, then every Gorenstein flat $R$-module is Gorenstein $(\mathcal{L,A})$-flat. 

\item Every cycle of an exact and $(\mathcal{A} \otimes_R -)$-acyclic complex in $\Ch(\mathcal{L})$ is Gorenstein $(\mathcal{L,A})$-flat. 
\end{enumerate}
\end{remark}

In the main results of the present article, we shall consider certain duality pairs $(\mathcal{L,A})$ with some additional conditions, and Gorenstein flat modules relative to the pair $(\mathcal{L},{}^\perp\mathcal{A} \cap \mathcal{A})$. So in what follows, we set the notation
\[
\nu := {}^{\perp}\mathcal{A} \cap \mathcal{A}.
\]
Let us show some characterizations and general properties of the class $\mathcal{GF}_{(\mathcal{L},\nu)}$.

\begin{lemma}\label{lem:ort}
If $\mathcal{L}^+ \subseteq \mathcal{A}$, then $\Tor_{\geq 1}^{R}({}^{\perp}\mathcal{A},\mathcal{L}) = 0$.
\end{lemma}

\begin{proof}
For every $M \in \mathsf{Mod}(R^{\rm o})$ and $L \in \mathsf{Mod}(R)$, it is well known the existence of a natural isomorphism $\Tor^R_{i}(M,L)^{+} \cong \Ext^{i}_{R^{\rm o}}(M, L^{+})$ for every $i \geq 1$ (see for instance G\"{o}bel and Trlifaj's \cite[Lem. 1.2.11 (b)]{GT}, or \cite[Thm. 3.2.1]{EJ00} for a version with $\Ext^i_R$). Now in the case where $M \in {}^\perp \A$ and $L \in \Le$, we have that $L^+ \in \mathcal{A}$, and so $\Ext^{i}_{R^{\rm o}}(M, L^{+}) = 0$. It follows that $\Tor^R_{i}(M,L)^{++} = 0$. Since $\Tor^R_i(M,L)$ is a pure abelian subgroup of $\Tor^R_i(M,L)^{++}$, we can conclude $\Tor^R_i(M, L) = 0$. 
\end{proof}

The next result follows the spirit of Bennis' \cite[Lem. 2.4]{GFclosed}.

\begin{proposition}\label{Bennis} 
Consider the following assertions for $M \in \mathsf{Mod}(R)$:
\begin{enumerate}
\item[(a)] $M \in \mathcal{GF}_{(\mathcal{L},\nu)}$.

\item[(b)] $M \in \nu^\top$ and admits a $(\nu \otimes _R -)$-acyclic $\mathcal{L}$-coresolution. 

\item[(c)] There is an exact sequence $M \rightarrowtail L \twoheadrightarrow G$ with $L \in \mathcal{L}$ and $G \in \GF_{(\mathcal{L},\nu)}$.
\end{enumerate}
Then, the following assertions hold:
\begin{enumerate}
\item (a) $\Rightarrow$ (c).

\item If $\mathcal{L}^+ \subseteq \mathcal{A}$, then (a) $\Rightarrow$ (b) $\Leftarrow$ (c). If in addition $\mathcal{L}$ is a relative generator in $\mathsf{Mod}(R)$, then (a) $\Leftarrow$ (b) $\Rightarrow$ (c).\footnote{Note that if $\mathcal{P}(R) \subseteq \mathcal{L}$, then $\mathcal{L}$ is trivially a relative generator in $\mathsf{Mod}(R)$. Moreover, the the implications (a) $\Leftarrow$ (b) $\Rightarrow$ (c) are also valid in the case where $\mathcal{P}(R) \subseteq \mathcal{L}$, and without assuming $\mathcal{L}^+ \subseteq \mathcal{A}$.}
\end{enumerate}
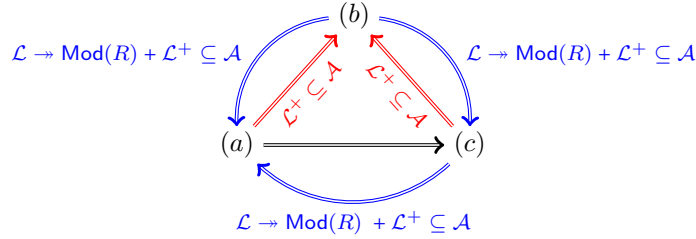
\begin{figure}[h!]
\begin{tikzpicture}[description/.style={fill=white,inner sep=2pt}]
\matrix (m) [matrix of math nodes, row sep=3.5em, column sep=2.5em, text height=1.25ex, text depth=0.25ex]
{ 
 & (b) & \\
(a) & & (c) \\
};
\path[->]
(m-2-1) edge [double] (m-2-3)
(m-2-3) edge [double, bend left = 1.5cm, blue] node[below] {\footnotesize$\mathcal{L} \twoheadrightarrow \mathsf{Mod}(R)$ \text{ + }$\mathcal{L}^+ \subseteq \mathcal{A}$} (m-2-1)
(m-2-1) edge [double, red] node[below,sloped] {\footnotesize$\mathcal{L}^+ \subseteq \mathcal{A}$} (m-1-2)
(m-2-3) edge [double, red] node[below,sloped] {\footnotesize$\mathcal{L}^+ \subseteq \mathcal{A}$} (m-1-2)
(m-1-2) edge [double, bend right = 1.5cm, blue] node[left] {\footnotesize$\mathcal{L} \twoheadrightarrow \mathsf{Mod}(R) \text{ + } \mathcal{L}^+ \subseteq \mathcal{A} \mbox{ \ }$} (m-2-1)
(m-1-2) edge [double, bend left = 1.5cm, blue] node[right] {\footnotesize$\mbox{ \ } \mathcal{L} \twoheadrightarrow \mathsf{Mod}(R) \text{ + } \mathcal{L}^+ \subseteq \mathcal{A} \mbox{ \ }$} (m-2-3)
;
\end{tikzpicture}
\caption{Implications in Proposition \ref{Bennis}.}
\end{figure}
\end{proposition}

\begin{proof}
The implication (a) $\Rightarrow$ (c) follows directly from the definition of $\mathcal{GF}_{(\mathcal{L},\nu)}$. Now assume $\mathcal{L}^+ \subseteq \mathcal{A}$ for the rest of the proof. 
\begin{itemize}
\item  (a) $\Rightarrow$ (b): We only show that $M \in \nu^\top$, as the rest of the assertion is part of the definition of Gorenstein $(\mathcal{L},\nu)$-flat $R$-modules. Since $M \in \GF _{(\mathcal{L},\nu)}$, we can consider an exact and $(\nu \otimes _R -)$-acyclic sequence $\Omega^{\mathcal{L}}_1(M) \rightarrowtail L_0 \twoheadrightarrow M$ with $L_0 \in \mathcal{L}$ and $\Omega^{\mathcal{L}}_1(M) \in \mathcal{GF}_{(\mathcal{L},\nu)}$. Then for every $A \in \nu$, we have the exact sequence 
\[
\Tor ^R _1 (A, L_0) \to \Tor ^R _1 (A,M) \to A \otimes _R \Omega^{\mathcal{L}}_1(M) \to A \otimes _R L_0,
\] 
where $\Tor^R_{\geq 1}(A, L_0) = 0$ by Lemma \ref{lem:ort}, and $A \otimes _R \Omega^{\mathcal{L}}_1(M) \to A \otimes _R L_0$ is a monomorphism. Hence, it follows that $\Tor^R_1(A,M) = 0$. Similarly, we obtain $\Tor^R_1(A,\Omega^{\mathcal{L}}_1(M)) = 0$, and thus one can show inductively that $\Tor^R_{\geq 1}(A,M) = 0$.

\item (c) $\Rightarrow$ (b): Let us assume that there is an exact sequence $M \rightarrowtail L \twoheadrightarrow G$ with $L \in \Le$ and $G \in \GF_{(\mathcal{L},\nu)}$. For $A \in \nu$ consider the exact sequence 
\[
\Tor^R_{i+1}(A,G) \to \Tor^R_i (A,M) \to \Tor^R_i(A,L).
\] 
From the implication (a) $\Rightarrow$ (b) we have that $\Tor^R_{i+1}(A, G) = 0$. Also, by Lemma \ref{lem:ort} we have $\Tor^R_i(A,L) = 0$. It follows that the sequence $M \rightarrowtail L \twoheadrightarrow G$ is $(\nu \otimes _R -)$-acyclic and that $M \in \nu^\top$. Now consider for $G$ a $(\nu \otimes _R -)$-acyclic $\mathcal{L}$-coresolution $G \rightarrowtail L^0 \to L^1 \to \cdots$. Splicing at $G$ this complex and the previous short exact sequence yields a $(\nu \otimes _R -)$-acyclic $\mathcal{L}$-coresolution $M \rightarrowtail L \to L^0 \to L^1 \to \cdots$.

\item Now assume that $\mathcal{L}$ is a relative generator in $\mathsf{Mod}(R)$. Note that (b) $\Rightarrow$ (c) will follow after showing (b) $\Rightarrow$ (a). So suppose (b) holds, that is, there exists a $(\nu \otimes _R -)$-acyclic $\mathcal{L}$-coresolution $M \rightarrowtail L^0 \to L^1 \to \cdots$. On the other hand, there exists an exact sequence $\Omega^{\mathcal{L}}_1(M) \rightarrowtail L_0 \twoheadrightarrow M$ with $L_0 \in \Le$. Applying $A \otimes _R -$ to this sequence, with $A \in \nu$, yields the exact sequence 
\[
\Tor^R_{i+1}(A,M) \to \Tor^R_i(A,\Omega^{\mathcal{L}}_1(M)) \to \Tor^R_i(A, L_0),
\] 
where $\Tor^R_i(A, L_0) = 0$ by  Lemma \ref{lem:ort}, and $\Tor^R_{i+1}(A,M) = 0$ from the assumption. Then, $\Tor^R_{\geq 1}(A,\Omega^{\mathcal{L}}_1(M)) = 0$. Furthermore, the exact sequence $\Omega^{\mathcal{L}}_1(M) \rightarrowtail L_0 \twoheadrightarrow M$ is also $(\nu \otimes _R -)$-acyclic since $M \in \nu^\top$. Hence, glueing at $M$ the previous sequence and the $\mathcal{L}$-coresolution of $M$ yields a $(\nu \otimes _R -)$-acyclic $\mathcal{L}$-coresolution $\Omega^{\mathcal{L}}_1(M) \rightarrowtail L_0 \to L ^0 \to L^1 \to \cdots$ with $\Omega^{\mathcal{L}}_1(M) \in \nu^\top$. We can repeat this procedure for $\Omega^{\mathcal{L}}_1(M)$, and for any $\mathcal{L}$-syzygy, in order to obtain a $(\nu \otimes _R -)$-acyclic $\mathcal{L}$-resolution of $M$. Therefore, $M \in \GF_{(\mathcal{L},\nu)}$.
\end{itemize}
\end{proof}

In order to have interesting homological and homotopical properties from relative Gorenstein flat modules, one needs $\mathcal{GF}_{(\mathcal{L}, \nu)}$ to be closed under extensions. We are not aware if this closure property holds in general. A similar problem has been tackled previously. For example, in \cite[Thm. 2.14]{EIP20} Estrada, Iacob and the second author showed that Gorenstein $(\mathcal{F}(R),\mathcal{B})$-flat $R$-modules are closed under extensions, provided that $\mathcal{B}$ is closed under products and contains an elementary cogenerator of its definable closure. Another approach is the one studied by Wang, Yang and Zhu in \cite[Coroll. 2.19]{WangYangZhu19}, where they show that if $(\Le,\A)$ is a complete duality pair with $\Le$ closed under epikernels and $\Tor^R_{\geq 1}(\mathcal{A,L}) = 0$, then $\mathcal{GF}_{(\mathcal{L,A})}$ is closed under extensions. This motivates the following concept.

\begin{definition}
We shall say that the pair $(\mathcal{L,A})$ is \textbf{Tor-orthogonal} if $\Tor^R_{\geq 1}(\mathcal{A,L}) = 0$.
\end{definition}

We realize that the Tor-orthogonality relation $\Tor^R_{\geq 1}(\mathcal{A,L}) = 0$ is very restrictive, as it seems that duality pairs satisfying it are scarce in the literature. Of course this is trivial for the classical duality pair $(\mathcal{F}(R),\mathcal{I}(R^{\rm o}))$, but in general duality pairs obtained from flat and injective modules relative to modules of finite type, for instance, are not $\Tor$-orthogonal. The following example shows this in a precise way.

\begin{example}\label{ex:FPinfty}
Let $\mathcal{FP}_{\infty}(R^{\rm o})$ denote the class of $R^{\rm o}$-modules \textbf{of type $\bm{\text{FP}_\infty}$}, defined as those $M \in \mathsf{Mod}(R^{\rm o})$ for which there is a finitely generated projective resolution. The classes of \textbf{level} $R$-modules and \textbf{absolutely clean} $R^{\rm o}$-modules are defined as the orthogonal complements
\begin{align*}
\mathcal{LV}(R) & := (\mathcal{FP}_\infty(R^{\rm o}))^{\top_1} & & \text{and} & \mathcal{AC}(R^{\rm o}) & = (\mathcal{FP}_\infty(R^{\rm o}))^{\perp_1},
\end{align*}
respectively. By \cite[Thms. 2.12 \& 2.14]{BGH14}, it is known that $(\mathcal{L}(R),\mathcal{AC}(R^{\rm o}))$ is a complete duality pair. Moreover, by \cite[Coroll. 4.2]{BP17} this duality pair is product closed and bicomplete (in the sense of Definition \ref{def:bicomplete} below). However, it is not true in general that $\Tor^R_{\geq 1}(\mathcal{AC}(R^{\rm o}),\mathcal{L}(R)) = 0$. 

Consider for instance the commutative quotient ring $R = k[x_1,x_2,\dots] / (x_i x_j)_{i,j \geq 1}$ where $k$ is a field. This is an example of a (non coherent) 2-coherent ring, where $\mathcal{FP}_\infty(R^{\rm o})$ coincides with the class of finitely generated projective $R^{\rm o}$-modules. This is turn implies that $\mathsf{Mod}(R) = \mathcal{LV}(R) = \mathcal{AC}(R^{\rm o})$ (see \cite[Prop. 2.5]{BGH14} and \cite[Ex. 1.4]{BP17} for details). One can see that there are modules $M$ and $N$ over this ring such that $\Tor^R_{\geq 1}(M,N) \neq 0$. Indeed, suppose that $\Tor^R_{\geq 1}(\mathsf{Mod}(R),N) = 0$ for every $N \in \mathsf{Mod}(R)$. Then, every $R$-module is flat, and so $\mathcal{LV}(R) \subseteq \mathcal{F}(R)$. This in turn implies that $R$ is a coherent ring by \cite[Coroll. 2.11]{BGH14}, getting thus a contradiction. 
\end{example}

\begin{example}
Some particular interesting families of Gorenstein $(\mathcal{L,A})$-flat $R$-modules are obtained without requiring that $(\mathcal{L,A})$ is a complete duality pair. This is the case of:
\begin{itemize} 
\item Gorenstein flat $R$-modules, which are Gorenstein $(\mathcal{F}(R),\mathcal{I}(R^{\rm o}))$-flat;

\item \v{S}aroch and \v{S}\v{t}ov\'{\i}\v{c}ek's \cite[\S \ 4]{SS}: projectively coresolved Gorenstein flat $R$-modules, which are Gorenstein $(\mathcal{P}(R),\mathcal{I}(R^{\rm o}))$-flat; and 

\item \cite[Defs. 2.1 \& 2.6]{EIP20}: Gorenstein $\mathcal{B}$-flat $R$-modules and projectively coresolved Gorenstein $\mathcal{B}$-flat $R$-modules, which are Gorenstein $(\mathcal{F}(R),\mathcal{B})$-flat and Gorenstein $(\mathcal{P}(R),\mathcal{B})$-flat, respectively, where $\mathcal{B}$ is a class of $R^{\rm o}$-modules closed under products and containing an elementary cogenerator of its definable closure.
\end{itemize}
For all of the previous choices of $(\mathcal{L,A})$, the class $\mathcal{GF}_{(\mathcal{L,A})}$ is closed under extensions and $(\mathcal{L,A})$ is Tor-orthogonal. 
\end{example}

In view of Example \ref{ex:FPinfty}, Lemma \ref{lem:ort} and Proposition \ref{Bennis}, it is worth studying Gorenstein flat $R$-modules relative to $(\mathcal{L},\nu)$. Indeed, one of the main results proved below is that $\mathcal{GF}_{(\mathcal{L},\nu)}$ is closed under extensions whenever $(\mathcal{L,A})$ is a complete duality pair satisfying certain approximation properties, but dropping the Tor-orthogonality condition $\Tor^R_{\geq 1}(\mathcal{A,L}) = 0$. This particular family of duality pairs is specified in the following definition, introduced in \cite[Def. 3.2]{WangDi}.

\begin{definition}\label{def:bicomplete}
We shall say that $(\mathcal{L,A})$ is a \textbf{bicomplete duality pair} if the following conditions are satisfied:
\begin{enumerate}
\item $(\mathcal{L,A})$ is a complete duality pair.

\item $({}^{\perp}\mathcal{A},\mathcal{A})$ is a hereditary complete cotorsion pair in $\mathsf{Mod}(R^{\rm o})$. 
\end{enumerate}
\end{definition}

\begin{remark}\label{rmk:hereditarysym}
From \cite[Prop. 2.3]{GillespieDuality} and Wang and Di \cite[Lem. 3.4]{WangDi}, we have that if $(\mathcal{L,A})$ is a bicomplete duality pair, the $\mathcal{L}$ is closed under epikernels, and hence $(\mathcal{L},\mathcal{L}^\perp)$ is a hereditary perfect cotorsion pair. These observations are also valid for duality pairs of classes of chain complexes (see Section \ref{appx:complexes}).
\end{remark}

\begin{example}\label{ex:weakn}
Besides the product closed bicomplete duality pair $(\mathcal{LV}(R),\mathcal{AC}(R^{\rm o}))$ mentioned earlier, we can find other examples of these duality pairs scattered in the literature, along with methods to induce them:
\begin{enumerate}
\item For each $n \in \mathbb{Z}_{\geq 0}$, consider the class $\mathcal{FP}_\infty(R^{\rm o}) \cap \mathcal{P}(R^{\rm o})^\wedge_n$ of $R^{\rm o}$-modules of type $\text{FP}_\infty$ with projective dimension at most $n$. The Tor and Ext orthogonal complements 
\begin{align*}
\hspace{1cm}\mathcal{WF}_n(R) & := (\mathcal{FP}_\infty(R^{\rm o}) \cap \mathcal{P}(R^{\rm o})^\wedge_n)^{\top_1}, \\ 
\mathcal{WI}_n(R^{\rm o}) & := (\mathcal{FP}_\infty(R^{\rm o}) \cap \mathcal{P}(R^{\rm o})^\wedge_n)^{\perp_1},
\end{align*}
were defined and studied by Arunachalam, Raja, Chelliah and Annadevasahaya Mani in \cite{weakn}, and are known as \textbf{weak $\bm{n}$-flat} $R$-modules and \textbf{weak $\bm{n}$-injective} $R^{\rm o}$-modules, respectively (see \cite[Def. 4.1]{weakn}). From \cite[Props. 4.3, 4.5, 4.7 \& Thm. 4.1]{weakn}, we have that these classes form a product closed bicomplete duality pair $(\mathcal{WF}_n(R),\mathcal{WI}_n(R^{\rm o}))$. 

\item The following notions and results generalize Example \ref{ex:FPinfty}, and are due to Amini, Amzil and Bennis (see \cite{AminiAmzilBennis21}). For each $n \in \mathbb{Z}_{\geq 0}$, we say that an $R^{\rm o}$-module $U$ is $\bm{n}$\textbf{-super finitely presented} if there exists a projective resolution 
\[
\cdots \to F_{n+1} \to F_n \to \cdots \to F_1 \to F_0 \twoheadrightarrow U
\] 
such that $F_k$ is finitely generated for every $k \geq n$ (see \cite[Def. 2.1]{AminiAmzilBennis21}). In particular, when $n = 0$, one obtains the class of $R^{\rm o}$-modules of type $\text{FP}_\infty$. Let $\mathsf{Pres}^\infty_n(R^{\rm o})$ denote the class of $n$-super finitely presented $R^{\rm o}$-modules. The orthogonal complements
\begin{align*}
\hspace{1cm}\mathcal{WF}^n(R) & := (\mathsf{Pres}^\infty_n(R^{\rm o}))^{\top_{n+1}}, \\ 
\mathcal{WI}^n(R^{\rm o}) & := (\mathsf{Pres}^\infty_n(R^{\rm o}))^{\perp_{n+1}},
\end{align*}
are known as the classes of $\bm{n}$\textbf{-weak flat} $R$-modules and $\bm{n}$\textbf{-weak injective} $R^{\rm o}$-modules, respectively (see \cite[Def. 2.2]{AminiAmzilBennis21}). It is clear that $\mathcal{WF}^n(R)$ is closed under extensions and that $R \in \mathcal{WF}^n(R)$. Moreover, from \cite[Thm. 2.12, Props. 2.9, 4.1 \& 4.2, Rmk. 4.14]{AminiAmzilBennis21} we have that $(\mathcal{WF}^n(R), \mathcal{WI}^n(R^{\rm o}))$ is a product closed bicomplete duality pair. 

\item By Holm and J{\o}rgensen \cite[Lem. 2.5]{HJ09}, we know that if $R$ is a commutative noetherian ring with a dualizing complex, then $(\mathcal{GF}(R),\mathcal{GI}(R^{\rm o}))$ is a product closed complete duality pair. Moreover, by \v{S}aroch and \v{S}\v{t}ov\'{\i}\v{c}ek \cite[Thm. 5.6]{SS} we have that $({}^\perp(\mathcal{GI}(R^{\rm o})),\mathcal{GI}(R^{\rm o}))$ is a hereditary complete cotorsion pair\footnote{This result is actually valid for arbitrary rings.}, and so the previous duality pair is bicomplete. 

\item One method to obtain new product closed bicomplete duality pairs from old ones is by means of resolution and coresolution dimensions. More precisely, if $(\mathcal{L,A})$ is a complete duality pair of modules where $(^\perp\mathcal{A},\mathcal{A})$ is a hereditary cotorsion pair cogenerated by a set, then for every $m \in \mathbb{Z}_{>0}$ we have from \cite[Prop. 3.11]{WangDi} that $(\mathcal{L}^\wedge_m,\mathcal{A}^\vee_m)$ is also a complete duality pair where $({}^\perp(\mathcal{A}^\vee_m), \mathcal{A}^\vee_m)$ is a hereditary cotorsion pair cogenerated by a set. Moreover, it is clear that if $\mathcal{L}$ is closed under products, then so is $\mathcal{L}^\wedge_m$. 

So for each $m \in \mathbb{Z}_{>0}$, the previous three examples give rise to two new product closed bicomplete duality pairs of modules over an arbitrary ring, namely, $((\mathcal{WF}_n(R))^\wedge_m,(\mathcal{WI}_n(R^{\rm o}))^\vee_m)$ and $((\mathcal{WF}^n(R))^\wedge_m, (\mathcal{WI}^n(R^{\rm o}))^\vee_m)$, and the product closed bicomplete duality pair $((\mathcal{GF}(R))^\wedge_m,(\mathcal{GI}(R^{\rm o}))^\vee_m)$ in the case where $R$ is a commutative noetherian ring with a dualizing complex.

\item Another method consists on intersecting duality pairs in the following way. Suppose $(\mathcal{L}_\lambda,\mathcal{A}_\lambda)$ is a family of product closed complete duality pairs of modules such that $({}^\perp\mathcal{A}_\lambda,\mathcal{A}_\lambda)$ is a hereditary cotorsion pair cogenerated by a set, with $\lambda$ running over an index set $\Lambda$. If we set 
\begin{align*}
\mathcal{L} & := \bigcap_{\lambda \in \Lambda} \mathcal{L}_\lambda & \text{and} & & \mathcal{A} & := \bigcap_{\lambda \in \Lambda} \mathcal{A}_\lambda,
\end{align*} 
then $(\mathcal{L,A})$ is a product closed complete duality pair where $(^\perp\mathcal{A},\mathcal{A})$ is also a hereditary cotorsion pair cogenerated by a set.

Indeed, it is straightforward to check that $(\mathcal{L,A})$ is a product closed complete duality pair. Concerning the statement about $(^\perp\mathcal{A},\mathcal{A})$, suppose $\mathcal{S}_\lambda$ is a cogenerating set for $({}^\perp\mathcal{A}_\lambda,\mathcal{A}_\lambda)$, that is, $\mathcal{A}_\lambda = (\mathcal{S}_\lambda)^{\perp_1}$. Then, one can note that 
\[
\bigcap_{\lambda \in \Lambda} \mathcal{A}_\lambda = \bigcap_{\lambda \in \Lambda} (\mathcal{S}_\lambda)^{\perp_1} = \Big(\bigcup_{\lambda \in \Lambda} \mathcal{S}_\lambda \Big)^{\perp_1},
\] 
and hence the set $\bigcup_{\lambda \in \Lambda} \mathcal{S}_\lambda$ cogenerates the (complete) cotorsion pair $({}^{\perp_1}\mathcal{A},\mathcal{A})$. Finally, if each $\mathcal{A}_\lambda$ is coresolving for every $\lambda \in \Lambda$, it is clear that the same property holds for $\mathcal{A}$.  
\end{enumerate}
\end{example}

The closure under extensions of $\mathcal{GF}_{(\mathcal{L},\nu)}$ will be possible thanks to the Pontryagin duality relation between $\mathcal{GF}_{(\mathcal{L},\nu)}$ and $\mathcal{GI}_{(\nu,\mathcal{A})}$ described in Theorem \ref{dualidad} below. We shall need the following result.

\begin{proposition} \label{criterio}
Let $(\mathcal{X,Y})$ be a GI-admissible pair in a Grothendieck category $\mathcal{G}$. If we are given a short exact sequence $A \rightarrowtail B \twoheadrightarrow C$ in $\mathcal{G}$ with $A \in (\mathcal{X} \cap \mathcal{Y})^{\perp_1}$ and $B, C \in \GI_{(\X, \Y)}$, then $A \in \GI_{(\mathcal{X,Y})}$.
\end{proposition}

\begin{proof}
Since $(\X, \Y)$ is a GI-admissible pair, we have by the dual of \cite[Coroll. 3.25]{BMS} that $\X \cap \Y$ is a relative generator in $\GI _{(\X, \Y)}$. So we can consider a short exact sequence $C' \rightarrowtail W \twoheadrightarrow C$ with $W \in \X \cap \Y$ and $C' \in \GI_{(\X, \Y)}$. On the other hand, the pullback of $B \twoheadrightarrow C \twoheadleftarrow W$ yields the following commutative exact diagram
\[
%\parbox{1.5in}{
\begin{tikzpicture}[description/.style={fill=white,inner sep=2pt}] 
\matrix (m) [ampersand replacement=\&, matrix of math nodes, row sep=2.5em, column sep=2.5em, text height=1.25ex, text depth=0.25ex] 
{ 
{} \& C' \& C' \\
A \& N \& W \\
A \& B \& C \\
}; 
\path[->] 
(m-2-2)-- node[pos=0.5] {\footnotesize$\mbox{\bf pb}$} (m-3-3) 
;
\path[>->]
(m-1-2) edge (m-2-2) (m-1-3) edge (m-2-3)
(m-2-1) edge (m-2-2) (m-3-1) edge (m-3-2)
;
\path[->>]
(m-2-2) edge (m-3-2) edge (m-2-3) (m-2-3) edge (m-3-3) (m-3-2) edge (m-3-3)
;
\path[-,font=\scriptsize]
(m-1-2) edge [double, thick, double distance=2pt] (m-1-3)
(m-2-1) edge [double, thick, double distance=2pt] (m-3-1)
;
\end{tikzpicture}. 
%} 
\]
By the dual of \cite[Coroll. 3.3]{BMS} we know that $\GI_{(\X, \Y)}$ is closed under extensions and direct summands, and thus we have $N \in \GI_{(\X, \Y)}$. We also know that $A \in (\mathcal{X} \cap \mathcal{Y})^{\perp_1}$, which in turn implies that $A$ is a direct summand of $N \in \GI_{(\X, \Y)}$, and hence $A \in \GI_{(\X, \Y)}$. 
\end{proof}

\begin{theorem}\label{dualidad}
Consider the following assertions for $M \in \mathsf{Mod}(R)$ and a pair $(\mathcal{L,A})$ such that $\mathcal{L}^+ \subseteq \A$:
\begin{enumerate}
\item[(a)] $M \in \GF_{(\mathcal{L},\nu)}$.

\item[(b)] $M^+ \in \GI_{(\nu, \A)}$.

\item[(c)] $M \in \nu^\top$ and admits a $\Hom_R(-,\mathcal{L})$-acyclic $\mathcal{L}$-coresolution. 
\end{enumerate}
The following statements hold true:
\begin{enumerate}
\item (a) $\Rightarrow$ (b). 

\item Suppose in addition that $\mathcal{A}^+ \subseteq \mathcal{L}$:
\begin{enumerate}
\item[(i)] If $\mathcal{L}$ is a relative generator in $\mathsf{Mod}(R)$, then (c) $\Rightarrow$ (a).\footnote{This implication also holds if we assume $\mathcal{P}(R) \subseteq \mathcal{L}$ instead.} 

\item[(ii)] If $\mathcal{L}$ is preenveloping and $(\nu,\mathcal{A})$ is a GI-admissible pair, then (b) $\Rightarrow$ (c).
\end{enumerate}
\end{enumerate}
In particular, if $(\mathcal{L,A})$ is a product closed bicomplete duality pair, then (a), (b) and (c) are equivalent. 
\begin{figure}[h!]
\begin{tikzpicture}[description/.style={fill=white,inner sep=2pt}]
\matrix (m) [matrix of math nodes, row sep=3.5em, column sep=2.5em, text height=1.25ex, text depth=0.25ex]
{ 
{} & (b) & {} \\
(a) & {} & (c) \\
};
\path[->]
(m-2-1) edge [double] (m-1-2)
(m-2-3) edge [double,blue] node[below,sloped] {\footnotesize$\mathcal{L} \twoheadrightarrow \mathsf{Mod}(R)$} (m-2-1)
(m-1-2) edge [double,red] node[right] {\footnotesize$\mbox{ \ } \mathcal{L} \text{ is preenveloping } + (\nu,\mathcal{A}) \text{ is GI-admissible}$} (m-2-3)
;
\end{tikzpicture}
\caption{At each implication, the containment $\mathcal{L}^+ \subseteq \mathcal{A}$ is assumed.}
\end{figure}
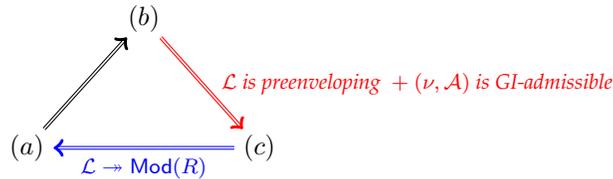
\end{theorem}

\begin{proof}
{} \
\begin{enumerate}
\item It is straightforward from the assumption and the fact that $(-)^+$ is an exact functor and the natural isomorphism $(A \otimes _R L_\bullet)^+ \cong \Hom_{R^{\rm o}} (A, L^+_\bullet)$. 

\item 
\begin{enumerate}
\item[(i)] Let $M \in \nu^\top$ with a $\Hom_R(-,\Le)$-acyclic $\mathcal{L}$-coresolution 
\[
\eta \colon M \rightarrowtail L^0 \to L^1 \to \cdots.
\] 
Now for any $A \in \nu$ we have a natural isomorphism $(A \otimes_R \eta)^+ \cong \Hom_R(\eta, A^+)$, where $\Hom_R(\eta, A^+)$ is exact since $\mathcal{A}^+ \subseteq \mathcal{L}$. It follows that $A \otimes_R \eta$ is exact. In other words, we have that $M$ admits a $(\nu \otimes_R -)$-acyclic $\mathcal{L}$-coresolution. The result then follows by Proposition \ref{Bennis} (2).
 
\item[(ii)] Suppose $M^+ \in \mathcal{GI}_{(\nu,\mathcal{A})}$. Firstly, since $\GI_{(\nu, \A)} \subseteq \nu^\perp$, we have that $\Tor^R_i(\nu,M)^+ \cong \Ext_{R^{\rm o}}^i(\nu,M^+) = 0$ for every $i \geq 1$, and so $M \in \nu^\top$. On the other hand, there exists a short exact sequence $N \rightarrowtail A \twoheadrightarrow M^+$ with $A \in \mathcal{A}$ and $N \in \mathcal{GI}_{(\nu,\mathcal{A})}$, which yields a short exact sequence $M^{++} \rightarrowtail A^+ \twoheadrightarrow N^+$ with $A^+ \in \mathcal{L}$. Now since $\mathcal{L}$ is preenveloping, there is an $\mathcal{L}$-preenvelope $\phi^1 \colon M \to L^0$ and an arrow $L^0 \to A^+$ making the following triangle commute:
\[
%\parbox{1.5in}{
\begin{tikzpicture}[description/.style={fill=white,inner sep=2pt}] 
\matrix (m) [ampersand replacement=\&, matrix of math nodes, row sep=2.5em, column sep=2.5em, text height=1.25ex, text depth=0.25ex] 
{ 
M \& M^{++} \& A^+ \\
L^0 \& {} \& {} \\
}; 
\path[dotted,->] 
(m-2-1) edge (m-1-3) 
;
\path[->]
(m-1-1) edge node[left] {\footnotesize$\phi^1$} (m-2-1)
;
\path[>->]
(m-1-1) edge (m-1-2) (m-1-2) edge (m-1-3)
;
\end{tikzpicture} 
%}
\]
It follows that $\phi^1$ is monic, and thus we obtain a short exact and $\Hom_R(-,\mathcal{L})$-acyclic sequence $M \rightarrowtail L^0 \twoheadrightarrow C_1$ with $L^0 \in  \Le$. Then, we have the exact sequence of the form $C_1^+ \rightarrowtail (L^0)^+ \twoheadrightarrow M^+$ where $(L^0)^+ \in \mathcal{A}$ and $M^+ \in \mathcal{GI}_{(\nu, \mathcal{A})}$. We shall prove that $\Ext_R^1(\nu, C_1^+) = 0$ in order to apply Proposition \ref{criterio}. So let us take $A' \in \nu$ and consider the following commutative diagram with exact rows:
\[
%\parbox{1.5in}{
\hspace{2cm}\begin{tikzpicture}[description/.style={fill=white,inner sep=2pt}] 
\matrix (m) [ampersand replacement=\&, matrix of math nodes, row sep=2.5em, column sep=1.25em, text height=1.25ex, text depth=0.25ex] 
{ 
\Hom_{R^{\rm o}}(A', (L^0)^+) \& \Hom_{R^{\rm o}}(A', M^+) \& \Ext^1_{R^{\rm o}}(A', C_1^+) \& \Ext^1_{R^{\rm o}}(A', (L^0)^+) \\
\Hom_R(L^0, (A')^+) \& \Hom_R(M, (A')^+) \& {} \& {} \\
}; 
\path[->]
(m-1-1) edge (m-1-2) (m-1-2) edge (m-1-3) (m-1-3) edge (m-1-4)
(m-1-1) edge node[left] {\footnotesize$\cong$} (m-2-1)
(m-1-2) edge node[left] {\footnotesize$\cong$} (m-2-2)
;
\path[->>]
(m-2-1) edge (m-2-2)
;
\end{tikzpicture} 
%}
\]
The arrow $\Hom_R(\phi^1, (A')^+) \colon \Hom_R(L^0, (A')^+) \to \Hom_R(M, (A')^+)$ is epic since $(A')^+ \in \mathcal{L}$, which in turn implies that $\Hom_{R^{\rm o}}(A', (L^0)^+) \to \Hom_{R^{\rm o}}(A', M^+)$ is epic. On the other hand, $(L^0)^+ \in \A$ together with $A' \in \nu$ implies that $\Ext_R^1(A', (L^0)^+) = 0$, and hence $\Ext_R^1(A', C_1^+) = 0$. It then follows by Proposition \ref{criterio} that $C^+_1 \in \mathcal{GI}_{(\nu, \A)}$. By repeating this process we can construct a $\Hom_R(-,\mathcal{L})$-acyclic $\mathcal{L}$-coresolution of $M$. 
\end{enumerate}
\end{enumerate}
For the last assertion, it suffices to note that if $({}^{\perp}\mathcal{A},\mathcal{A})$ is a hereditary complete cotorsion pair, then $(\nu,\mathcal{A})$ is a GI-admissible pair. 
\end{proof}

One important consequence from the previous theorem is that it implies that Gorenstein flat $R$-modules relative to a product closed bicomplete duality pair $(\mathcal{L,A})$ are closed under extensions. Before showing this, let us prove a relation between this closure property and a characterization of the $R$-modules in $\mathcal{GF}_{(\mathcal{L},\nu)}$, which is dual to Proposition \ref{criterio}. The following result is a generalization of \cite[Thm. 2.3]{GFclosed}. For pedagogical reasons, we provide a proof using arguments similar to those appearing in \cite{GFclosed}, in order to understand the minimal assumptions required for the relative case.

\begin{proposition}\label{prop:GF_closed_under_exts}
Consider the following assertions:
\begin{enumerate}
\item[(a)] Given a short exact sequence $G_1 \rightarrowtail G_0 \twoheadrightarrow M$ with $G_0, G_1 \in \mathcal{GF} _{(\mathcal{L},\nu)}$ and $M \in \nu^{\top_1}$, then $M \in \mathcal{GF} _{(\mathcal{L},\nu)}$.

\item[(b)] $\mathcal{GF} _{(\mathcal{L},\nu)}$ is closed under extensions.

\item[(c)] $\mathcal{GF} _{(\mathcal{L},\nu)}$ is preresolving. 
\end{enumerate}
If $\mathcal{L}$ is a relative generator in $\mathsf{Mod}(R)$\footnote{One may assume that $\mathcal{P}(R) \subseteq \mathcal{L}$ instead.}, then the following hold:
\begin{enumerate}
\item If $\mathcal{L}$ is closed under extensions with $\mathcal{L}^+ \subseteq \mathcal{A}$, then (a) $\Rightarrow$ (b) $\Rightarrow$ (c).

\item If $(\mathcal{L,A})$ is a duality pair such that $\mathcal{A}$ is closed under extensions and $\nu$ is a relative generator in $\mathcal{A}$, then (c) $\Rightarrow$ (a). 
\end{enumerate}
In particular, if $(\mathcal{L,A})$ is a product closed bicomplete duality pair, then (a), (b) and (c) are equivalent. 
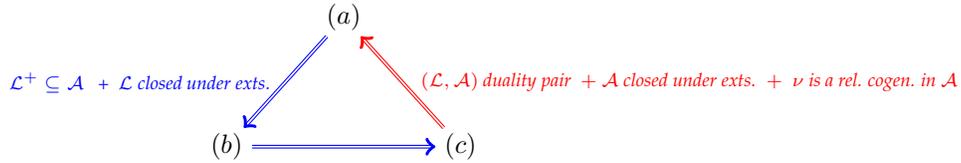
\begin{figure}[h!]
\begin{tikzpicture}[description/.style={fill=white,inner sep=2pt}]
\matrix (m) [matrix of math nodes, row sep=3.5em, column sep=2.5em, text height=1.25ex, text depth=0.25ex]
{ 
{} & (a) & {} \\
(b) & {} & (c) \\
};
\path[->]
(m-1-2) edge [double, blue] node[left] {\scriptsize$\mathcal{L}^+ \subseteq \mathcal{A}$ \text{ } + \text{ $\mathcal{L}$ closed under exts. }} (m-2-1)
(m-2-1) edge [double, blue] (m-2-3)
(m-2-3) edge [double, red] node[right] {\scriptsize$\mbox{ \ } (\mathcal{L},\mathcal{A}) \text{ duality pair } + \mathcal{A} \text{ closed under exts. } + \text{ } \nu \text{ is a rel. cogen. in } \mathcal{A}$} (m-1-2)
;
\end{tikzpicture}
\caption{At each implication, it is assumed that $\mathcal{L}$ is a relative generator in $\mathsf{Mod}(R)$.}
\end{figure}
\end{proposition}

\begin{proof} \
\begin{enumerate}
\item 
\begin{itemize}
\item (a) $\Rightarrow$ (b): Consider a short exact sequence $M_1 \rightarrowtail M_2 \twoheadrightarrow M_3$ with $M_1, M_3 \in \mathcal{GF}_{(\mathcal{L},\nu)}$. We show that $M_2 \in \mathcal{GF}_{(\mathcal{L},\nu)}$ from (a). By Proposition \ref{Bennis} (2), we have that $M_1, M_3 \in \nu^\top$, and so $M_2 \in \nu^\top$ since $\nu^\top$ is closed under extensions. On the other hand, there is an exact sequence $M_3' \rightarrowtail L \twoheadrightarrow M_3$ with $L \in \mathcal{L}$ and $M'_3 \in \mathcal{GF}_{(\mathcal{L},\nu)}$. Taking the pullback of $M_2 \twoheadrightarrow M_3 \twoheadleftarrow L$ yields the following commutative exact diagram: 
\[
%\parbox{1.5in}{
\hspace{1.5cm}\begin{tikzpicture}[description/.style={fill=white,inner sep=2pt}] 
\matrix (m) [ampersand replacement=\&, matrix of math nodes, row sep=2.5em, column sep=2.5em, text height=1.25ex, text depth=0.25ex] 
{ 
{} \& M'_3 \& M'_3 \\
M_1 \& N \& L \\
M_1 \& M_2 \& M_3 \\
}; 
\path[->] 
(m-2-2)-- node[pos=0.5] {\footnotesize$\mbox{\bf pb}$} (m-3-3) 
;
\path[>->]
(m-1-2) edge (m-2-2) (m-1-3) edge (m-2-3)
(m-2-1) edge (m-2-2) (m-3-1) edge (m-3-2)
;
\path[->>]
(m-2-2) edge (m-3-2) edge (m-2-3) (m-2-3) edge (m-3-3) (m-3-2) edge (m-3-3)
;
\path[-,font=\scriptsize]
(m-1-2) edge [double, thick, double distance=2pt] (m-1-3)
(m-2-1) edge [double, thick, double distance=2pt] (m-3-1)
;
\end{tikzpicture}. 
%}
\]
Now for $M_1$ there is also an exact sequence $M_1 \rightarrowtail L' \twoheadrightarrow M'_1$ with $L \in \mathcal{L}$ and $M'_1 \in \mathcal{GF}_{(\mathcal{L},\nu)}$. Taking the pushout of $L' \leftarrowtail M_1 \rightarrowtail N$ yields the following commutative exact diagram:
\[
%\parbox{1.5in}{
\hspace{1.5cm}\begin{tikzpicture}[description/.style={fill=white,inner sep=2pt}] 
\matrix (m) [ampersand replacement=\&, matrix of math nodes, row sep=2.5em, column sep=2.5em, text height=1.25ex, text depth=0.25ex] 
{ 
M_1 \& N \& L \\
L' \& N' \& L \\
M'_1 \& M'_1 \\
}; 
\path[->] 
(m-1-1)-- node[pos=0.5] {\footnotesize$\mbox{\bf po}$} (m-2-2) 
;
\path[>->]
(m-1-1) edge (m-1-2) edge (m-2-1)
(m-1-2) edge (m-2-2)
(m-2-1) edge (m-2-2)
;
\path[->>]
(m-1-2) edge (m-1-3)
(m-2-2) edge (m-2-3)
(m-2-1) edge (m-3-1)
(m-2-2) edge (m-3-2)
;
\path[-,font=\scriptsize]
(m-1-3) edge [double, thick, double distance=2pt] (m-2-3)
(m-3-1) edge [double, thick, double distance=2pt] (m-3-2)
;
\end{tikzpicture}. 
%}
\]
Since $\mathcal{L}$ is closed under extensions, one has that $N' \in \mathcal{L}$. Then from the central column and Proposition \ref{Bennis} (2) we have that $N \in \mathcal{GF}_{(\mathcal{L},\nu)}$. Finally, from the central column of the first pullback diagram and the assumption (a), we have that $M_2 \in \mathcal{GF}_{(\mathcal{L},\nu)}$. 

\item (b) $\Rightarrow$ (c): It follows in the same way as the proof of (1) {$\Rightarrow$} (2) in \cite[Thm. 2.3]{GFclosed}, and by using Proposition \ref{Bennis} (2). 
\end{itemize}

\item Suppose we are given a short exact sequence $G_1 \rightarrowtail G_0 \twoheadrightarrow M$, where $G_0, G_1 \in \mathcal{GF}_{(\mathcal{L},\nu)}$ and $M \in \nu^{\top_1}$. First, there is a short exact sequence $G_1 \rightarrowtail L_1 \twoheadrightarrow G'_1$ where $L_1 \in \mathcal{L}$ and $G'_1 \in \mathcal{GF}_{(\mathcal{L},\nu)}$. Taking the pushout of $L_1 \leftarrowtail G_1 \rightarrowtail G_0$ yields the following commutative exact diagram:
\[
%\parbox{1.5in}{
\hspace{1.5cm}\begin{tikzpicture}[description/.style={fill=white,inner sep=2pt}] 
\matrix (m) [ampersand replacement=\&, matrix of math nodes, row sep=2.5em, column sep=2.5em, text height=1.25ex, text depth=0.25ex] 
{ 
G_1 \& G_0 \& M \\
L_1 \& G'_0 \& M \\
G'_1 \& G'_1 \\
}; 
\path[->] 
(m-1-1)-- node[pos=0.5] {\footnotesize$\mbox{\bf po}$} (m-2-2) 
;
\path[>->]
(m-1-1) edge (m-1-2) edge (m-2-1)
(m-1-2) edge (m-2-2)
(m-2-1) edge (m-2-2)
;
\path[->>]
(m-1-2) edge (m-1-3)
(m-2-2) edge (m-2-3)
(m-2-1) edge (m-3-1)
(m-2-2) edge (m-3-2)
;
\path[-,font=\scriptsize]
(m-1-3) edge [double, thick, double distance=2pt] (m-2-3)
(m-3-1) edge [double, thick, double distance=2pt] (m-3-2)
;
\end{tikzpicture}. 
%}
\]
Since $\mathcal{GF}_{(\mathcal{L},\nu)}$ is closed under extensions by (c), we have that $G'_0 \in \mathcal{GF}_{(\mathcal{L},\nu)}$. So there is a short exact sequence $G'_0 \rightarrowtail L_0 \twoheadrightarrow G''_0$ with $L_0 \in \mathcal{L}$ and $G''_0 \in \mathcal{GF}_{(\mathcal{L},\nu)}$. Now taking the pushout of $L_0 \leftarrowtail G'_0 \twoheadrightarrow M$ yields the following commutative exact diagram:
\[
%\parbox{1.5in}{
\hspace{1.5cm}\begin{tikzpicture}[description/.style={fill=white,inner sep=2pt}] 
\matrix (m) [ampersand replacement=\&, matrix of math nodes, row sep=2.5em, column sep=2.5em, text height=1.25ex, text depth=0.25ex] 
{ 
L_1 \& G'_0 \& M \\
L_1 \& L_0 \& M' \\
{} \& G''_0 \& G''_0 \\
}; 
\path[->] 
(m-1-2)-- node[pos=0.5] {\footnotesize$\mbox{\bf po}$} (m-2-3) 
;
\path[>->]
(m-1-1) edge (m-1-2)
(m-2-1) edge (m-2-2)
(m-1-2) edge (m-2-2)
(m-1-3) edge (m-2-3)
;
\path[->>]
(m-1-2) edge (m-1-3)
(m-2-2) edge (m-2-3)
(m-2-2) edge (m-3-2)
(m-2-3) edge (m-3-3)
;
\path[-,font=\scriptsize]
(m-1-1) edge [double, thick, double distance=2pt] (m-2-1)
(m-3-2) edge [double, thick, double distance=2pt] (m-3-3)
;
\end{tikzpicture}. 
%}
\]
From this point, our proof differs from the one appearing in \cite{GFclosed}. Consider the induced short exact sequence of $R^{\rm o}$-modules $(M')^+ \rightarrowtail L^+_0 \twoheadrightarrow L^+_1$, where $L^+_0, L^+_1 \in \mathcal{A}$. Since $\nu$ is a relative generator in $\mathcal{A}$, there is an exact sequence $A \rightarrowtail N \twoheadrightarrow L^+_1$ with $A \in \mathcal{A}$ and $N \in \nu$. Now take the pullback of $L^+_0 \twoheadrightarrow L^+_1 \twoheadleftarrow N$ to obtain the following commutative exact diagram:
\[
%\parbox{1.5in}{
\hspace{1.5cm}\begin{tikzpicture}[description/.style={fill=white,inner sep=2pt}] 
\matrix (m) [ampersand replacement=\&, matrix of math nodes, row sep=2.5em, column sep=2.5em, text height=1.25ex, text depth=0.25ex] 
{ 
{} \& A \& A \\
(M')^+ \& A' \& N \\
(M')^+ \& L^+_0 \& L^+_1 \\
}; 
\path[->] 
(m-2-2)-- node[pos=0.5] {\footnotesize$\mbox{\bf pb}$} (m-3-3) 
;
\path[>->]
(m-1-2) edge (m-2-2)
(m-1-3) edge (m-2-3)
(m-2-1) edge (m-2-2)
(m-3-1) edge (m-3-2)
;
\path[->>]
(m-2-2) edge (m-2-3) edge (m-3-2)
(m-2-3) edge (m-3-3)
(m-3-2) edge (m-3-3)
;
\path[-,font=\scriptsize]
(m-1-2) edge [double, thick, double distance=2pt] (m-1-3)
(m-2-1) edge [double, thick, double distance=2pt] (m-3-1)
;
\end{tikzpicture}. 
%}
\]
Since $\mathcal{A}$ is closed under extensions, we have that $A' \in \mathcal{A}$. Moreover, $\Ext^1_{R^{\rm o}}(N,(M')^+) \cong \Tor^R_1(N,M')^+$, where $\Tor^R_1(N,M') = 0$ since $M \in \nu^{\top_1}$ from the assumption, $G''_0 \in \nu^\top$ by Proposition \ref{Bennis} (2), and $\nu^{\top_1}$ is closed under extensions. It follows that the central row in the previous diagram splits, and so $(M')^+ \in \mathcal{A}$ since $\mathcal{A}$ is closed under direct summands. The latter also implies that $M' \in \mathcal{L}$, and thus $M \in \mathcal{GF}_{(\mathcal{L},\nu)}$ by Proposition \ref{Bennis} (2). 
\end{enumerate}
\end{proof}

Let us point out in the following result some consequences of Theorem \ref{dualidad}, which summarizes the homological aspects of Gorenstein $(\mathcal{L},\nu)$-flat $R$-modules.

\begin{corollary} \label{resolvente} 
Let $(\mathcal{L,A})$ be a product closed bicomplete duality pair. The following assertions hold:
\begin{enumerate}
\item $\mathcal{GF}_{(\mathcal{L},\nu)}$ is a resolving class closed under direct summands. 

\item $(\mathcal{GF}_{(\mathcal{L},\nu)},\mathcal{GI}_{(\nu,\mathcal{A})})$ is a perfect duality pair. Furthermore, $(\mathcal{GF}_{(\mathcal{L},\nu)},\mathcal{GC}_{(\mathcal{L},\nu)})$ is a hereditary perfect cotorsion pair in $\mathsf{Mod}(R)$. In particular, every $R$-module has a Gorenstein $(\mathcal{L},\nu)$-flat cover. If in addition, $\mathcal{GF}_{(\mathcal{L}, \nu)}$ is closed under arbitrary direct products, then every $R$-module has a Gorenstein $(\mathcal{L},\nu)$-flat preenvelope. 

\item $\mathcal{GF}_{(\mathcal{L},\nu)}$ is a Kaplansky class closed under direct limits.

\item If $G_1 \rightarrowtail G_0 \twoheadrightarrow M$ is a short exact sequence with $G_0, G_1 \in \mathcal{GF}_{(\mathcal{L},\nu)}$ and $M \in \nu^{\top_1}$, then $M \in \mathcal{GF}_{(\mathcal{L},\nu)}$. 
%\item $\GF_{(\mathcal{GF}_{(\mathcal{L},\nu)},\nu)} = \GF_{(\mathcal{L},\nu)}$. 
\end{enumerate}
\end{corollary}

\begin{proof} \
\begin{enumerate}
\item By Proposition \ref{prop:GF_closed_under_exts} and the fact that $\mathcal{P}(R) \subseteq \Le \subseteq \GF _{(\Le, \nu)}$, it suffices to show that $\mathcal{GF}_{(\mathcal{L},\nu)}$ is closed under direct summands and extensions. Both closure properties will follow from the duality relation between $\mathcal{GF}_{(\mathcal{L},\nu)}$ and $\mathcal{GI}_{(\nu,\mathcal{A})}$ proved in Theorem \ref{dualidad}. Indeed, suppose that $N \in \mathsf{Mod}(R)$ is a direct summand of $M \in \mathcal{GF}_{(\mathcal{L},\nu)}$, that is, there is a split monomorphism $N \rightarrowtail M$. Applying the (contravariant) functor $(-)^+$ yields a split epimorphism $M^+ \twoheadrightarrow N^+$, and so $N^+$ is a direct summand of $M^+ \in \mathcal{GI}_{(\nu,\mathcal{A})}$ (by Theorem \ref{dualidad}). Being $(\nu,\mathcal{A})$ a GI-admissible pair, we know from the dual of \cite[Coroll. 3.33]{BMS} that $\mathcal{GI}_{(\nu,\mathcal{A})}$ is closed under direct summands, and so $N^+ \in \mathcal{GI}_{(\nu,\mathcal{A})}$. Again by Theorem \ref{dualidad}, we conclude that $N \in \mathcal{GF}_{(\mathcal{L},\nu)}$. The remaining closure property follows similarly, using the fact that $(-)^+$ is an exact functor along with Theorem \ref{dualidad} and the dual of \cite[Coroll. 3.33]{BMS}. 
  
\item First, the equivalence between $M \in \mathcal{GF}_{(\mathcal{L},\nu)}$ and $M^+ \in \mathcal{GI}_{(\nu,\mathcal{A})}$ follows by Theorem \ref{dualidad}. On the other hand, $\mathcal{GI}_{(\nu,\mathcal{A})}$ is closed under direct summands and finite direct sums by the dual of \cite[Coroll. 3.33]{BMS}, since $(\nu,\mathcal{A})$ is a GI-admissible pair. Hence, $(\mathcal{GF}_{(\mathcal{L},\nu)},\mathcal{GI}_{(\nu,\mathcal{A})})$ is a duality pair. Clearly, $\mathcal{GF}_{(\mathcal{L},\nu)}$ contains $R$, and is closed under extensions by part (1). Finally, consider a family $\{ M_i \text{ : } i \in I \}$ of $R$-modules in $\mathcal{GF}_{(\mathcal{L},\nu)}$. For the coproduct $\bigoplus_{i \in I} M_i$, one can easily note that $\bigoplus_{i \in I} M_i \in \nu^\top$. On the other hand, each $M_i$ admits a $(\nu \otimes_R -)$-acyclic $\mathcal{L}$-coresolution 
\[
\eta_i \colon M_i \rightarrowtail L^0_i \to L^1_i \to \cdots.
\] 
Since $\mathcal{L}$ is closed under coproducts, and the coproduct of exact complexes of $R$-modules is again exact, we have that 
\[
\bigoplus \eta_i \colon \bigoplus_{i \in I} M_i \rightarrowtail \bigoplus_{i \in I} L^0_i \to \bigoplus_{i \in I} L^1_i \to \cdots
\] 
is an $\mathcal{L}$-coresolution of $\bigoplus_{i \in I} M_i$. Moreover, since tensor products preserve coproducts, the complexes $A \otimes_R (\bigoplus \eta_i)$ and $\bigoplus (A \otimes_R \eta_i)$ are isomorphic, where each complex $A \otimes_R \eta_i$ is exact for every $A \in \nu$. It then follows that $\bigoplus \eta_i$ is $(\nu \otimes_R -)$-acyclic. Again by Proposition \ref{Bennis}, we obtain that $\bigoplus_{i \in I} M_i \in \mathcal{GF}_{(\mathcal{L},\nu)}$. 

\item It is a consequence of part (2) and Proposition \ref{prop:purity_and_Kaplansky}.

\item Follows by Proposition \ref{prop:GF_closed_under_exts}. 
%\item Follows by \cite[Thm. 2.7]{WangDi} since $\mathcal{GF}_{(\mathcal{L}, \nu)}$ is closed under extensions. 
\end{enumerate}
\end{proof}

%%%%%%%%%%%%%%%%%%%%%%%%%%%%%%%%%%%%%%%%%%%%%%%%
%%%%%%%%%%%%%%%%%%%%%%%%%%%%%%%%%%%%%%%%%%%%%%%%
%%%%%%%%%%%%%%%%%%%%%%%%%%%%%%%%%%%%%%%%%%%%%%%%
%%%%%%%%%%%%%%%%%%%%%%%%%%%%%%%%%%%%%%%%%%%%%%%%

\section{Gorenstein flat dimensions relative to duality pairs}\label{sec:relative_G-flat_dimensions} 

We are now interested in studying homological dimensions constructed from Gorenstein $(\mathcal{L,A})$-flat modules.

\begin{definition}
Given $M \in \mathsf{Mod}(R)$, the \textbf{Gorenstein $\bm{(\mathcal{L,A})}$-flat dimension} of $M$ is defined as 
\[
\Gfd _{(\mathcal{L,A})}(M) :=  \resdim_{\mathcal{GF}_{(\mathcal{L,A})}}(M).
\]
\end{definition}

In this section, we focus on Gorenstein flat dimensions relative to $(\mathcal{L},\nu)$.

\begin{proposition}\label{sequences}
If $\GF _{(\Le,\nu)}$ is closed under extensions, then for every $M \in \mathsf{Mod}(R)$ with finite Gorenstein $(\mathcal{L},\nu)$-flat dimension, the following statements are true:
\begin{enumerate}
\item There exists a short exact sequence $K \rightarrowtail G \twoheadrightarrow M$ with $G \in \GF_{(\Le,\nu)}$ and $\resdim_{\Le}(K) = \Gfd_{(\mathcal{L},\nu)}(M) - 1$. 

\item Suppose in addition that $\mathcal{L}$ is a relative generator in $\mathsf{Mod}(R)$ closed under extensions. If either $\Le$ is closed under epikernels or $\mathcal{L}^+ \subseteq \mathcal{A}$, then there exists a short exact sequence $M \rightarrowtail H \twoheadrightarrow G'$ with $\resdim_{\Le}(H) = \Gfd_{(\mathcal{L},\nu)}(M)$ and $G' \in \GF_{(\Le,\nu)}$. 
\end{enumerate}
\end{proposition}

\begin{proof} 
Part (1) follows from \cite[Thm. 2.8 (a)]{BMS}. We split the proof of Part (2) into two cases:
\begin{itemize}
\item In the case where $\mathcal{L}$ is a relative generator in $\mathsf{Mod}(R)$ closed under epikernels, note that the proof given in \cite[Lem. 2.2]{Bennis10} can be adapted replacing $\resdim_{\Le}(-)$ by the flat dimension $ \mathrm{fd}(-)$, and using the fact that for every exact sequence $K' \rightarrowtail F \twoheadrightarrow K$ with $F \in \Le$ and $\resdim_{\Le}(K') < \infty$, one has that $\resdim_{\Le}(K) = \resdim_{\Le}(K') +1$.

\item Now suppose that $\mathcal{L}$ is a relative generator in $\mathsf{Mod}(R)$ closed under extensions with $\mathcal{L}^+ \subseteq \mathcal{A}$. By \cite[Thm. 2.8 (a)]{BMS}, there is a short exact sequence $M \rightarrowtail H \twoheadrightarrow G'$ with $G' \in \mathcal{GF}_{(\mathcal{L},\nu)}$ and $\resdim_{\mathcal{L}}(H) \leq \Gfd_{(\mathcal{L},\nu)}(M)$. So it is only left to show that $\resdim_{\mathcal{L}}(H) = \Gfd_{(\mathcal{L},\nu)}(M)$. Let us suppose that $\resdim_{\mathcal{L}}(H) < \Gfd_{(\mathcal{L},\nu)}(M) = n$. Then, there is an exact sequence 
\[
\hspace{2.25cm} L_k \rightarrowtail L_{k-1} \to \cdots \to L_1 \to L_0 \twoheadrightarrow H,
\]      
where $k < n$ and with $L_i \in \mathcal{L}$ for every $0 \leq i \leq k$. Now taking the pullback of $M \rightarrowtail H \twoheadleftarrow L_0$ yields the following commutative exact diagram:
\[
\hspace{3cm}\begin{tikzpicture}[description/.style={fill=white,inner sep=2pt}]
\matrix (m) [matrix of math nodes, row sep=2.5em, column sep=2.5em, text height=1.25ex, text depth=0.25ex]
{ 
L_k & L_k & {} \\
L_{k-1} & L_{k-1} & {} \\
\vdots & \vdots & {} \\
L_1 & L_1 & {} \\
L'_0 & L_0 & G' \\
M & H & G' \\
};
\path[->]
(m-5-1)-- node[pos=0.5] {\footnotesize$\mbox{\bf pb}$} (m-6-2)
(m-2-1) edge (m-3-1) (m-2-2) edge (m-3-2)
(m-3-1) edge (m-4-1) (m-3-2) edge (m-4-2)
(m-4-1) edge (m-5-1) (m-4-2) edge (m-5-2)
;
\path[>->]
(m-1-1) edge (m-2-1)
(m-1-2) edge (m-2-2)
(m-5-1) edge (m-5-2)
(m-6-1) edge (m-6-2)
;
\path[->>]
(m-5-1) edge (m-6-1)
(m-5-2) edge (m-6-2)
(m-5-2) edge (m-5-3)
(m-6-2) edge (m-6-3)
;
\path[-,font=\scriptsize]
(m-1-1) edge [double, thick, double distance=2pt] (m-1-2)
(m-2-1) edge [double, thick, double distance=2pt] (m-2-2)
(m-4-1) edge [double, thick, double distance=2pt] (m-4-2)
(m-5-3) edge [double, thick, double distance=2pt] (m-6-3)
;
\end{tikzpicture}.
\]
By Proposition \ref{prop:GF_closed_under_exts}, we have that $\mathcal{GF}_{(\mathcal{L},\nu)}$ is closed under epikernels, since $\mathcal{GF}_{(\mathcal{L},\nu)}$ is closed under extensions. Then $L'_0 \in \mathcal{GF}_{(\mathcal{L},\nu)}$. It follows that $\Gfd_{(\mathcal{L},\nu)}(M) < n$, getting a contradiction. 
\end{itemize}
\end{proof}

The following result extends the duality between $\mathcal{GF}_{(\mathcal{L},\nu)}$ and $\mathcal{GI}_{(\nu,\mathcal{A})}$ (Theorem \ref{dualidad}) to their corresponding dimensions. Recall that if $(\mathcal{X,Y})$ is a GI-admissible pair, the \emph{$(\mathcal{X,Y})$-Gorenstein injective dimension} of $N \in \mathsf{Mod}(R^{\rm o})$ is defined as
\[
\Gid_{(\mathcal{X,Y})}(N) := \coresdim_{\mathcal{GI}_{(\mathcal{X,Y})}(R^{\rm o})}(N).
\]

\begin{proposition} \label{iguales}
For every $M \in \mathsf{Mod}(R)$ and every product closed bicomplete duality pair $(\mathcal{L}, \mathcal{A})$, the equality $\Gid_{(\nu,\mathcal{A})}(M^+) = \Gfd_{(\mathcal{L},\nu)}(M)$ holds true.
\end{proposition}

\begin{proof}
Let us first analyze the case where $M \in \GF_{(\Le,\nu)}^\wedge$, and say $\Gfd_{(\Le,\nu)}(M) = n$. Then, $M$ admits a finite Gorenstein $(\mathcal{L},\nu)$-flat resolution of length $n$. By Theorem \ref{dualidad} and the exactness of $(-)^+$, we can note that $M^+$ admits a finite $(\nu,\mathcal{A})$-Gorenstein injective coresolution of length $n$. Thus, $\Gid_{(\nu,\mathcal{A})}(M^+) \leq n$, and we can set $\Gid_{(\nu,\A)}(M^+) = m$. Since $(\mathcal{GF}_{(\mathcal{L},\nu)},\mathcal{GC}_{(\mathcal{L},\nu)})$ is a perfect cotorsion pair by Corollary \ref{resolvente}, we can construct a partial Gorenstein $(\mathcal{L},\nu)$-flat resolution
\[
\Omega^{\mathcal{GF}_{(\mathcal{L},\nu)}}_{m}(M) \rightarrowtail G_{m-1} \to \cdots \to G_0 \twoheadrightarrow M 
\]
with $G_i \in \GF _{(\Le,\nu)}$ for every $1 \leq i \leq n-1$. Again by Theorem \ref{dualidad} and the exactness of $(-)^+$, we obtain the exact sequence
\[
 M^{+} \rightarrowtail  G_0 ^{+} \to \cdots \to G_{m-1}^{+} \twoheadrightarrow (\Omega^{\mathcal{GF}_{(\mathcal{L},\nu)}}_{m}(M))^{+}, 
\]
with $G_i ^+ \in \GI_{(\nu,\A)}$ for every $1 \leq i \leq m-1$. Moreover, since $\nu$ is closed under direct summands, by the dual of \cite[Coroll. 4.10]{BMS} we have that $(\Omega^{\mathcal{GF}_{(\mathcal{L},\nu)}}_{m}(M))^+ \in \GI_{(\nu,\A)}$. This in turn implies that $\Omega^{\mathcal{GF}_{(\mathcal{L},\nu)}}_{m}(M) \in \GF_{(\Le,\nu)}$, and hence $\Gfd_{(\mathcal{L},\nu)}(M) \leq m$.

Finally, from the previous reasoning, it is clear that $\Gfd _{(\Le,\nu)}(M) = \infty$ and $\Gid_{(\nu,\A)}(M^{+}) < \infty$ (or $\Gid _{(\nu,\A)}(M^+) = \infty$ and $\Gfd _{(\Le,\nu)}(M) < \infty$) combined imply a contradiction. 
\end{proof}

Theorem \ref{theo:finiteness_GFdim} below gives a functorial description of the Gorenstein $(\mathcal{L},\nu)$-flat dimension and shows its stability (recall Definition \ref{def:stable}). In order to prove it, it will be useful to recall the following functorial characterization of relative Gorenstein injective dimensions, which follows from the dual of \cite[Corolls. 4.10 \& 4.11 (a)]{BMS}.

\begin{lemma} \label{GI-admisible}
Let $(\X, \Y) $ be a GI-admissible pair in a Grothendieck category $\mathcal{G}$ such that $\X \cap \Y$ is closed under direct summands. For every $N \in \mathcal{G}$ with $\Gid_{(\mathcal{X,Y})}(N) < \infty$ and $n \in \mathbb{Z}_{\geq 0}$, the following statements are equivalent:
\begin{enumerate}
\item[(a)] $\Gid _{(\X, \Y)}(N) \leq n$.

\item[(b)] $\Ext^{\geq n+1}_{\mathcal{G}}((\X \cap \Y)^{\vee},N) = 0$.

\item[(c)] $\Ext^{\geq n+1}_{\mathcal{G}}(\X \cap \Y,N) = 0$.

\item[(d)] If $N \rightarrowtail G^0 \to G^1  \to \cdots \to G^{n-1} \twoheadrightarrow K^n $ is an exact sequence with $G^i \in \GI_{(\X,\Y)}$ for every $1 \leq i \leq n-1$, then $K^n \in \GI_{(\X, \Y)}$.
\end{enumerate} 
\end{lemma}

\begin{theorem}\label{theo:finiteness_GFdim}
Consider the statements below for $M \in \mathsf{Mod}(R)$ with $\Gfd_{(\mathcal{L},\nu)}(M) < \infty$ and $n \in \mathbb{Z}_{\geq 0}$:
\begin{enumerate}
\item[(a)] $\Gfd_{(\Le,\nu)}(M) \leq n$.

\item[(b)] $\Tor^R_{\geq n+1}(\nu^\vee, M) = 0$.

\item[(c)] $\Tor^R_{\geq n+1}(\nu, M) = 0$.

\item[(d)] If $K_n \rightarrowtail G_{n-1} \to \cdots \to G_0 \twoheadrightarrow M$ is an exact sequence with $G_i \in \GF_{(\Le,\nu)}$ for every $1 \leq i \leq n-1$, then $K_n \in \GF_{(\Le,\nu)}$.
\end{enumerate}
The following assertions hold true whenever $\mathcal{L}^+ \subseteq \mathcal{A}$:
\begin{enumerate}
\item If $\mathcal{A}^+ \subseteq \mathcal{L}$, $\mathcal{L}$ is preenveloping and $(\nu,\mathcal{A})$ is a GI-admissible pair, then (a) $\Rightarrow$ (b) $\Leftrightarrow$ (c).
 
\item If $\mathcal{L}$ is a relative generator in $\mathsf{Mod}(R)$ closed under extensions, coproducts and direct summands, and $\GF_{(\Le,\nu)}$ is closed under extensions, then (a) $\Leftrightarrow$ (d). If in addition $(\mathcal{L,A})$ is a duality pair and $(\nu,\mathcal{A})$ is a GI-admissible pair, then (b) $\Rightarrow$ (a) as well. 
\end{enumerate} 
In particular, if $(\mathcal{L,A})$ is a product closed bicomplete duality pair, then the four assertions are equivalent and the Gorenstein $(\mathcal{L},\nu)$-flat dimension is stable. 
\end{theorem}

\begin{proof} \
\begin{enumerate}
\item The implication (b) $\Rightarrow$ (c) is trivial, (b) $\Leftarrow$ (c) follows by a dimension shifting argument, and (a) $\Rightarrow$ (c) follows by induction on $n$, Theorem \ref{dualidad} (1) and (2)-(ii).

\item By Proposition \ref{prop:GF_closed_under_exts} we have that $\GF_{(\Le,\nu)}$ is a preresolving class. Now let $P$ be a projective $R$-module. Since $\mathcal{L}$ is a relative generator in $\mathsf{Mod}(R)$, we have a split epimorphism $L \twoheadrightarrow P$. So $P$ is a direct summand of $L$, which in turn implies that $P \in \mathcal{L}$. Hence, $\mathcal{P}(R) \subseteq \mathcal{GF}_{(\mathcal{L},\nu)}$ and so $\mathcal{GF}_{(\mathcal{L},\nu)}$ is resolving. On the other hand, the closure under coproducts of $\mathcal{L}$ implies the same property for $\GF_{(\Le,\nu)}$. Then using Eilenberg's swindle one can show that $\GF_{(\Le,\nu)}$ is closed under direct summands. Therefore, Auslander and Bridger's \cite[Lem. 3.12]{AuBri69} implies the equivalence between (a) and (d).

We show (b) $\Rightarrow$ (a). Suppose $\Gfd_{(\Le,\nu)}(M) < \infty$ and $\Tor^R_{\geq n+1}(\nu^\vee, M) = 0$. From the equivalence between (a) and (d), we can construct for some $m > n$ a partial projective resolution
\[
\Omega^{\mathcal{P}(R)}_{m}(M) \rightarrowtail P_{m-1} \cdots \to P_1 \to P_0 \twoheadrightarrow M, 
\]
that is, with $P_k \in \mathcal{P}(R)$ for every $0 \leq k \leq m-1$, and guarantee that $\Omega^{\mathcal{P}(R)}_{m}(M) \in \GF_{(\Le,\nu)}$. Consider $\Omega^{\mathcal{P}(R)}_{n}(M)$ in the previous partial projective resolution, and the short exact sequence 
\[
\Omega^{\mathcal{P}(R)}_{m}(M) \rightarrowtail P_{m-1} \twoheadrightarrow \Omega^{\mathcal{P}(R)}_{m-1}(M).
\] 
By dimension shifting, 
\[
\Tor^R _i(A, \Omega^{\mathcal{P}(R)}_{m-1}(M)) \cong \Tor^R_{m-1 + i}(A, M) = 0
\] 
for every $A \in \nu$ and $i \in \mathbb{Z}_{>0}$, that is, $\Omega^{\mathcal{P}(R)}_{m-1}(M) \in \nu^\perp$. Then, since $\Omega^{\mathcal{P}(R)}_{m}(M), P_{m-1} \in \mathcal{GF}_{(\mathcal{L},\nu)}$, we get from Proposition \ref{prop:GF_closed_under_exts} and the hypothesis that $\Omega^{\mathcal{P}(R)}_{m-1}(M) \in \mathcal{GF}_{(\mathcal{L},\nu)}$. Therefore, continuing this process finitely many times, we can conclude that $\Omega^{\mathcal{P}(R)}_{n}(M) \in \mathcal{GF}_{(\mathcal{L},\nu)}$. 
\end{enumerate}
\end{proof}

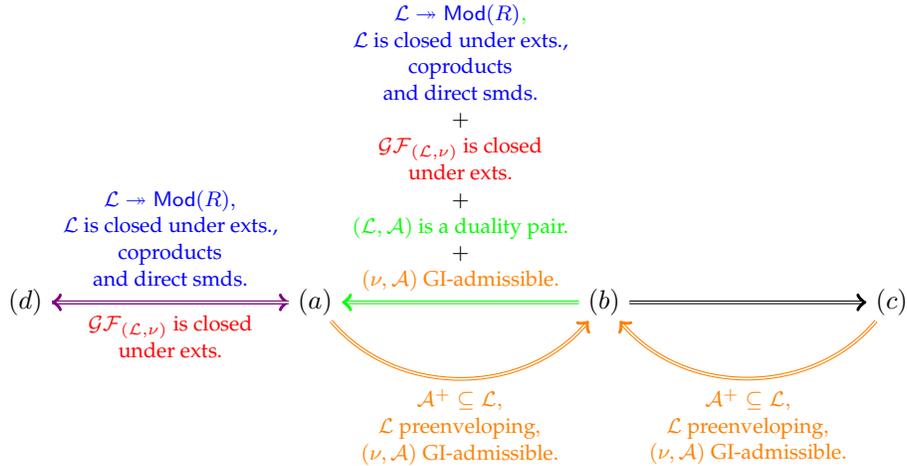
\begin{figure}[h!]
\begin{tikzpicture}[description/.style={fill=white,inner sep=2pt}]
\matrix (m) [matrix of math nodes, row sep=1em, column sep=9em, text height=1.25ex, text depth=0.25ex]
{ 
(d) & (a) & (b) & (c) \\
};
\path[->]
(m-1-1) edge [double, violet] node[above] {{\color{blue}{\footnotesize$\begin{array}{c} \mathcal{L} \twoheadrightarrow \mathsf{Mod}(R), \\ \text{$\mathcal{L}$ is closed under exts.,} \\ \text{coproducts} \\ \text{and direct smds.} \end{array}$}}} (m-1-2)
(m-1-2) edge [double, violet] node[below] {{\color{red}{\footnotesize$\begin{array}{c} \text{$\mathcal{GF}_{(\mathcal{L},\nu)}$ is closed} \\ \text{under exts.} \end{array}$}}} (m-1-1)
(m-1-3) edge [double, green] node[above] {\footnotesize$\begin{array}{c} {\color{blue}{\mathcal{L} \twoheadrightarrow \mathsf{Mod}(R)}}, \\ \text{{\color{blue}{$\mathcal{L}$ is closed under exts.,}}} \\ \text{{\color{blue}{coproducts}}} \\ \text{{\color{blue}{and direct smds.}}} \\ {\color{black}{+}} \\ \text{{\color{red}{$\mathcal{GF}_{(\mathcal{L},\nu)}$ is closed}}} \\ \text{{\color{red}{under exts.}}} \\ {\color{black}{+}} \\ \text{{\color{green}{$(\mathcal{L,A})$ is a duality pair.}}} \\ {\color{black}{+}} \\ \text{{\color{orange}{$(\nu,\mathcal{A})$ GI-admissible.}}} \end{array}$} (m-1-2)
(m-1-2) edge [double, bend right=50, orange] node[below] {\footnotesize$\begin{array}{c} \mathcal{A}^+ \subseteq \mathcal{L}, \\ \text{$\mathcal{L}$ preenveloping}, \\ \text{$(\nu,\mathcal{A})$ GI-admissible.} \end{array}$} (m-1-3)
(m-1-3) edge [double] (m-1-4)
(m-1-4) edge [double, bend left=50, orange] node[below] {\footnotesize$\begin{array}{c} \mathcal{A}^+ \subseteq \mathcal{L}, \\ \text{$\mathcal{L}$ preenveloping}, \\ \text{$(\nu,\mathcal{A})$ GI-admissible.} \end{array}$} (m-1-3)
;
\end{tikzpicture}
\caption{A display of the implications studied in Theorem \ref{theo:finiteness_GFdim} for pairs $(\mathcal{L,A})$ with $\mathcal{L}^+ \subseteq \mathcal{A}$.}
\end{figure}

\begin{remark}\label{rem:dim_Tor_orthogonal}
Proceeding as in the proof of (b) $\Rightarrow$ (a) in the previous theorem, one can note a similar result for Tor-orthogonal duality pairs, which sort of generalizes \cite[Lem. 2.10]{WangYangZhu19}. Specifically, let $(\mathcal{L,A})$ be a Tor-orthogonal duality pair with $\mathcal{P}(R) \subseteq \mathcal{L}$. Consider the following assertions for an $R$-module $M$ with finite Gorenstein $(\mathcal{L,A})$-flat dimension:
\begin{enumerate}
\item[(a)] $\Gfd_{(\mathcal{L,A})}(M) \leq n$.

\item[(b)] $\Tor^R _{\geq n+1}(\mathcal{A},M) = 0$.

\item[(c)] $\Tor^R _{\geq n+1}(\mathcal{A}^\vee,M) = 0$.
\end{enumerate}
Then (a) $\Rightarrow$ (b) $\Rightarrow$ (c). Furthermore, if $\GF_{(\mathcal{L,A})}$ is closed under extensions, then (c) $\Rightarrow$ (a), and all the assertions are equivalent. In particular, the equivalence between the three assertions hold if $(\mathcal{L,A})$ is a perfect symmetric Tor-orthogonal duality pair with $\mathcal{L}$ closed under epikernels. 
\end{remark}

The following result is a generalization of Corollary \ref{resolvente} for the Gorenstein $(\mathcal{L},\nu)$-flat dimension.

\begin{corollary} \label{resolvente_dimensiones} 
Let $(\mathcal{L,A})$ be a product closed bicomplete duality pair. The following assertions hold for every $n \in \mathbb{Z}_{\geq 0}$:
\begin{enumerate}
\item $\mathcal{GF}_{(\mathcal{L},\nu)}{}^\wedge_n$ is a resolving class closed under direct summands. 

\item $(\mathcal{GF}_{(\mathcal{L},\nu)}{}^\wedge_n,\mathcal{GI}_{(\nu,\mathcal{A})}{}^\vee_n)$ is a perfect duality pair. Furthermore, $\mathcal{GF}_{(\mathcal{L},\nu)}{}^\wedge_n$ is the left half of a hereditary perfect cotorsion pair in $\mathsf{Mod}(R)$, and so every $R$-module has a $\mathcal{GF}_{(\mathcal{L},\nu)}{}^\wedge_n$-cover. If in addition, $\mathcal{GF}_{(\mathcal{L}, \nu)}$ is closed under arbitrary direct products, then every $R$-module has a $\mathcal{GF}_{(\mathcal{L},\nu)}{}^\wedge_n$-preenvelope. 

\item $\mathcal{GF}_{(\mathcal{L},\nu)}{}^\wedge_n$ is a Kaplansky class closed under direct limits.
\end{enumerate}
\end{corollary}

\begin{proof}
Part (1) follows by Theorem \ref{theo:finiteness_GFdim}, and part (3) is a consequence of part (2) and \cite[Prop. 2.3]{GillespieDuality}. The conditions that make $(\mathcal{GF}_{(\mathcal{L},\nu)}{}^\wedge_n,\mathcal{GI}_{(\nu,\mathcal{A})}{}^\vee_n)$ and the properties of the Tor functor. 
\end{proof}

\begin{remark}\label{rem:GF_cogenerated}
Given a product closed bicomplete duality pair $(\mathcal{L,A})$, by the previous corollary we can find for each $n \in \mathbb{Z}_{\geq 0}$ a regular cardinal $\kappa$ such that every $R$-module in $\mathcal{GF}_{(\mathcal{L},\nu)}{}^\wedge_n$ is a transfinite extension of $R$-modules in $\mathcal{GF}_{(\mathcal{L},\nu)}{}^\wedge_n$ with cardinality $\leq \kappa$. So one can obtain a set $\mathcal{S}$ of representatives for the isomorphism classes of all $R$-modules in $\mathcal{GF}_{(\mathcal{L},\nu)}{}^\wedge_n$ with cardinality $\leq \kappa$. By Eklof and Trlifaj Theorem, $\mathcal{S}$ cogenerates a complete cotorsion pair $({}^{\perp_1}(\mathcal{S}^{\perp_1}),\mathcal{S}^{\perp_1})$, and by \cite[Coroll. 3.2.4]{GT} the class ${}^{\perp_1}(\mathcal{S}^{\perp_1})$ consists of all direct summands of transfinite extensions of $R$-modules in $\mathcal{S} \cup \{ R \}$. Since $\mathcal{S} \cup \{ R \} \subseteq \mathcal{GF}_{(\mathcal{L},\nu)}{}^\wedge_n$ and $\mathcal{GF}_{(\mathcal{L},\nu)}{}^\wedge_n$ is closed under direct limits and direct summands, it follows that $\mathcal{GF}_{(\mathcal{L},\nu)}{}^\wedge_n = {}^{\perp_1}(\mathcal{S}^{\perp_1})$. Hence, the hereditary and perfect cotorsion pair $(\mathcal{GF}_{(\mathcal{L},\nu)}{}^\wedge_n,(\mathcal{GF}_{(\mathcal{L},\nu)}{}^\wedge_n)^{\perp_1})$ is also  cogenerated by a set. 
\end{remark}

%%%%%%%%%%%%%%%%%%%%%%%%%%%%%%%%%%%%%%%%%%%%
%%%%%%%%%%%%%%%%%%%%%%%%%%%%%%%%%%%%%%%%%%%%

\subsection*{Global and finitistic relative Gorenstein weak dimensions}

We conclude this section defining and pointing out some properties of the Gorenstein $(\mathcal{L,A})$-weak global and finitistic dimensions of $R$.

\begin{definition}
The \textbf{left Gorenstein $\bm{(\mathcal{L,A})}$-weak global dimension} of $R$ is defined as the value 
\[
{\rm l.Gwgdim}_{(\mathcal{L,A})}(R) := \sup \{ \Gfd_{(\mathcal{L,A})}(M) {\rm \ : \ } M \in \mathsf{Mod}(R) \}.
\]
\end{definition}

We have the following result regarding the finiteness of ${\rm l.Ggwdim}_{(\mathcal{L},\nu)}(R)$, as a consequence of Theorem \ref{theo:finiteness_GFdim}.

\begin{proposition}\label{prop:finiteness_global_dim}
Let $(\Le, \A)$ be a product closed bicomplete duality pair in $\mathsf{Mod}(R)$ and suppose that ${\rm l.Gwgdim}_{(\mathcal{L},\nu)}(R) < \infty$. Then, the following statements are equivalent:
\begin{enumerate}
\item[(a)] ${\rm l.Gwgdim}_{(\mathcal{L},\nu)}(R) \leq n < \infty$.

\item[(b)] $\mathrm{fd}(A) \leq n$ for every $A \in \nu$.

\item[(c)] $\mathrm{fd}(N) \leq n$ for every $N \in \nu^\vee$.
\end{enumerate}
\end{proposition}

\begin{definition}
Given a class $\mathcal{X} \subseteq \mathsf{Mod}(R)$, the \textbf{left $\bm{\mathcal{X}}$-finitistic dimension} of $R$ is defined as the value
\[
l.\mathcal{X}\mbox{-}\mathrm{findim}(R) := \sup \{ \resdim_{\mathcal{X}}(M) {\rm \ : \ } M \in \mathcal{X}^\wedge \}.
\]
In particular, the \textbf{left Gorenstein $\bm{(\mathcal{L,A})}$-weak finitistic dimension} of $R$ is defined as
\[
l.\mathrm{GF}_{(\mathcal{L,A})}\mbox{-}\mathrm{findim}(R) := \sup \{ \Gfd_{(\mathcal{L,A})}(M) {\rm \ : \ } M \in \GF_{(\mathcal{L,A})}^\wedge \}.
\]
\end{definition}

\begin{lemma}\label{lem:GI_A_cogorro}
If $(\nu,\mathcal{A})$ is a GI-admissible pair in a Grothendieck category with $\mathcal{A}$ and $\nu$ closed under direct summands, then the equality $\mathcal{GI}_{(\nu,\mathcal{A})} \cap \mathcal{A}^\vee = \mathcal{A}$ holds. If in addition, $\mathcal{A}$ is closed under monocokernels (and so $(\nu,\mathcal{A})$ is a right Frobenius pair), then $\mathcal{GI}_{(\nu,\mathcal{A})}{}^\vee_m \cap \mathcal{A}^\vee = \mathcal{A}^\vee_m$ holds for every $m \in \mathbb{Z}_{> 0}$. 
\end{lemma}

\begin{proof}
We first show the case $m = 0$. So let $N$ be a $(\nu,\mathcal{A})$-Gorenstein injective object with $\coresdim_{\mathcal{A}}(N) < \infty$. By the dual of \cite[Thm. 2.8]{BMS}, there exists a short exact sequence $N \rightarrowtail A \twoheadrightarrow N'$ with $A \in \mathcal{A}$ and $N' \in \nu^\vee$. This sequence splits by Lemma \ref{GI-admisible}, and hence $N$ is a direct summand of $A$. It follows that $N \in \mathcal{A}$. 

For the rest of the proof, assume that $\mathcal{A}$ is closed under monocokernels. If $N \in \mathcal{GI}_{(\nu,\mathcal{A})}{}^\vee_m \cap \mathcal{A}^\vee$ with $m \in \mathbb{Z}_{> 0}$, since $(\nu,\mathcal{GI}_{(\nu,\mathcal{A})})$ is a right Frobenius pair by the dual of \cite[Thm. 2.8 \& Coroll. 4.10]{BMS}, we can find a short exact sequence $N \rightarrowtail G \twoheadrightarrow N'$ where $G \in \mathcal{GI}_{(\nu,\mathcal{A})}$ and $N' \in \nu^\vee_{m-1}$. Now by the dual of \cite[Thm. 2.1 (a)]{BMPS}, we have that $\mathcal{A}^\vee$ is closed under extensions, and thus $G \in \mathcal{GI}_{(\nu,\mathcal{A})} \cap \mathcal{A}^\vee = \mathcal{A}$ by the case $m = 0$. Hence, $N \in \mathcal{A}^\vee_m$. 
\end{proof}

The following is the dual version of the previous lemma. It is a consequence of the lemma itself, Remark \ref{rmk:hereditarysym} and Proposition \ref{iguales}.

\begin{lemma}\label{lem:GF_L_gorro}
If $(\mathcal{L,A})$ is a product closed bicomplete duality pair, then 
\[
\mathcal{GF}_{(\mathcal{L},\nu)}{}^\wedge_m \cap \mathcal{L}^\wedge = \mathcal{L}^\wedge_m
\] 
for every $m \in \mathbb{Z}_{\geq 0}$. 
\end{lemma}

The following result generalizes \cite[Thm. 3.24]{Holm04} to our setting.

\begin{proposition}\label{prop:finiteness_relative_Gflat_findim}
Suppose $\mathcal{L}$ is a relative generator in $\mathsf{Mod}(R)$ closed under extensions, and that either $\Le$ is closed under epikernels or $\mathcal{L}^+ \subseteq \mathcal{A}$. If $\GF_{(\Le,\nu)}$ is closed under extensions, then the inequality
\begin{align}
l.\mathrm{GF}_{(\mathcal{L},\nu)}\mbox{-}\mathrm{findim}(R) \leq l.\mathcal{L}\mbox{-}\mathrm{findim}(R) \label{GFmenorL}
\end{align}
holds true. Furthermore, the equality holds in the case where $(\mathcal{L,A})$ is a product closed bicomplete duality pair. 
\end{proposition}

\begin{proof}
We may assume that $n := l.\mathcal{L}\mbox{-}\mathrm{findim}(R) < \infty$. For any $M \in \GF_{(\Le,\nu)}^\wedge$, we know by Proposition \ref{sequences} (2), that there is a short  exact sequence $M \rightarrowtail  H \twoheadrightarrow G$ with $G \in \GF_{(\Le,\nu)}$ and $\Gfd_{(\Le,\nu)}(M) = \resdim _{\Le}(H) \leq n$. Hence, \eqref{GFmenorL} holds. 

In order to show the remaining inequality, assuming that $(\mathcal{L,A})$ is a product closed bicomplete duality pair, it suffices to note from Lemma \ref{lem:GF_L_gorro} that, for every $N \in \Le^\wedge$, the equality $\Gfd_{(\Le,\nu)}(N) = \resdim_{\Le}(N)$ holds.
\end{proof}

\begin{remark}
As a consequence of Remark \ref{rem:dim_Tor_orthogonal}, we have that the weak Gorenstein global dimension relative to $(\mathcal{L,A})$ is complete determined by the flat dimension of $\mathcal{A}$. Specifically, let $(\mathcal{L,A})$ be a Tor-orthogonal duality pair with $\mathcal{P}(R) \subseteq \mathcal{L}$, and $R$ be a ring with ${\rm l.Gwgdim}_{(\mathcal{L,A})}(R) < \infty$. Consider the following assertions:
\begin{enumerate}
\item[(a)] ${\rm l.Gwgdim}_{(\mathcal{L,A})}(R) \leq n$.

\item[(b)] $\mathrm{fd}(A) \leq n$ for every $A \in  \mathcal{A}$.

\item[(c)] $\mathrm{fd}(N) \leq n$ for every $N \in  \mathcal{A}^\vee$.
\end{enumerate}
Then, (a) $\Rightarrow$ (b) $\Rightarrow$ (c). Furthermore, if $\GF_{(\mathcal{L,A})}$ is closed under extensions, then all the assertions are equivalent, and the equality
\[
{\rm l.Gwgdim}_{(\mathcal{L,A})}(R)  = \mathrm{fd} (\mathcal{A}) = \mathrm{fd} (\mathcal{A}^\vee)
\]
holds true. In particular, the latter holds if $(\mathcal{L,A})$ is a perfect symmetric Tor-orthogonal duality pair with $\mathcal{L}$ closed under epikernels. 
\end{remark}

%%%%%%%%%%%%%%%%%%%%%%%%%%%%%%%%%%%%%%%%%%%%%%%%
%%%%%%%%%%%%%%%%%%%%%%%%%%%%%%%%%%%%%%%%%%%%%%%%
%%%%%%%%%%%%%%%%%%%%%%%%%%%%%%%%%%%%%%%%%%%%%%%%
%%%%%%%%%%%%%%%%%%%%%%%%%%%%%%%%%%%%%%%%%%%%%%%%

\section{Model structures from relative Gorenstein flat modules and dimensions}\label{sec:relative_G-flat_models}

In this section, we construct for each $n \in \mathbb{Z}_{\geq 0}$ a hereditary and cofibrantly generated abelian model structure on $\mathsf{Mod}(R)$ where the cofibrant objects are formed by the class of modules with Gorenstein $(\mathcal{L},\nu)$-flat dimension at most $n$. We assume that $(\mathcal{L,A})$ is a product closed bicomplete duality pair for the whole section. 

Remark \ref{rmk:hereditarysym}, Example \ref{ex:weakn} (3) and Corollary \ref{resolvente_dimensiones} imply that we have two hereditary perfect cotorsion pairs 
\begin{align}
(\mathcal{GF}_{(\mathcal{L},\nu)}{}^\wedge_n,(\mathcal{GF}_{(\mathcal{L},\nu)}{}^\wedge_n)^{\perp}) & & \text{and} & & (\mathcal{L}^\wedge_n,(\mathcal{L}^\wedge_n)^{\perp}). \label{eqn:2hccp}
\end{align}
Moreover, the first pair is cogenerated by a set by Remark \ref{rem:GF_cogenerated}, and a similar argument shows that $(\mathcal{L}^\wedge_n,(\mathcal{L}^\wedge_n)^{\perp})$ is also cogenerated by a set. In order to obtain the mentioned model structure, it remains to show that these pairs are compatible.

\begin{proposition}\label{prop:compatibilidadGF}
For each $n \in \mathbb{Z}_{\geq 0}$, the following equality holds:
\[
\mathcal{GF}_{(\mathcal{L},\nu)}{}^\wedge_n \cap (\mathcal{GF}_{(\mathcal{L},\nu)}{}^\wedge_n)^{\perp} = \mathcal{L}^\wedge_n \cap (\mathcal{L}^\wedge_n)^{\perp}.
\]
\end{proposition}

\begin{proof} \
\begin{itemize}
\item[($\subseteq$)] First, since $\mathcal{L} \subseteq \mathcal{GF}_{(\mathcal{L},\nu)}$, we have that $(\mathcal{GF}_{(\mathcal{L},\nu)}{}^\wedge_n)^{\perp} \subseteq (\mathcal{L}^\wedge_n)^{\perp}$. Suppose we are given $M \in \mathcal{GF}_{(\mathcal{L},\nu)}{}^\wedge_n \cap (\mathcal{GF}_{(\mathcal{L},\nu)}{}^\wedge_n)^{\perp}$. It is then clear that $M \in (\mathcal{L}^\wedge_n)^{\perp}$. On the other hand, by Proposition \ref{sequences} (2) there is a short exact sequence $M \rightarrowtail H \twoheadrightarrow G$ with $H \in \mathcal{L}^\wedge_n$ and $G \in \mathcal{GF}_{(\mathcal{L},\nu)}$. Since $M \in (\mathcal{GF}_{(\mathcal{L},\nu)}{}^\wedge_n)^{\perp}$, the previous sequence splits, and so $M$ is a direct summand of $H \in \mathcal{L}^\wedge_n$. Hence, $M \in \mathcal{L}^\wedge_n$. 

\item[$(\supseteq)$] Let $N \in \mathcal{L}^\wedge_n \cap (\mathcal{L}^\wedge_n)^{\perp}$. Then clearly we have $N \in \mathcal{GF}_{(\mathcal{L},\nu)}{}^\wedge_n$. On the other hand, since $(\mathcal{GF}_{(\mathcal{L},\nu)}{}^\wedge_n, (\mathcal{GF}_{(\mathcal{L},\nu)}{}^\wedge_n)^{\perp})$ is a complete cotorsion pair, there is a short exact sequence $N \rightarrowtail K \twoheadrightarrow C$ with $K \in (\mathcal{GF}_{(\mathcal{L},\nu)}{}^\wedge_n)^{\perp}$ and $C \in \mathcal{GF}_{(\mathcal{L},\nu)}{}^\wedge_n$. It follows that $K \in \mathcal{GF}_{(\mathcal{L},\nu)}{}^\wedge_n \cap (\mathcal{GF}_{(\mathcal{L},\nu)}{}^\wedge_n)^{\perp} \subseteq \mathcal{L}^\wedge_n \cap (\mathcal{L}^\wedge_n)^{\perp}$, from the previous part and the fact that $\mathcal{GF}_{(\mathcal{L},\nu)}{}^\wedge_n$ is closed under extensions. Moreover, we can note that $C \in \mathcal{L}^\wedge$. Indeed, let $m := \resdim_{\mathcal{L}}(K)$. If $m = 0$ there is nothing to show. If $m > 0$, we can consider a short exact sequence $K' \rightarrowtail L \twoheadrightarrow K$ with $L \in \mathcal{L}$ and $K' \in \mathcal{L}^\wedge_m \subseteq \mathcal{L}^\wedge_n$. Taking the pullback of $N \rightarrowtail K \twoheadleftarrow L$ yields the following commutative exact diagram: 
\[
\hspace{1cm}\begin{tikzpicture}[description/.style={fill=white,inner sep=2pt}]
\matrix (m) [matrix of math nodes, row sep=2.5em, column sep=2.5em, text height=1.25ex, text depth=0.25ex]
{ 
K' & K' & {} \\
N' & L & C \\ 
N & K & C \\
};
\path[->]
(m-3-1)-- node[pos=0.5] {\footnotesize$\mbox{\bf pb}$} (m-2-2)
;
\path[>->]
(m-3-1) edge (m-3-2)
(m-2-1) edge (m-2-2)
(m-1-1) edge (m-2-1)
(m-1-2) edge (m-2-2)
;
\path[->>]
(m-2-2) edge (m-2-3) (m-3-2) edge (m-3-3)
(m-2-2) edge (m-3-2) (m-2-1) edge (m-3-1)
;
\path[-,font=\scriptsize]
(m-1-1) edge [double, thick, double distance=2pt] (m-1-2)
(m-2-3) edge [double, thick, double distance=2pt] (m-3-3)
;
\end{tikzpicture}.
\]
Since $N, K' \in \mathcal{L}^\wedge_n$ and $\mathcal{L}^\wedge_n$ is closed under extensions, we have that $N' \in \mathcal{L}^\wedge_n$. It then follows that $C \in \mathcal{GF}_{(\mathcal{L},\nu)}{}^\wedge_n \cap \mathcal{L}^\wedge$. By Lemma \ref{lem:GF_L_gorro}, we get $C \in \mathcal{L}^\wedge_n$. Now since $N \in (\mathcal{L}^\wedge_n)^{\perp}$, we have that $N$ is a direct summand of $K \in (\mathcal{GF}_{(\mathcal{L},\nu)}{}^\wedge_n)^{\perp}$. Therefore, $N \in \mathcal{GF}_{(\mathcal{L},\nu)}{}^\wedge_n \cap (\mathcal{GF}_{(\mathcal{L},\nu)}{}^\wedge_n)^{\perp}$. 
\end{itemize}
\end{proof}

The following theorem describes some homotopical aspects of Gorenstein $(\mathcal{L},\nu)$-flat $R$-modules.

\begin{theorem}\label{theo:GF_model_structure}
Let $(\mathcal{L,A})$ be a product closed bicomplete duality pair. Then, for each $n \in \mathbb{Z}_{\geq 0}$, there exists a unique hereditary and cofibrantly generated abelian model structure on $\mathsf{Mod}(R)$ given by
\[
\mathfrak{M}^{{\rm GF}\mbox{-}n}_{(\mathcal{L},\nu)} := (\mathcal{GF}_{(\mathcal{L},\nu)}{}^\wedge_n,\mathcal{W},(\mathcal{L}^\wedge_n)^{\perp}).
\]
Its homotopy category ${\rm Ho}(\mathfrak{M}^{{\rm GF}\mbox{-}n}_{(\mathcal{L},\nu)})$ is triangle equivalent\footnote{In the sense of \cite{Happel}.} to the stable category 
\[
(\mathcal{GF}_{(\mathcal{L},\nu)}{}^\wedge_n \cap (\mathcal{L}^\wedge_n)^{\perp}) / \sim,
\] 
where two morphisms $f, g \colon M \to N$ between $R$-modules in $\mathcal{GF}_{(\mathcal{L},\nu)}{}^\wedge_n \cap (\mathcal{L}^\wedge_n)^{\perp}$ (a Frobenius category) are homotopic (denoted $f \sim g$) if, and only if, $f - g$ factors through an $R$-module in $\mathcal{L}^\wedge_n \cap (\mathcal{L}^\wedge_n)^\perp$. 
\end{theorem}

\begin{proof}
The pairs in \eqref{eqn:2hccp} are hereditary and complete, and satisfy $\mathcal{L}^\wedge_n \subseteq \mathcal{GF}_{(\mathcal{L},\nu)}{}^\wedge_n$, $(\mathcal{GF}_{(\mathcal{L},\nu)}{}^\wedge_n)^\perp \subseteq (\mathcal{L}^\wedge_n)^{\perp}$ and $\mathcal{GF}_{(\mathcal{L},\nu)}{}^\wedge_n \cap (\mathcal{GF}_{(\mathcal{L},\nu)}{}^\wedge_n)^\perp = \mathcal{L}^\wedge_n \cap (\mathcal{L}^\wedge_n)^\perp$ by Proposition \ref{prop:compatibilidadGF}. Hence, there is a thick class $\mathcal{W}$ of $R$-modules such that $(\mathcal{GF}_{(\mathcal{L},\nu)}{}^\wedge_n,\mathcal{W},(\mathcal{L}^\wedge_n)^\perp)$ is a Hovey triple. The existence of the model structure then follows by \cite[Thm. 2.2, Lem. 6.7 \& Coroll. 6.8]{HoveyModel}. The last assertion is an application of \cite[Prop. 4.2 \& Thm. 4.3]{GillespieHereditary}.
\end{proof}

%%%%%%%%%%%%%%%%%%%%%%%%%%%%%%%%%%%%%
%%%%%%%%%%%%%%%%%%%%%%%%%%%%%%%%%%%%%
%%%%%%%%%%%%%%%%%%%%%%%%%%%%%%%%%%%%%
%%%%%%%%%%%%%%%%%%%%%%%%%%%%%%%%%%%%%

\section{Relative Gorenstein injective and flat complexes}\label{appx:complexes}

We study the chain complex version of Gorenstein $(\mathcal{L},\nu)$-flat $R$-modules, dimensions and model structures. In what follows, Hom functors in $\Ch(R)$ will be denoted by $\Hom_{\Ch(R)}(-,-)$, and its right derived functors by $\Ext^i_{\Ch(R)}(-,-)$. It is also important that the reader is familiar with the \emph{modified} tensor product and the internal Hom functors (see for instance \cite[\S 4.2]{JRGR} or \cite[pp. 70--71]{libro})
\begin{align*}
- \overline{\otimes} - \colon \Ch(R^{\rm o}) \times \Ch(R) & \to \Ch(\mathbb{Z}) & & \text{and} & \overline{\mathcal{H}{\rm om}}(-,-) \colon \Ch(R) \times \Ch(R) & \to \Ch(\mathbb{Z}). 
\end{align*}
The left derived functors of $(- \overline{\otimes} -)$ will be denoted by $\overline{\Tor}^R_i(-,-)$, while the right derived functors of $\overline{\mathcal{H}{\rm om}}(-,-)$ will be denoted by $\overline{\mathcal{E}{\rm xt}}(-,-)$, for $i \in \mathbb{Z}_{>0}$. 

Pontryagin duality between $\Ch(R)$ and $\Ch(R^{\rm o})$ is defined in terms of $\overline{\mathcal{H}{\rm om}}(-,-)$. Indeed, recall from \cite[Def. 4.4.9 \& Prop. 4.4.10]{libro} that the Pontryagin dual of $X_\bullet$ is the $R^{\rm o}$-complex defined by
\[
X^+ := \overline{\mathcal{H}{\rm om}}(X_\bullet,D^0(\mathbb{Q / Z})).
\]
It can also be defined as the complex
\[
\cdots \to X_{-m-1}^+ \xrightarrow{(-1)^{m-1}\Hom_{\mathbb{Z}}(\partial^{X_\bullet}_{-m})} X^+_{-m} \to \cdots.
\]
Bicomplete, complete, symmetric, perfect and (co)product closed duality pairs of chain complexes are defined similarly as for modules. We only point out that, for perfect duality pairs $(\mathscr{L,A})$ of complexes, one requires that the \emph{disk} complex $\cdots \to 0 \xrightarrow{0} R \xrightarrow{{\rm id}} R \xrightarrow{0} 0 \to \cdots$ belongs to $\mathscr{L}$. 

Now we define the chain complex analogs for $\mathcal{GF}_{(\mathcal{L},\nu)}$ and $\mathcal{GI}_{(\nu,\mathcal{A})}$. In what follows, 
\[
\mathscr{L} \subseteq \Ch(R) \text{ \ and \ } \mathscr{A} \subseteq \Ch(R^{\rm o})
\]
are fixed classes of $R$-complexes and $R^{\rm o}$-complexes, respectively.

\begin{definition}
We say that an $R$-complex $X_\bullet \in \Ch(R)$ is \textbf{Gorenstein $\bm{(\mathscr{L,A})}$-flat} if there is an exact and $(\mathscr{A} \overline{\otimes} -)$-acyclic complex $\mathbb{L}_\square$ in $\Ch(\mathscr{L})$\footnote{Note that with this notation we mean a complex of complexes, in this case a complex whose components are complexes in $\mathscr{L}$.} such that $X_\bullet \simeq Z_0(\mathbb{L}_\square)$.
\end{definition}

The class of Gorenstein $(\mathscr{L,A})$-flat $R$-complexes will be denoted by $\mathcal{GF}_{(\mathscr{L,A})}$.

\begin{remark}\label{rem:validity_for_complexes}
All of the results in Sections \ref{sec:prelim}, \ref{sec:relative_G-flat}, \ref{sec:relative_G-flat_dimensions} and \ref{sec:relative_G-flat_models} are also valid in $\Ch(R)$ for the class $\mathcal{GF}_{(\mathscr{L,A})}$, as the arguments presented for modules carry over to chain complexes. We only mention some particular details where one needs to be careful:
\begin{itemize} 
\item Proof of Lemma \ref{lem:ort}: One has to cite \cite[Prop. 4.4.13 (2)]{libro} instead of \cite[Lem. 1.2.11 (b)]{GT} and \cite[Thm. 3.2.1]{EJ00}. 

\item Proof of Theorem \ref{dualidad}: One needs \cite[Prop. 4.3.3]{libro} for the adjunctions between $\overline{\otimes}$ and $\overline{\mathcal{H}{\rm om}}$.

\item Proof of Proposition \ref{prop:GF_closed_under_exts} (b) $\Rightarrow$ (c): One can check that \cite[Thm. 2.3]{GFclosed} carries over to chain complexes.   

\item Proof of Proposition \ref{sequences}: One can check that \cite[Lem. 2.2]{Bennis10} carries over to chain complexes. 
\end{itemize}
\end{remark}

%%%%%%%%%%%%%%%%%%%%%%%%%%%%%%%%%%%%%%%%%%%%%
%%%%%%%%%%%%%%%%%%%%%%%%%%%%%%%%%%%%%%%%%%%%%

\subsection*{Relations with relative Gorenstein flat modules}

Before studying some relations between relative Gorenstein flat modules and their chain complex counterpart, let us recall first how to construct certain chain complexes from modules and classes of modules. 

Given a class $\mathcal{X} \subseteq \mathsf{Mod}(R)$, we shall consider the following induced classes in $\Ch(R)$, with the terminology and notation borrowed from \cite{GillespieCh}:
\begin{itemize}
\item \emph{$\mathcal{X}$-complexes:} exact complexes $X_\bullet \in \Ch(R)$ such that $Z_m(X_\bullet) \in \mathcal{X}$ for every $m \in \mathbb{Z}$. The class of $\mathcal{X}$-complexes is denoted by $\widetilde{\mathcal{X}}$. 
\end{itemize}
We also need to consider the following special complexes:
\begin{itemize}
\item \emph{Disk complexes:} Given an $R$-module $M \in \mathsf{Mod}(R)$, the $m$-th disk centered at $M$ is defined as the complex $D^m(M)$ with $M$ at degrees $m$ and $m-1$, and zero elsewhere, such that the only nonzero differential $(D^m(M))_m \to (D^m(M))_{m-1}$ is the identity. 

\item \emph{Shifted complexes:} Given an $R$-complex $X_\bullet \in \Ch(R)$, its $n$-th suspension or shift is defined as the complex $X_\bullet[n]$ with components 
\[
(X_\bullet[n])_m := X_{m-n}
\]
for every $m \in \mathbb{Z}$, and whose differentials are given by
\[
\partial^{X_\bullet[n]}_m := (-1)^n \partial^{X_\bullet}_m.
\]
\end{itemize}

\begin{proposition}\label{Thm2}
The following assertions hold true for every $R^{\rm o}$-complex $Y_\bullet$:
\begin{enumerate}
\item If $Y_{\bullet} \in \mathcal{GI}_{(\widetilde{\nu},\widetilde{\A})}$, then $Y_m \in \GI_{(\nu,\A)}$ for every $m \in \mathbb{Z}$, provided that $\mathcal{A}$ is closed under extensions. 

\item If in addition $(\nu,\mathcal{A})$ is a GI-admissible pair with $\mathcal{A}$ closed under direct summands and monocokernels, and containing the injective $R^{\rm o}$-modules, then $Y_m \in \GI_{(\nu,\A)}$ for every $m \in \mathbb{Z}$ implies that $Y_\bullet \in \mathcal{GI}_{(\widetilde{\nu},\widetilde{\A})}$. In other words,
\[
\mathcal{GI}_{(\widetilde{\nu},\widetilde{\A})} = \Ch(\mathcal{GI}_{(\nu,\mathcal{A})}).
\]
\end{enumerate}
\end{proposition}

\begin{proof} \
\begin{enumerate}
\item By definition of $\mathcal{GI}_{(\widetilde{\nu},\widetilde{\mathcal{A}})}$ there is an exact and $\Hom_{\Ch(R^{\rm o})}(\widetilde{\nu},-)$-acyclic sequence of complexes in $\widetilde{\mathcal{A}}$, say 
\[
\mathbb{A}_\square = \cdots \to A_{\bullet}^1 \to A_{\bullet}^0 \to A_{\bullet}^{-1} \to A_{\bullet}^{-2} \to \cdots,
\] 
such that $Y_\bullet \simeq \Ker(A_{\bullet}^{0} \to A_{\bullet}^{-1})$. In order to show that $Y_m \in \mathcal{GI}_{(\nu,\mathcal{A})}$, let us note that for $A \in \nu$, one has $D^m(A) \in \widetilde{\nu}$. Moreover, by \cite[Lem. 3.1 (1)]{GillespieCh} we have that the complexes $\Hom_{\Ch(R^{\rm o})}(D^m(A),\mathbb{A}_\square)$ and $\Hom_{R^{\rm o}}(A,\mathbb{A}_m)$ are naturally isomorphic, where 
\[
\mathbb{A}_m = \cdots \to A_m^1 \to A_m^0 \to A_m^{-1} \to A_m^{-2} \to \cdots
\] 
is exact in $\Ch(\mathcal{A})$. Moreover, the exactness of $\Hom_{\Ch(R^{\rm o})}(D^m(A),\mathbb{A}_\square)$ implies the exactness of $\Hom_{R^{\rm o}}(A,\mathbb{A}_m)$. It follows that $\mathbb{A}_m$ is an exact and $\Hom_{R^{\rm o}}(\nu,-)$-acyclic complex in $\Ch(\mathcal{A})$ such that $Z_0(\mathbb{A}_m) \simeq Y_m$. In other words, $Y_m \in \mathcal{GI}_{(\nu,\mathcal{A})}$.

\item First, since $\nu$ is a relative generator for $\GI_{(\nu, \A)}$ with $\pd_{\GI_{(\nu, \A)}(R^{\rm o})}(\nu) = 0$ and $\mathcal{A}$ closed under direct summands, we have by the dual of \cite[Prop. 4.3 (2)]{BMP} that $\widetilde{\nu}$ is a relative generator in $\Ch(\GI_{(\nu,\A)})$ satisfying 
\begin{align}
\pd_{\Ch(\GI_{(\nu,\A)}(R^{\rm o}))}(\widetilde{\nu}) & = 0. \label{pdChGInu}
\end{align} 
Then, for $Y_{\bullet} \in \Ch(\GI _{(\nu, \A)})$ there is a short exact sequence $K_{\bullet}^0 \rightarrowtail A_{\bullet}^0 \twoheadrightarrow Y_{\bullet}$ with $A_{\bullet}^0 \in \widetilde{\nu} \subseteq \widetilde{\mathcal{A}}$ and $K_{\bullet }^0 \in \Ch(\GI _{(\nu, \A)})$. By \eqref{pdChGInu}, we can note that the previous sequence is $\Hom_{\Ch(R^{\rm o})}(\widetilde{\nu},-)$-acyclic. By repeating this procedure, we can obtain a $\Hom_{\Ch(R^{\rm o})}(\widetilde{\nu},-)$-acyclic $\widetilde{\mathcal{A}}$-resolution of $Y_\bullet$. Finally, since $\widetilde{\mathcal{I}(R^{\rm o})}$ is precisely the class of injective $R^{\rm o}$-complexes, $\mathcal{I}(R^{\rm o}) \subseteq \mathcal{A}$, we have that $\widetilde{\mathcal{A}}$ contains the class of injective $R^{\rm o}$-complexes. Thus, we can complete the previous resolution to a $\Hom_{\Ch(R^{\rm o})}(\widetilde{\nu},-)$-acyclic complex in $\Ch(\widetilde{\mathcal{A}})$ whose $0$-th cycle is isomorphic to $Y_\bullet$. 
\end{enumerate}   
\end{proof}

Theorem \ref{dualidad_complejos} characterizes the Gorenstein $(\mathcal{L},\nu)$-flat complexes and extends the duality relation proved in Theorem \ref{dualidad} to the chain complex level.

\begin{theorem}\label{dualidad_complejos}
Let $(\Le, \A)$ be a product closed bicomplete duality pair in $\mathsf{Mod}(R)$. The following statements are equivalent:
\begin{enumerate}
\item[(a)] $X_{\bullet} \in \mathcal{GF}_{(\widetilde{\Le},\widetilde{\nu})}$.

\item[(b)] $X_m \in \GF_{(\Le,\nu)}$ for every $m \in \mathbb{Z}$.

\item[(c)] $X_{\bullet}^+ \in \mathcal{GI}_{(\widetilde{\nu},\widetilde{\A})}$.

\item[(d)] $X_{\bullet}$ admits a $\Hom_{\Ch(R)}(-,\widetilde{\mathcal{L}})$-acyclic $\widetilde{\mathcal{L}}$-coresolution.
\end{enumerate}
In particular, 
\[
\mathcal{GF}_{(\widetilde{\mathcal{L}},\widetilde{\nu})} = \Ch(\mathcal{GF}_{(\mathcal{L},\nu)}).
\]
\end{theorem}

\begin{proof} \
\begin{itemize}
\item (a) $\Rightarrow$ (b): Follows by properties of the modified tensor product $- \overline{\otimes} -$, concretely \cite[Prop. 4.2.1 (4)]{JRGR}.

\item (b) $\Rightarrow$ (c): Follows by Theorem \ref{dualidad} and Proposition \ref{Thm2}.

\item (c) $\Rightarrow$ (d): The idea is to replicate the arguments from (b) $\Rightarrow$ (c) in Theorem \ref{dualidad} within the ambient of chain complexes. First of all, note it is clear that $X^{+}_\bullet[n] \in \mathcal{GI}_{(\widetilde{\nu},\widetilde{\A})}$ for every $n \in \mathbb{Z}$. Now for every $A_\bullet \in \widetilde{\nu}$, we have that $\Ext_{\Ch(R^{\rm o})}^{i}(A_\bullet, X^+_\bullet[n]) =0$, and thus by \cite[Props. 4.4.7 \& 4.4.13 (2)]{libro} this yields $\overline{\Tor}_i(A_\bullet, X_{\bullet})^{+} \cong \overline{\mathcal{E}{\rm xt}}^{i} (A_\bullet, X_{\bullet}^+) = 0$ for every $i \in \mathbb{Z}_{> 0}$. In other words, $X_\bullet \in \widetilde{\nu}^\top$. 

Another consequence of having $X^+_\bullet \in \mathcal{GI}_{(\widetilde{\nu},\widetilde{\A})}$ is the existence of an epimorphism $A^0_\bullet \twoheadrightarrow X^+_\bullet$, with $A^0_\bullet \in \widetilde{\mathcal{A}}$. This induces a monomorphism $X^{++}_\bullet \rightarrowtail (A^0_\bullet)^+$, where $(A^0_\bullet)^+ \in \widetilde{\mathcal{L}}$ since $\mathcal{A}^+ \subseteq \mathcal{L}$ and the functor $(-)^+$ is exact. On the other hand, since $(\mathcal{L,A})$ is a product closed duality pair, we have by \cite[\S \ 4.2 \& Thm. 4.2.1]{TZhao}\footnote{The citation corresponds to the arxived version of the article (\texttt{arXiv:1703.10703}), which does not appear in the published version.} that $(\widetilde{\mathcal{L}},\widetilde{\mathcal{A}})$ is a product closed duality pair of complexes. By \cite[Thm. 3.2]{YangProceedings}, the class $\widetilde{\mathcal{L}}$ is preenveloping. We can then consider the following diagram
\[
\hspace{1.5cm}\begin{tikzpicture}[description/.style={fill=white,inner sep=2pt}]
\matrix (m) [matrix of math nodes, row sep=2.5em, column sep=2.5em, text height=1.25ex, text depth=0.25ex]
{ 
X_\bullet & X_{\bullet}^{++} & (A^0_\bullet)^+ \\
L^0_{\bullet} & {} & {} \\
};
\path[->]
(m-1-1) edge node[left] {\footnotesize$\phi^0$} (m-2-1)
;
\path[>->]
(m-1-1) edge (m-1-2) (m-1-2) edge (m-1-3)
;
\path[dotted,->]
(m-2-1) edge (m-1-3)
;
\end{tikzpicture}
\]
where $\phi^0$ is a $\widetilde{\mathcal{L}}$-preenvelope of $X_\bullet$. The dotted arrow exists since $(A^0_\bullet)^+ \in \widetilde{\mathcal{L}}$. It follows that $\phi^0$ is a monomorphism, and so we have an exact sequence $\eta_\bullet \colon X_{\bullet} \rightarrowtail  L_{\bullet}^0 \twoheadrightarrow X_\bullet^1$ with $X_\bullet^1 :=  \Coker(\phi^0)$, is $\Hom_{\Ch(R)}(-,\widetilde{\mathcal{L}})$-acyclic. 

Now we show that $(X_\bullet^1)^+ \in \mathcal{GI}_{(\widetilde{\nu},\widetilde{\A})}$. Let $A_\bullet \in \widetilde{\nu}$. Again by \cite[Thm. 4.2.1]{TZhao}, we have that $A^+_\bullet \in \widetilde{\mathcal{L}}$, and so $A^+_\bullet[n] \in \widetilde{\mathcal{L}}$ for every $n \in \mathbb{Z}$. Now applying the functor $\Hom_{\Ch(R)}(-, A_{\bullet}^+[n])$ to $\eta_\bullet$ yields the exact sequence $\Hom_{\Ch(R)}(\eta_\bullet, A_{\bullet}^+[n])$. By \cite[Props. 4.4.7 \& 4.4.11]{libro}, the complex $\overline{\mathcal{H}{\rm om}}(\eta_\bullet, A_{\bullet}^+)$ is exact and naturally isomorphic to $(A_\bullet \overline{\otimes} \eta_\bullet)^+$. It follows that $A_\bullet \overline{\otimes} \eta_\bullet$ is exact for every $A_\bullet \in \widetilde{\nu}$. In particular, we have the exact sequence
\[
\hspace{1cm}(A_{\bullet} \overline{\otimes} L^0_{\bullet})^+ \twoheadrightarrow (A_{\bullet} \overline{\otimes} X_{\bullet})^+ \to \overline{\Tor}_1(A_{\bullet}, X_{\bullet}^1)^+ \to \overline{\Tor}^R_1(A_{\bullet}, L^0_{\bullet})^{+}
\]
is exact, where $\overline{\Tor}_1(A_\bullet, L^0_\bullet)^+ \cong \overline{\mathcal{E}{\rm xt}}^1 (A_{\bullet}, (L^0_{\bullet})^+) = 0$ since $A_{\bullet} \in \widetilde{\nu}$ and $(L^0_{\bullet})^+ \in \widetilde{\mathcal{A}} \subseteq \mathcal{GI}_{(\widetilde{\nu},\widetilde{\A})}$. By the exactness of the previous sequence we get $\overline{\Tor}_1(A_{\bullet}, X^1_{\bullet})^{+} = 0$, and hence $\overline{\mathcal{E}{\rm xt}}^1(A_{\bullet}, (X^1_{\bullet})^+) = 0$. The latter implies $\Ext^1_{\Ch(R^{\rm o})}(A_{\bullet}, (X^1_{\bullet})^+) = 0$ by \cite[Prop. 4.4.7]{libro}. Degreewise, we then obtain $\Ext^1_{R^{\rm o}}(\nu, (X^1_m)^+) = 0$ for every $m \in \mathbb{Z}$. We also have the short exact sequence $(X^1_m)^+ \rightarrowtail  (L^0_m)^+ \twoheadrightarrow X^+_m$, where $(L^0_m)^+, X^+_m \in \mathcal{GI}_{(\nu,\mathcal{A})}$ by Proposition \ref{Thm2}. Therefore, from Proposition \ref{criterio} we obtain that $(X^1_m)^+ \in \mathcal{GI}_{(\nu,\mathcal{A})}$, and again Proposition \ref{Thm2} allows us to conclude that $(X^1_\bullet)^+ \in \mathcal{GI}_{(\widetilde{\nu},\widetilde{\mathcal{A}})}$. 

Repeating the previous procedure infinitely many times gives rise to a $\Hom_{\Ch(R)}(-,\mathcal{L})$-acyclic $\widetilde{\mathcal{L}}$-coresolution of $X_\bullet$. 
  
\item (d) $\Rightarrow$ (a): It suffices to note that $\mathcal{P}(R) \subseteq \mathcal{L}$, and thus that $\widetilde{\mathcal{L}}$ contains the class $\widetilde{\mathcal{P}(R)}$ of projective $R$-complexes. 
\end{itemize}
\end{proof}

%%%%%%%%%%%%%%%%%%%%%%%%%%%%%%%%%%%%%%%%%%%%%
%%%%%%%%%%%%%%%%%%%%%%%%%%%%%%%%%%%%%%%%%%%%%

\subsection*{Several model structures on chain complexes from product closed bicomplete duality pairs of modules}

As mentioned in Remark \ref{rem:validity_for_complexes}, the results in Section \ref{sec:relative_G-flat_models} are valid in the category of chain complexes, and so in particular Theorem \ref{theo:GF_model_structure}. Below we write the corresponding statement for further reference.

In what follows, if $(\mathscr{L,A})$ is a product closed bicomplete duality pair of complexes, we use the notation
\[
\mathscr{V} := {}^\perp\mathscr{A} \cap \mathscr{A}.
\]

\begin{theorem}\label{theo:GF_model_structure_Ch}
Let $(\mathscr{L,A})$ be a product closed bicomplete duality pair. Then, for each $n \in \mathbb{Z}_{\geq 0}$, there exists a unique hereditary and cofibrantly generated abelian model structure on $\Ch(R)$ given by
\[
\mathfrak{M}^{{\rm GF}\mbox{-}n}_{(\mathscr{L},\mathscr{V})} := (\mathcal{GF}_{(\mathscr{L},\mathscr{V})}{}^\wedge_n,\mathscr{W},(\mathscr{L}^\wedge_n)^{\perp}).
\]
Its homotopy category ${\rm Ho}(\mathfrak{M}^{{\rm GF}\mbox{-}n}_{(\mathscr{L},\mathscr{V})})$ is triangle equivalent to the stable category 
\[
(\mathcal{GF}_{(\mathscr{L,V})}{}^\wedge_n \cap (\mathscr{L}^\wedge_n)^{\perp}) / \sim,
\] 
where $f \sim g$ if, and only if, $f - g$ factors through an $R$-complex in $\mathscr{L}^\wedge_n \cap (\mathscr{L}^\wedge_n)^\perp$. 
\end{theorem}

For the rest of this section, we shall focus on the case $n = 0$ for a better understanding and display of the results below (although similar statements will be also valid for $n > 0$), and assume that we are given a product closed bicomplete duality pair $(\mathcal{L,A})$ of classes of modules $\mathcal{L} \subseteq \mathsf{Mod}(R)$ and $\mathcal{A} \subseteq \mathsf{Mod}(R^{\rm o})$. At this point, from the previous theorem and other results of Gillespie \cite{GillespieCh,GillespieDW,GillespieHowTo,GillespieHereditary}, Enochs and Jenda \cite{EJ002}, Yang and Ding \cite[Thm. 2.4]{onaquestion}, and Yang and Liu \cite[Thm. 3.5]{YLmodelsCh}), we are able to construct six different model category structures on $\Ch(R)$ from the class $\mathcal{GF}_{(\mathcal{L},\nu)}$. Most of these are models for the derived category $D(R)$ of the ground ring $R$, or for certain derived categories of subcategories of $\Ch(R)$. 

Let us recall how certain cotorsion pairs in $\Ch(R)$ are induced by a cotorsion pair $(\mathcal{X,Y})$ in $\mathsf{Mod}(R)$. 
\begin{itemize}
\item \cite[Prop. 3.6]{GillespieCh}: $({\rm dg} \ \widetilde{\mathcal{X}}, \widetilde{\mathcal{Y}})$ and $(\widetilde{\mathcal{X}}, {\rm dg} \ \widetilde{\mathcal{Y}})$ are cotorsion pairs in $\Ch(R)$, where ${\rm dg} \ \widetilde{\mathcal{X}}$ is the class of $R$-complexes $X_\bullet \in \Ch(\mathcal{X})$ such that every chain map $X_\bullet \to Y_\bullet$ is null homotopic whenever $Y_\bullet \in \widetilde{\mathcal{Y}}$, and ${\rm dg} \ \widetilde{\mathcal{Y}}$ is defined dually (see \cite[Def. 3.3]{GillespieCh}). Moreover, if $(\mathcal{X,Y})$ is hereditary, then these pairs are compatible by \cite[Thm. 3.12]{GillespieCh}, and
\[
\widetilde{\mathcal{X}} = {\rm dg} \ \widetilde{\mathcal{X}} \cap \mathscr{E} \text{ \ and \ } \widetilde{\mathcal{Y}} = {\rm dg} \ \widetilde{\mathcal{Y}} \cap \mathscr{E},
\]
where $\mathscr{E}$ denotes the class of exact $R$-complexes. If $(\mathcal{X,Y})$ is complete (resp., cogenerated by a set), then so are these pairs by \cite[Thm. 2.4]{onaquestion} (resp., \cite[Thm. 3.5]{YLmodelsCh}). 

\item \cite[Prop. 3.2]{GillespieDW}: $(\Ch(\mathcal{X}),(\Ch(\mathcal{X}))^{\perp_1})$ and $({}^{\perp_1}(\Ch(\mathcal{Y})),\Ch(\mathcal{Y}))$ are cotorsion pairs in $\Ch(R)$, where $(\Ch(\mathcal{X}))^{\perp_1}$ coincides with the class of complexes $Y_\bullet \in \Ch(\mathcal{Y})$ such that every chain map $X_\bullet \to Y_\bullet$ is null homotopic whenever $X_\bullet \in \Ch(\mathcal{X})$, and ${}^{\perp_1}(\Ch(\mathcal{Y}))$ has a dual description. If $(\mathcal{X,Y})$ is hereditary, one can easily note that so are these pairs. If $(\mathcal{X,Y})$ is cogenerated by a set, then so are $(\Ch(\mathcal{X}),(\Ch(\mathcal{X}))^{\perp_1})$ and $({}^{\perp_1}(\Ch(\mathcal{Y})),\Ch(\mathcal{Y}))$, by \cite[Thm. 7.2.14]{EJ002} and \cite[Prop. 4.4]{GillespieDW}, respectively. 

\item \cite[Prop. 3.3]{GillespieDW}: $(\Ch(\mathcal{X}) \cap \mathscr{E},(\Ch(\mathcal{X}) \cap \mathscr{E})^{\perp_1})$ and $({}^{\perp_1}(\Ch(\mathcal{Y}) \cap \mathscr{E}),\Ch(\mathcal{Y}) \cap \mathscr{E})$ are cotorsion pairs in $\Ch(R)$, where $(\Ch(\mathcal{X}) \cap \mathscr{E})^{\perp_1}$ coincides with the class of complexes $Y_\bullet \in \Ch(\mathcal{Y})$ such that every chain map $X_\bullet \to Y_\bullet$ is null homotopic whenever $X_\bullet \in \Ch(\mathcal{X}) \cap \mathscr{E}$, and ${}^{\perp_1}(\Ch(\mathcal{Y}) \cap \mathscr{E})$ has a dual description. If $(\mathcal{X,Y})$ is hereditary, one can easily note that so are these pairs. If $(\mathcal{X,Y})$ is cogenerated by a set, then so are $(\Ch(\mathcal{X}) \cap \mathscr{E},(\Ch(\mathcal{X}) \cap \mathscr{E})^{\perp_1})$ and $({}^{\perp_1}(\Ch(\mathcal{Y}) \cap \mathscr{E}),\Ch(\mathcal{Y}) \cap \mathscr{E})$, by \cite[Thm. 7.2.15]{EJ002} and \cite[Prop. 4.6]{GillespieDW}, respectively. 
\end{itemize}
It then follows that if $(\mathcal{X,Y})$ is a hereditary cotorsion pair cogenerated by a set, we have the following hereditary and cofibrantly generated abelian model category structure on $\Ch(R)$:
\begin{align*}
\text{differential graded models: \ } & \mathfrak{M}_{\rm dg} = ({\rm dg} \ \widetilde{\mathcal{X}}, \mathscr{E}, {\rm dg} \ \widetilde{\mathcal{Y}}).
\end{align*}
Since the exact complexes form the class of trivial objects, its homotopy category is triangle equivalent to the derived category $D(R)$. See for instance \cite[Thm. 5.3]{GillespieHereditary}. We can also note the following relations, whose proofs are straightforward.

\begin{proposition}\label{prop:cores_dgdwex_models}
Let $(\mathcal{X,Y})$ be a hereditary cotorsion pair cogenerated by a set. Then, the induced cotorsion pairs 
\[
({\rm dg} \ \widetilde{\mathcal{X}}, \widetilde{\mathcal{Y}}), \ (\widetilde{\mathcal{X}}, {\rm dg} \ \widetilde{\mathcal{Y}}), \ (\Ch(\mathcal{X}),(\Ch(\mathcal{X}))^{\perp}) \ \text{and} \ (\Ch(\mathcal{X}) \cap \mathscr{E},(\Ch(\mathcal{X}) \cap \mathscr{E})^{\perp})
\]
have the same cores. Moreover, if we denote these cores by ${\rm core}_{\rm dg}$, ${\rm core}_{\rm dw}$ and ${\rm core}_{\rm ex}$, respectively, they coincide with the class
\[
\widetilde{\mathcal{X} \cap \mathcal{Y}} = \{ \text{contractible $R$-complexes with components in $\mathcal{X} \cap \mathcal{Y}$} \}.
\] 
\end{proposition}

From the previous proposition and comments, along with Corollary \ref{resolvente_dimensiones}, Remark \ref{rem:GF_cogenerated} and Propositions \ref{prop:compatibilidadGF}, we can note the following.

\begin{corollary}\label{coro:derived}
There exist the following hereditary and cofibrantly generated abelian model structures on $\Ch(R)$,
\begin{align*}
\mathfrak{M}^{\rm GF}_{\rm dg} & = ({\rm dg} \ \widetilde{\mathcal{GF}_{(\mathcal{L},\nu)}}, \mathscr{E}, {\rm dg} \ \widetilde{\mathcal{GC}_{(\mathcal{L},\nu)}}),
\end{align*}
whose homotopy category is triangle equivalent to the derived category $D(R)$ of the ground ring $R$. Its core coincides with the class of contractible complexes with components in $\mathcal{L} \cap \mathcal{L}^\perp$.
\end{corollary}

Another consequence of Proposition \ref{prop:cores_dgdwex_models} is that the pairs $(\Ch(\mathcal{X}),(\Ch(\mathcal{X}))^\perp)$ and $(\widetilde{\mathcal{X}},{\rm dg} \ \widetilde{\mathcal{Y}})$ are compatible. Then by Proposition \ref{prop:HoveyTriple}, we have that there exists a model structure on $\Ch(R)$ given by
\[
\mathfrak{M}_{\rm dw} := (\Ch(\mathcal{X}),\mathscr{W}_{\rm dw},{\rm dg} \ \widetilde{\mathcal{Y}}).
\]
Similarly, since $(\Ch(\mathcal{X}) \cap \mathscr{E},(\Ch(\mathcal{X}) \cap \mathscr{E})^\perp)$ and $(\widetilde{\mathcal{X}},{\rm dg} \ \widetilde{\mathcal{Y}})$ are compatible, there is a model structure on $\Ch(R)$ given by
\[
\mathfrak{M}_{\rm ex} := (\Ch(\mathcal{X}) \cap \mathscr{E},\mathscr{W}_{\rm ex},{\rm dg} \ \widetilde{\mathcal{Y}}).
\]
Here, 
\begin{align*}
\mathscr{W}_{\rm dw} & = \{ W_\bullet \in \Ch(R) {\rm \ : \ } \exists \text{ a s.e.s } Y_\bullet \rightarrowtail X_\bullet \twoheadrightarrow W_\bullet \text{ with } X_\bullet \in \widetilde{\mathcal{X}} \text{ and } Y_\bullet \in (\Ch(\mathcal{X}))^\perp \} \\
& = \{ W_\bullet \in \Ch(R) {\rm \ : \ } \exists \text{ a s.e.s } W_\bullet \rightarrowtail Y'_\bullet \twoheadrightarrow X'_\bullet \text{ with } X'_\bullet \in \widetilde{\mathcal{X}} \text{ and } Y'_\bullet \in (\Ch(\mathcal{X}))^\perp \}, \\
\mathscr{W}_{\rm ex} & = \{ W_\bullet \in \Ch(R) {\rm \ : \ } \exists \text{ a s.e.s } Y_\bullet \rightarrowtail X_\bullet \twoheadrightarrow W_\bullet \text{ with } X_\bullet \in \widetilde{\mathcal{X}} \text{ and } Y_\bullet \in (\Ch(\mathcal{X}) \cap \mathscr{E})^\perp \} \\
& = \{ W_\bullet \in \Ch(R) {\rm \ : \ } \exists \text{ a s.e.s } W_\bullet \rightarrowtail Y'_\bullet \twoheadrightarrow X'_\bullet \text{ with } X'_\bullet \in \widetilde{\mathcal{X}} \text{ and } Y'_\bullet \in (\Ch(\mathcal{X}) \cap \mathscr{E})^\perp \},
\end{align*}
and these two models have the same core, namely, the class of contractible $R$-complexes with components in $\mathcal{X} \cap \mathcal{Y}$ (again, by Proposition \ref{prop:cores_dgdwex_models}). Concerning homotopy categories, we have the following.

\begin{proposition}\label{prop:homotopy_models_KKac}
The homotopy categories ${\rm Ho}(\mathfrak{M}_{\rm dw})$ and ${\rm Ho}(\mathfrak{M}_{\rm ex})$ are triangle equivalent to the derived categories 
\[
\frac{K(\mathcal{X})}{\widetilde{\mathcal{X}}} \text{ \ and \ } \frac{K_{\rm ac}(\mathcal{X})}{\widetilde{\mathcal{X}}}. 
\]
Here, $K(\mathcal{X})$ (resp., $K_{\rm ac}(\mathcal{X})$) is the full subcategory of the homotopy category $K(R)$ consisting of all (exact) $R$-complexes in $\Ch(\mathcal{X})$.
\end{proposition}

\begin{proof}
We only give a proof for the statement concerning $K(\mathcal{X}) / \widetilde{\mathcal{X}}$. The one about $K_{\rm ac}(\mathcal{X}) / \widetilde{\mathcal{X}}$ follows in a similar fashion. 

We follow some of the arguments in \cite[Prop. 4.2]{GillespieHereditary}. By \cite[Prop. 5.2 \& Coroll. 5.4]{GillespieModelsonexact}, ${\rm Ho}(\mathfrak{M}_{\rm dw})$ is triangle equivalent to the homotopy category of the submodel structure of $\mathfrak{M}_{\rm dw}$ restricted to its cofibrant objects, that is, the triple 
\[
\mathfrak{M}_{\rm dw\mbox{-}cof} := (\Ch(\mathcal{X}),\widetilde{\mathcal{X}}, \Ch(\mathcal{X}) \cap {\rm dg} \ \widetilde{\mathcal{Y}})
\] 
on the subcategory $\Ch(\mathcal{X})$. The homotopy category of the latter triple is in turn triangle equivalent to the Verdier quotient of $K(\mathcal{X})$ by the localizing subcategory $\widetilde{X}$. In other words, the homotopy category $\mathfrak{M}_{\rm dw}$ is obtained by ``killing'' in $K(\mathcal{X})$ the (trivial) objects in $\widetilde{\mathcal{X}}$. %Let us explain this last assertion in more detail to finish the proof. Let $A_\bullet$ be an object in $\widetilde{\mathcal{A}}$. Then, the canonical map $0 \to A_\bullet$ is a trivial cofibration. If we consider the universal functor $\gamma \colon \Ch(\mathcal{A}) \to {\rm Ho}(\mathfrak{M}_{\rm dw\mbox{-}cof})$, then $\gamma(0 \to A_\bullet)$ is an isomorphism in ${\rm Ho}(\mathfrak{M}_{\rm dw\mbox{-}cof})$, that is, $A_\bullet$ is isomorphic to the zero complex in the homotopy category. Now let $X_\bullet$ be an $R$-complex isomorphic to zero in ${\rm Ho}(\mathfrak{M}_{\rm dw\mbox{-}cof})$. If we consider the canonical map $X_\bullet \to 0$, we have that $\gamma(X_\bullet \to 0)$ is an isomorphism in ${\rm Ho}(\mathfrak{M}_{\rm dw\mbox{-}cof})$. By the Fundamental Theorem of Model Categories (see for instance \cite[Thm. 1.2.10 (iv)]{HoveyBook}), we have that $X_\bullet \to 0$ is a weak equivalence in $\Ch(\mathcal{A})$. Then, we can write $X_\bullet \to 0 = (W_\bullet \to 0) \circ f$, where $W_\bullet \to 0$ is a trivial fibration (and so $W_\bullet$ is contractible with components in $\mathcal{A} \cap \mathcal{B}$) and $f \colon X_\bullet \to W_\bullet$ is a trivial cofibration (and so it is a monomorphism with cokernel in $\widetilde{\mathcal{A}}$). Thus, we have a short exact sequence $X_\bullet \rightarrowtail W_\bullet \twoheadrightarrow {\rm CoKer}(f)$. Since $W_\bullet$ and ${\rm CoKer}(f)$ are exact, so is $X_\bullet$. It follows that $Z_m(X_\bullet) \rightarrowtail Z_m(W_\bullet) \twoheadrightarrow Z_m({\rm CoKer}(f))$ is an exact sequence of $R$-modules for every $m \in \mathbb{Z}$, where $Z_m(W_\bullet), Z_m({\rm CoKer}(f)) \in \mathcal{A}$. Since $\mathcal{A}$ is resolving, we get $Z_m(X_\bullet) \in \mathcal{A}$, and hence $X_\bullet \in \widetilde{\mathcal{A}}$.
\end{proof}

\begin{corollary}\label{coro:rel_derived}
There exist the following hereditary and cofibrantly generated abelian model structures on $\Ch(R)$,
\begin{align*}
\mathfrak{M}^{\rm GF}_{\rm dw} & := (\Ch(\mathcal{GF}_{(\mathcal{L},\nu)}),\mathscr{W}^{\rm GF}_{\rm dw},{\rm dg} \ \widetilde{\mathcal{GC}_{(\mathcal{L},\nu)}}), \\
\mathfrak{M}^{\rm GF}_{\rm ex} & := (\Ch(\mathcal{GF}_{(\mathcal{L},\nu)}) \cap \mathscr{E},\mathscr{W}^{\rm GF}_{\rm ex},{\rm dg} \ \widetilde{\mathcal{GC}_{(\mathcal{L},\nu)}}),
\end{align*}
whose homotopy categories are triangle equivalent to the derived categories 
\[
\frac{K(\mathcal{GF}_{(\mathcal{L},\nu)})}{\widetilde{\mathcal{GF}_{(\mathcal{L},\nu)}}} \text{ \ and \ } \frac{K_{\rm ac}(\mathcal{GF}_{(\mathcal{L},\nu)})}{\widetilde{\mathcal{GF}_{(\mathcal{L},\nu)}}}.
\] 
The cores of these model structures coincide with the class of contractible complexes with components in $\mathcal{L} \cap \mathcal{L}^\perp$.
\end{corollary}

We present the last of our models in Corollary \ref{coro:last_model} below. First, we need the following lemma.

\begin{lemma}\label{lem:contractible}
The pair $(\widetilde{\mathcal{L}},\widetilde{\mathcal{A}})$ is a product closed bicomplete duality pair of complexes where 
\begin{align*}
{}^{\perp}\widetilde{\mathcal{A}} \cap \widetilde{\mathcal{A}} & = \widetilde{\nu} = \{\text{contractible $R^{\rm o}$-complexes with components in } \nu\}.
\end{align*}
\end{lemma}

\begin{proof}
By \cite[Thm. 4.2.1 - preprint version]{TZhao}, $(\widetilde{\mathcal{L}},\widetilde{\mathcal{A}})$ is a perfect and symmetric (and so complete), product closed duality pair of complexes. On the other hand, since $({}^{\perp}\mathcal{A},\mathcal{A})$ is a hereditary complete cotorsion pair in $\mathsf{Mod}(R^{\rm o})$, by \cite[Thm. 2.4]{onaquestion} we have that $({}^{\perp}\widetilde{\mathcal{A}},\widetilde{\mathcal{A}})$ is a complete cotorsion pair in $\Ch(R^{\rm o})$, which is clearly hereditary.
\end{proof}

The existence of the following model structures are a consequence of Theorems \ref{dualidad_complejos} and \ref{theo:GF_model_structure_Ch} and Lemma \ref{lem:contractible}, and the descriptions of their homotopy categories follow as in the proof of Proposition \ref{prop:homotopy_models_KKac}.

\begin{corollary}\label{coro:last_model}
There exist the following hereditary and cofibrantly generated abelian model structures on $\Ch(R)$, 
\begin{align*}
\mathfrak{M}^{{\rm GF}}_{(\widetilde{\mathcal{L}},\widetilde{\nu})} & := (\Ch(\mathcal{GF}_{(\mathcal{L},\nu)}),\mathscr{W},{\rm dg} \ \widetilde{\mathcal{L}^{\perp}}), \\
\mathfrak{M}^{{\rm GF\mbox{-}ex}}_{(\widetilde{\mathcal{L}},\widetilde{\nu})} & := (\Ch(\mathcal{GF}_{(\mathcal{L},\nu)}) \cap \mathscr{E},\mathscr{W}_{\rm ex},{\rm dg} \ \widetilde{\mathcal{L}^{\perp}}).
\end{align*}
Their cores coincide with the subcategory of contractible $R$-complexes with components in $\mathcal{L} \cap \mathcal{L}^\perp$, and their homotopy categories are triangle equivalent to the derived categories 
\[
\frac{K(\mathcal{GF}_{(\mathcal{L},\nu)})}{\widetilde{\mathcal{L}}} \text{ \ and \ } \frac{K_{\rm ac}(\mathcal{GF}_{(\mathcal{L},\nu)})}{\widetilde{\mathcal{L}}}.
\] 
\end{corollary}

Regarding the model structures appearing in the previous corollary, the class $\mathcal{L}$ has a more important role in the description of their fibrant objects. Also, the hereditary cotorsion pair $(\mathcal{L},\mathcal{L}^\perp)$, cogenerated by a set, induces the hereditary and cofibrantly generated abelian model structure on $\Ch(R)$ given by 
\[
\mathfrak{M}^{\mathcal{L}}_{\rm dg} := ({\rm dg} \ \widetilde{\mathcal{L}},\mathscr{E}, {\rm dg} \ \widetilde{\mathcal{L}^\perp}),
\] 
whose homotopy category is triangle equivalent to $D(R)$. Moreover, using the cotorsion pairs $(\Ch(\mathcal{L}),(\Ch(\mathcal{L}))^\perp)$ and $(\Ch(\mathcal{L}) \cap \mathscr{E},(\Ch(\mathcal{L}) \cap \mathscr{E})^\perp)$, compatible with $(\widetilde{\mathcal{L}}, {\rm dg} \ \widetilde{\mathcal{L}^\perp})$, one can also obtain the following result.

\begin{corollary}\label{coro:seriously_this_is_the_last}
There exist the following hereditary and cofibrantly generated abelian model structures on $\Ch(R)$,
\begin{align*}
\mathfrak{M}^{\mathcal{L}}_{\rm dw} & := (\Ch(\mathcal{L}),\mathscr{W}^{\mathcal{L}}_{\rm dw},{\rm dg} \ \widetilde{\mathcal{L}^{\perp}}), \\ 
\mathfrak{M}^{\mathcal{L}}_{\rm ex} & := (\Ch(\mathcal{L}) \cap \mathscr{E},\mathscr{W}^{\mathcal{L}}_{\rm ex},{\rm dg} \ \widetilde{\mathcal{L}^{\perp}}).
\end{align*}
Their cores coincide with the subcategory of contractible $R$-complexes with components in $\mathcal{L} \cap \mathcal{L}^\perp$, and their homotopy categories are triangle equivalent to the derived categories 
\[
\frac{K(\mathcal{L})}{\widetilde{\mathcal{L}}} \text{ \ and \ } \frac{K_{\rm ac}(\mathcal{L})}{\widetilde{\mathcal{L}}}.
\] 
\end{corollary}

%%%%%%%%%%%%%%%%%%%%%%%%%%%%%%%%%%%%%%%%%%%%%
%%%%%%%%%%%%%%%%%%%%%%%%%%%%%%%%%%%%%%%%%%%%%

\subsection*{Recollements of homotopy categories}

Recollements were introduced in \cite{PerverseSheaves} by Beilinson, Bernstein and Deligne. These are arrays of triangulated categories and functors described as follows.

Let $\mathcal{T}' \xrightarrow{F} \mathcal{T} \xrightarrow{G} \mathcal{T}''$ be a sequence of triangulated categories and functors. One says that this sequence is a \emph{localization sequence} if there exist right adjoints $F_\rho \colon \mathcal{T} \to \mathcal{T}'$ and $G_\rho \colon \mathcal{T}'' \to \mathcal{T}$ such that the following three conditions are satisfied:
\begin{enumerate}
\item $F_\rho \circ F$ is naturally isomorphic to ${\rm id}_{\mathcal{T}'}$.

\item $G \circ G_\rho$ is naturally isomorphic to ${\rm id}_{\mathcal{T}''}$.

\item For any object $T \in \mathcal{T}$, one has that $G(T) = 0$ if, and only if, $T \simeq F(T')$ for some $T' \in \mathcal{T}'$. 
\end{enumerate}
We display the previous localization sequence as the following diagram:
\[
\begin{tikzpicture}[description/.style={fill=white,inner sep=2pt}]
\matrix (m) [matrix of math nodes, row sep=2.5em, column sep=2.5em, text height=1.25ex, text depth=0.25ex]
{ 
\mathcal{T}' & \mathcal{T} & \mathcal{T}'' \\
};
\path[->]
(m-1-1) edge node[above] {\footnotesize$F$} (m-1-2) (m-1-2) edge node[above] {\footnotesize$G$} (m-1-3) 
(m-1-2) edge [bend left=50] node[below] {\footnotesize$F_\rho$} (m-1-1)
(m-1-3) edge [bend left=50] node[below] {\footnotesize$G_\rho$} (m-1-2)
;
\end{tikzpicture}
\]
\emph{Colocalization sequences} are defined dually, by the existence of left adjoints $F_\lambda \colon \mathcal{T} \to \mathcal{T}'$ and $G_\lambda \colon \mathcal{T}'' \to \mathcal{T}$ satisfying similar properties, and will be displayed as
\[
\begin{tikzpicture}[description/.style={fill=white,inner sep=2pt}]
\matrix (m) [matrix of math nodes, row sep=2.5em, column sep=2.5em, text height=1.25ex, text depth=0.25ex]
{ 
\mathcal{T}' & \mathcal{T} & \mathcal{T}'' \\
};
\path[->]
(m-1-1) edge node[above] {\footnotesize$F$} (m-1-2) (m-1-2) edge node[above] {\footnotesize$G$} (m-1-3) 
(m-1-2) edge [bend right=50] node[above] {\footnotesize$F_\lambda$} (m-1-1)
(m-1-3) edge [bend right=50] node[above] {\footnotesize$G_\lambda$} (m-1-2)
;
\end{tikzpicture}
\]
The sequence $\mathcal{T}' \xrightarrow{F} \mathcal{T} \xrightarrow{G} \mathcal{T}''$ is a \emph{recollement} if it is both a localization and a colocalization sequence. It will be displayed as 
\[
\begin{tikzpicture}[description/.style={fill=white,inner sep=2pt}]
\matrix (m) [matrix of math nodes, row sep=2.5em, column sep=2.5em, text height=1.25ex, text depth=0.25ex]
{ 
\mathcal{T}' & \mathcal{T} & \mathcal{T}'' \\
};
\path[->]
(m-1-1) edge node[above] {\footnotesize$F$} (m-1-2) (m-1-2) edge node[above] {\footnotesize$G$} (m-1-3) 
(m-1-2) edge [bend left=50] node[below] {\footnotesize$F_\rho$} (m-1-1)
(m-1-3) edge [bend left=50] node[below] {\footnotesize$G_\rho$} (m-1-2)
(m-1-2) edge [bend right=50] node[above] {\footnotesize$F_\lambda$} (m-1-1)
(m-1-3) edge [bend right=50] node[above] {\footnotesize$G_\lambda$} (m-1-2)
;
\end{tikzpicture}
\]
One may think of (co)localization sequences as split exact sequences. The ``splittings'' $F_\rho$ and $G_\rho$, for instance, attach to $\mathcal{T}$ the triangulated categories $\mathcal{T}'$ and $\mathcal{T}''$. 

One can obtain recollements of homotopy categories of three hereditary abelian model structures under certain compatibility conditions. These are specified in \cite[Thm. 3.6]{GillespieMock}. Specifically, the \emph{Left Recollement Theorem} states that if 
\[
\mathfrak{M}_1 = (\mathcal{Q}_1,\mathcal{W}_1,\mathcal{R}_1), \text{ \ } \mathfrak{M}_2 = (\mathcal{Q}_2,\mathcal{W}_2,\mathcal{R}_2) \text{ \ and \ } \mathfrak{M}_3 = (\mathcal{Q}_3,\mathcal{W}_3,\mathcal{R}_3)
\]
are hereditary abelian model structures with the same core and such that $\mathcal{Q}_2 = \mathcal{W}_3 \cap \mathcal{Q}_1$ and $\mathcal{Q}_3 \subseteq \mathcal{Q}_1$, then there is a recollement
\[
\begin{tikzpicture}[description/.style={fill=white,inner sep=2pt}]
\matrix (m) [matrix of math nodes, row sep=2.5em, column sep=2.5em, text height=1.25ex, text depth=0.25ex]
{ 
{\rm Ho}(\mathfrak{M}_2) & {\rm Ho}(\mathfrak{M}_1) & {\rm Ho}(\mathfrak{M}_3) \\
};
\path[->]
(m-1-1) edge (m-1-2) (m-1-2) edge (m-1-3) 
(m-1-2) edge [bend left=50] (m-1-1)
(m-1-3) edge [bend left=50] (m-1-2)
(m-1-2) edge [bend right=50] (m-1-1)
(m-1-3) edge [bend right=50] (m-1-2)
;
\end{tikzpicture}
\]
The functors and adjunctions in this diagram are specified in \cite{GillespieMock}. 

The following is a consequence of Corollaries \ref{coro:derived}, \ref{coro:rel_derived}, \ref{coro:last_model} and \ref{coro:seriously_this_is_the_last}. It suffices the check the hypotheses in the Left Recollement Theorem to the following triples of model structures:
\begin{align*}
\mathfrak{M}^{\rm GF}_{\rm dw} & := (\Ch(\mathcal{GF}_{(\mathcal{L},\nu)}),\mathscr{W}^{\rm GF}_{\rm dw},{\rm dg} \ \widetilde{\mathcal{GC}_{(\mathcal{L},\nu)}}), \\ 
\mathfrak{M}^{\rm GF}_{\rm ex} & := (\Ch(\mathcal{GF}_{(\mathcal{L},\nu)}) \cap \mathscr{E},\mathscr{W}^{\rm GF}_{\rm ex},{\rm dg} \ \widetilde{\mathcal{GC}_{(\mathcal{L},\nu)}}), \\ 
\mathfrak{M}^{\rm GF}_{\rm dg} & := ({\rm dg} \ \widetilde{\mathcal{GF}_{(\mathcal{L},\nu)}}, \mathscr{E}, {\rm dg} \ \widetilde{\mathcal{GC}_{(\mathcal{L},\nu)}}). \\ {} \\
\mathfrak{M}^{{\rm GF}}_{(\widetilde{\mathcal{L}},\widetilde{\nu})} & := (\Ch(\mathcal{GF}_{(\mathcal{L},\nu)}),\mathscr{W},{\rm dg} \ \widetilde{\mathcal{L}^{\perp}}), \\
\mathfrak{M}^{{\rm GF\mbox{-}ex}}_{(\widetilde{\mathcal{L}},\widetilde{\nu})} & := (\Ch(\mathcal{GF}_{(\mathcal{L},\nu)}) \cap \mathscr{E},\mathscr{W}_{\rm ex},{\rm dg} \ \widetilde{\mathcal{L}^{\perp}}), \\
\mathfrak{M}^{\mathcal{L}}_{\rm dg} & := ({\rm dg} \ \widetilde{\mathcal{L}},\mathscr{E}, {\rm dg} \ \widetilde{\mathcal{L}^\perp}). \\ {} \\
\mathfrak{M}^{\mathcal{L}}_{\rm dw} & := (\Ch(\mathcal{L}),\mathscr{W}^{\mathcal{L}}_{\rm dw},{\rm dg} \ \widetilde{\mathcal{L}^{\perp}}), \\
\mathfrak{M}^{\mathcal{L}}_{\rm ex} & := (\Ch(\mathcal{L}) \cap \mathscr{E},\mathscr{W}^{\mathcal{L}}_{\rm ex},{\rm dg} \ \widetilde{\mathcal{L}^{\perp}}), \\
\mathfrak{M}^{\mathcal{L}}_{\rm dg} & := ({\rm dg} \ \widetilde{\mathcal{L}},\mathscr{E}, {\rm dg} \ \widetilde{\mathcal{L}^\perp}).
\end{align*}

\begin{corollary}\label{coro:reco1}
There exist the following recollements:
\[
\begin{tikzpicture}[description/.style={fill=white,inner sep=2pt}]
\matrix (m) [matrix of math nodes, row sep=2.5em, column sep=2.5em, text height=2.25ex, text depth=2.25ex]
{ 
\frac{K_{\rm ac}(\mathcal{GF}_{(\mathcal{L},\nu)})}{\widetilde{\mathcal{GF}_{(\mathcal{L},\nu)}}} & \frac{K(\mathcal{GF}_{(\mathcal{L},\nu)})}{\widetilde{\mathcal{GF}_{(\mathcal{L},\nu)}}} & D(R), \\
};
\path[->]
(m-1-1) edge (m-1-2) (m-1-2) edge (m-1-3) 
(m-1-2) edge [bend left=50] (m-1-1)
(m-1-3) edge [bend left=50] (m-1-2)
(m-1-2) edge [bend right=50] (m-1-1)
(m-1-3) edge [bend right=50] (m-1-2)
;
\end{tikzpicture}
\]
\[
\begin{tikzpicture}[description/.style={fill=white,inner sep=2pt}]
\matrix (m) [matrix of math nodes, row sep=2.5em, column sep=2.5em, text height=2.25ex, text depth=2.25ex]
{ 
\frac{K_{\rm ac}(\mathcal{GF}_{(\mathcal{L},\nu)})}{\widetilde{\mathcal{L}}} & \frac{K(\mathcal{GF}_{(\mathcal{L},\nu)})}{\widetilde{\mathcal{L}}} & D(R), \\
};
\path[->]
(m-1-1) edge (m-1-2) (m-1-2) edge (m-1-3) 
(m-1-2) edge [bend left=50] (m-1-1)
(m-1-3) edge [bend left=50] (m-1-2)
(m-1-2) edge [bend right=50] (m-1-1)
(m-1-3) edge [bend right=50] (m-1-2)
;
\end{tikzpicture}
\]
\[
\begin{tikzpicture}[description/.style={fill=white,inner sep=2pt}]
\matrix (m) [matrix of math nodes, row sep=2.5em, column sep=2.5em, text height=2.25ex, text depth=2.25ex]
{ 
\frac{K_{\rm ac}(\mathcal{L})}{\widetilde{\mathcal{L}}} & \frac{K(\mathcal{L})}{\widetilde{\mathcal{L}}} & D(R). \\
};
\path[->]
(m-1-1) edge (m-1-2) (m-1-2) edge (m-1-3) 
(m-1-2) edge [bend left=50] (m-1-1)
(m-1-3) edge [bend left=50] (m-1-2)
(m-1-2) edge [bend right=50] (m-1-1)
(m-1-3) edge [bend right=50] (m-1-2)
;
\end{tikzpicture}
\]
\end{corollary}

%%%%%%%%%%%%%%%%%%%%%%%%%%%%%%%%%%%%%%%%%%%%%
%%%%%%%%%%%%%%%%%%%%%%%%%%%%%%%%%%%%%%%%%%%%%

\subsection*{Quillen adjunctions between model structures on complexes}

We conclude this article comparing the recollements in Corollary \ref{coro:reco1} via derived adjunctions from Quillen adjunctions. Let us first recall these concepts.

Let $\mathfrak{M}$ and $\mathfrak{M}'$ be two model categories and $(F,U,\varphi)$ be an adjunction from $\mathfrak{M}$ to $\mathfrak{M}'$, where $F \colon \mathfrak{M} \to \mathfrak{M}'$, $U \colon \mathfrak{M}' \to \mathfrak{M}$, $F$ is a left adjoint of $U$, and $\varphi$ is a natural isomorphism $\Hom(FX,Y) \to \Hom(X,UY)$ expressing $F$ as a left adjoint of $U$. We call $(F,U,\varphi)$ a \emph{Quillen adjunction} if $F$ is a \emph{left Quillen functor}, that is, if $F$ preserves cofibrations and trivial cofibrations. Equivalently, $(F,U,\varphi)$ is a Quillen adjunction if $U$ is a \emph{right Quillen functor}, that is, $U$ preserves fibrations and trivial fibrations (see \cite[Lem. 1.3.4]{HoveyBook}). 

In the previous situation, one can define the \emph{total left} and \emph{total right derived functors} of $F$ and $U$, denoted $LF \colon {\rm Ho}(\mathfrak{M}) \to {\rm Ho}(\mathfrak{M}')$ and $RU \colon {\rm Ho}(\mathfrak{M}') \to {\rm Ho}(\mathfrak{M})$, as the composites
\begin{align*}
LF & = {\rm Ho}(\mathfrak{M}) \xrightarrow{{\rm Ho}(Q)} {\rm Ho}(\mathfrak{M}_{\rm cof}) \xrightarrow{{\rm Ho}(F)} {\rm Ho}(\mathfrak{M}'), \\ 
RU & = {\rm Ho}(\mathfrak{M}') \xrightarrow{{\rm Ho}(R)} {\rm Ho}(\mathfrak{M}'_{\rm fib}) \xrightarrow{{\rm Ho}(U)} {\rm Ho}(\mathfrak{M}), 
\end{align*} 
where $\mathfrak{M}_{\rm cof}$ and $\mathfrak{M}'_{\rm fib}$ are the restrictions of $\mathfrak{M}$ and $\mathfrak{M}'$ to their cofibrant and fibrant objects, respectively, $Q$ is the cofibrant replacement functor, and $R$ is the fibrant replacement functor. By \cite[Lem. 1.3.10]{HoveyBook}, $LF$ and $RU$ are part of an adjunction $L(F,U,\varphi) := (LF,RU,R\varphi)$ known as the \emph{derived adjunction}. The following result follows the spirit of \cite[Prop. 4.6 \& 4.7]{EIP20}

\begin{proposition}\label{prop:derived_ad}
The following assertions hold true:
\begin{enumerate}
\item The identity ${\rm id} \colon (\Ch(R),\mathfrak{M}^{{\rm GF}}_{(\widetilde{\mathcal{L}},\widetilde{\nu})}) \to (\Ch(R),\mathfrak{M}^{\rm GF}_{\rm dw})$ is a left Quillen functor, and so there is a derived adjunction between 
\begin{align*}
\frac{K(\mathcal{GF}_{(\mathcal{L},\nu)})}{\widetilde{\mathcal{L}}} & & \text{and} & & \frac{K(\mathcal{GF}_{(\mathcal{L},\nu)})}{\widetilde{\mathcal{GF}_{(\mathcal{L},\nu)}}}.
\end{align*}

\item The identity ${\rm id} \colon (\Ch(R),\mathfrak{M}^{\mathcal{L}}_{\rm dw}) \to (\Ch(R),\mathfrak{M}^{{\rm GF}}_{(\widetilde{\mathcal{L}},\widetilde{\nu})})$ is a left Quillen functor, and so there is a derived adjunction between 
\begin{align*}
\frac{K(\mathcal{L})}{\widetilde{\mathcal{L}}} & & \text{and} & & \frac{K(\mathcal{GF}_{(\mathcal{L},\nu)})}{\widetilde{\mathcal{L}}}.
\end{align*}

\item The identity ${\rm id} \colon (\Ch(R),\mathfrak{M}^{{\rm GF\mbox{-}ex}}_{(\widetilde{\mathcal{L}},\widetilde{\nu})}) \to (\Ch(R),\mathfrak{M}^{\rm GF}_{\rm ex})$ is a left Quillen functor, and so there is a derived adjunction between
\begin{align*}
\frac{K_{\rm ac}(\mathcal{GF}_{(\mathcal{L},\nu)})}{\widetilde{\mathcal{L}}} & & \text{and} & & \frac{K_{\rm ac}(\mathcal{GF}_{(\mathcal{L},\nu)})}{\widetilde{\mathcal{GF}_{(\mathcal{L},\nu)}}}.
\end{align*}

\item The identity ${\rm id} \colon (\Ch(R),\mathfrak{M}^{\mathcal{L}}_{\rm ex}) \to (\Ch(R),\mathfrak{M}^{{\rm GF}\mbox{-}ex}_{(\widetilde{\mathcal{L}},\widetilde{\nu})})$ is a left Quillen functor, and so there is a derived adjunction between 
\begin{align*}
\frac{K_{\rm ac}(\mathcal{L})}{\widetilde{\mathcal{L}}} & & \text{and} & & \frac{K_{\rm ac}(\mathcal{GF}_{(\mathcal{L},\nu)})}{\widetilde{\mathcal{L}}}.
\end{align*}
\end{enumerate}
\end{proposition}

\begin{proof}
We only prove part (1), as the rest of the statements follow in the same way. First, it is immediate that ${\rm id} \colon (\Ch(R),\mathfrak{M}^{{\rm GF}}_{(\widetilde{\mathcal{L}},\widetilde{\nu})}) \to (\Ch(R),\mathfrak{M}^{\rm GF}_{\rm dw})$ is a left adjoint to the identity ${\rm id} \colon (\Ch(R),\mathfrak{M}^{\rm GF}_{\rm dw}) \to (\Ch(R),\mathfrak{M}^{{\rm GF}}_{(\widetilde{\mathcal{L}},\widetilde{\nu})})$. Secondly, the structures $\mathfrak{M}^{{\rm GF}}_{(\widetilde{\mathcal{L}},\widetilde{\nu})}$ and $\mathfrak{M}^{\rm GF}_{\rm dw}$ have the same class of cofibrations (monomorphisms with cokernel in $\Ch(\mathcal{GF}_{(\mathcal{L},\nu)})$), and every trivial cofibration in the former structure (a monomorphism with cokernel in $\widetilde{\mathcal{L}}$) is a trivial cofibration in the latter (a monomorphism with cokernel in $\widetilde{\mathcal{GF}_{(\mathcal{L},\nu)}}$). Hence, the result follows. 
\end{proof}

%\begin{example}
%In particular, for the product closed bicomplete duality pair 
%\[
%(\mathcal{LV}(R),\mathcal{AC}(R^{\rm o}))
%\] 
%(see Example \ref{ex:FPinfty}), all of the previous results can be applied to obtain model category structures on $\mathsf{Mod}(R)$ and $\Ch(R)$. The derived subcategories 
%\begin{align*}
%\frac{K(\mathcal{LV}(R))}{\widetilde{\mathcal{LV}(R)}} & & \text{and} & & \frac{K_{\rm ac}(\mathcal{LV}(R))}{\widetilde{\mathcal{LV}(R)}}, \\
%\frac{K(\mathcal{GF}_{(\mathcal{LV}(R),{}^\perp\mathcal{AC}(R^{\rm o}) \cap \mathcal{AC}(R^{\rm o}))})}{\widetilde{\mathcal{LV}(R)}} & & \text{and} & & \frac{K_{\rm ac}(\mathcal{GF}_{(\mathcal{LV}(R),{}^\perp\mathcal{AC}(R^{\rm o}) \cap \mathcal{AC}(R^{\rm o}))})}{\widetilde{\mathcal{LV}(R)}}, \\
%\frac{K(\mathcal{GF}_{(\mathcal{LV}(R),{}^\perp\mathcal{AC}(R^{\rm o}) \cap \mathcal{AC}(R^{\rm o}))})}{\widetilde{\mathcal{GF}_{(\mathcal{LV}(R),{}^\perp\mathcal{AC}(R^{\rm o}) \cap \mathcal{AC}(R^{\rm o}))}}} & & \text{and} & & \frac{K_{\rm ac}(\mathcal{GF}_{(\mathcal{LV}(R),{}^\perp\mathcal{AC}(R^{\rm o}) \cap \mathcal{AC}(R^{\rm o}))})}{\widetilde{\mathcal{GF}_{(\mathcal{LV}(R),{}^\perp\mathcal{AC}(R^{\rm o}) \cap \mathcal{AC}(R^{\rm o}))}}}, 
%\end{align*}
%are descriptions of the homotopy categories of these model structures, and are related via recollements and Quillen adjunctions. 

%Similar conclusions are also valid for the product closed bicomplete duality pair 
%\[
%(\mathcal{WF}_n(R),\mathcal{WI}_n(R^{\rm o}))
%\] 
%from Example \ref{ex:weakn}. 
%\end{example}

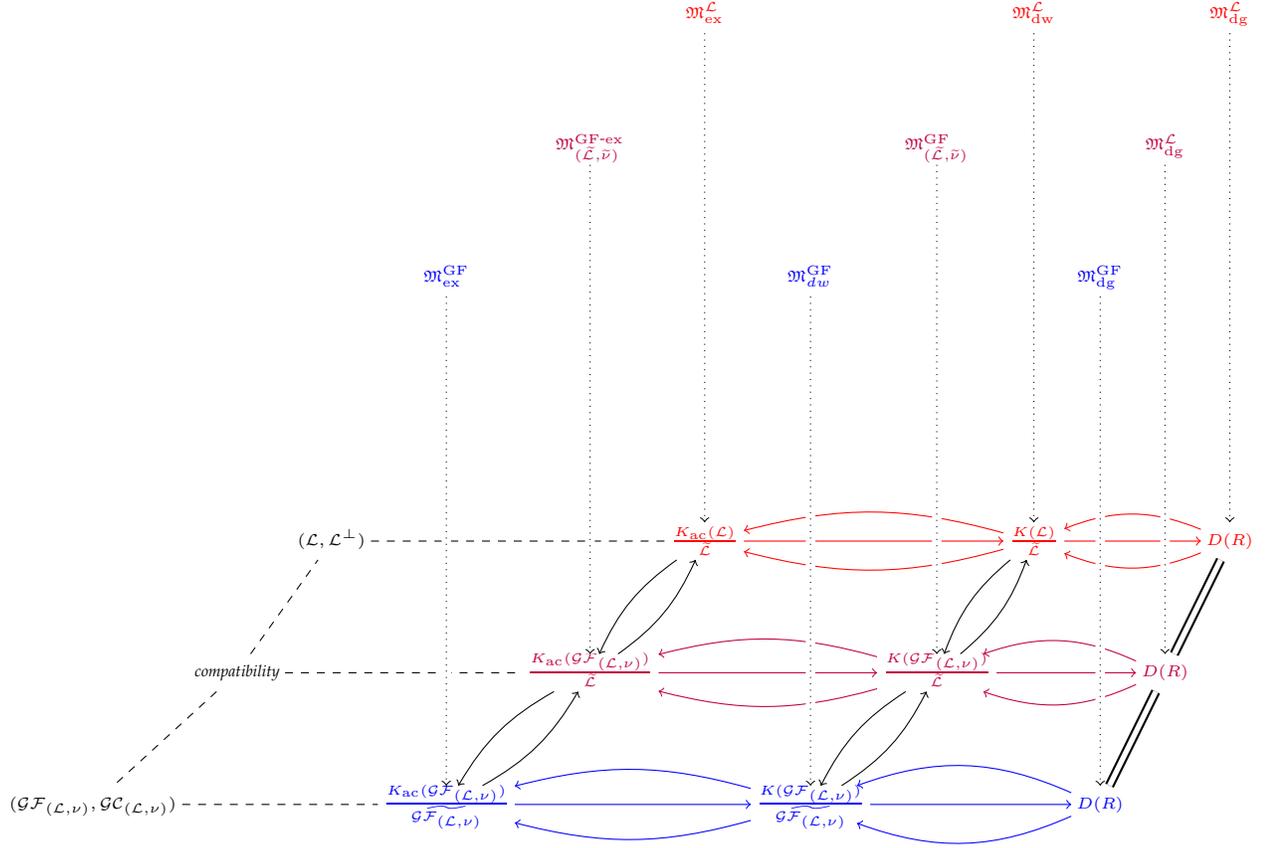
\begin{figure}[h!]
\tiny
\begin{tikzpicture}[description/.style={fill=white,inner sep=2pt}]
\matrix (m) [matrix of math nodes, row sep=6em, column sep=0.5em, text height=2.25ex, text depth=1.25ex]
{ 
{} & {} & {} & & & {\color{red}{ \mathfrak{M}^{\mathcal{L}}_{\rm ex} }} & & & {\color{red}{ \mathfrak{M}^{\mathcal{L}}_{\rm dw} }} & & & {\color{red}{ \mathfrak{M}^{\mathcal{L}}_{\rm dg} }} \\
{} & {} & & & {\color{purple}{ \mathfrak{M}^{{\rm GF\mbox{-}ex}}_{(\widetilde{\mathcal{L}},\widetilde{\nu})} }} & & & {\color{purple}{ \mathfrak{M}^{{\rm GF}}_{(\widetilde{\mathcal{L}},\widetilde{\nu})} }} & & & {\color{purple}{ \mathfrak{M}^{\mathcal{L}}_{\rm dg} }} \\
{} & & & {\color{blue}{ \mathfrak{M}^{\rm GF}_{\rm ex} }} & & & {\color{blue}{ \mathfrak{M}^{\rm GF}_{dw} }} & & & {\color{blue}{ \mathfrak{M}^{\rm GF}_{\rm dg} }} \\
{} \\
{} & {} & (\mathcal{L},\mathcal{L}^\perp) & & & {\color{red}{ \frac{K_{\rm ac}(\mathcal{L})}{\widetilde{\mathcal{L}}} }} & & & {\color{red}{ \frac{K(\mathcal{L})}{\widetilde{\mathcal{L}}} }} & & & {\color{red}{ D(R) }} \\
{} & \emph{compatibility} & & & {\color{purple}{ \frac{K_{\rm ac}(\mathcal{GF}_{(\mathcal{L},\nu)})}{\widetilde{\mathcal{L}}} }} & & & {\color{purple}{ \frac{K(\mathcal{GF}_{(\mathcal{L},\nu)})}{\widetilde{\mathcal{L}}} }} & & & {\color{purple}{ D(R) }} \\
(\mathcal{GF}_{(\mathcal{L},\nu)},\mathcal{GC}_{(\mathcal{L},\nu)}) & & & {\color{blue}{ \frac{K_{\rm ac}(\mathcal{GF}_{(\mathcal{L},\nu)})}{\widetilde{\mathcal{GF}_{(\mathcal{L},\nu)}}} }} & & & {\color{blue}{ \frac{K(\mathcal{GF}_{(\mathcal{L},\nu)})}{\widetilde{\mathcal{GF}_{(\mathcal{L},\nu)}}} }} & & & {\color{blue}{ D(R) }} \\
};
\path[->]
(m-5-6) edge [red] (m-5-9) (m-5-9) edge [red] (m-5-12)
(m-5-9) edge [red, bend left=15] (m-5-6)
(m-5-9) edge [red, bend right=15] (m-5-6)
(m-5-12) edge [red, bend left=22] (m-5-9)
(m-5-12) edge [red, bend right=22] (m-5-9)
(m-5-6) edge [bend right=15] (m-6-5)
(m-6-5) edge [bend right=15] (m-5-6)
(m-5-9) edge [bend right=15] (m-6-8)
(m-6-8) edge [bend right=15] (m-5-9)
(m-6-5) edge [purple] (m-6-8) (m-6-8) edge [purple] (m-6-11)
(m-6-8) edge [purple, bend left=15] (m-6-5)
(m-6-8) edge [purple, bend right=15] (m-6-5)
(m-6-11) edge [purple, bend left=22] (m-6-8)
(m-6-11) edge [purple, bend right=22] (m-6-8)
(m-6-5) edge [bend right=15] (m-7-4)
(m-7-4) edge [bend right=15] (m-6-5)
(m-6-8) edge [bend right=15] (m-7-7)
(m-7-7) edge [bend right=15] (m-6-8)
(m-7-4) edge [blue] (m-7-7) (m-7-7) edge [blue] (m-7-10)
(m-7-7) edge [blue, bend left=15] (m-7-4)
(m-7-7) edge [blue, bend right=15] (m-7-4)
(m-7-10) edge [blue, bend left=22] (m-7-7)
(m-7-10) edge [blue, bend right=22] (m-7-7)
;
\path[dashed,-]
(m-6-2) edge (m-5-3) edge (m-7-1)
(m-6-2) edge (m-6-5) (m-5-3) edge (m-5-6) (m-7-1) edge (m-7-4)
;
\path[-,font=\scriptsize]
(m-1-6) edge [white, double, thick, double distance=2pt] (m-5-6)
(m-1-9) edge [white, double, thick, double distance=2pt] (m-5-9)
(m-1-12) edge [white, double, thick, double distance=2pt] (m-5-12)
(m-2-5) edge [white, double, thick, double distance=2pt] (m-6-5)
(m-2-8) edge [white, double, thick, double distance=2pt] (m-6-8)
(m-2-11) edge [white, double, thick, double distance=2pt] (m-6-11)
(m-3-4) edge [white, double, thick, double distance=2pt] (m-7-4) 
(m-3-7) edge [white, double, thick, double distance=2pt] (m-7-7) 
(m-3-10) edge [white, double, thick, double distance=2pt] (m-7-10)
(m-5-12) edge [double, thick, double distance=2pt] (m-6-11)
(m-6-11) edge [double, thick, double distance=2pt] (m-7-10)
;
\path[dotted,->]
(m-1-6) edge (m-5-6) (m-1-9) edge (m-5-9) (m-1-12) edge (m-5-12)
(m-2-5) edge (m-6-5) (m-2-8) edge (m-6-8) (m-2-11) edge (m-6-11)
(m-3-4) edge (m-7-4) (m-3-7) edge (m-7-7) (m-3-10) edge (m-7-10)
;
\end{tikzpicture}
\caption{Display of all model structures on $\Ch(R)$, their homotopy categories, recollements and derived adjunctions between them.}
\end{figure}

%%%%%%%%%%%%%%%%%%%%%%%%%%%%%%%%%%%%%%%%%%%%%%%%%%
%%%%%%%%%%%%%%%%%%%%%%%%%%%%%%%%%%%%%%%%%%%%%%%%%%

\section*{Acknowledgements}

The authors want to thank professor Raymundo Bautista (Centro de Ciencias Matem\'aticas - Universidad Nacional Aut\'onoma de M\'exico) for several helpful discussions on the results of this article.

%%%%%%%%%%%%%%%%%%%%%%%%%%%%%%%%%%%%%
%%%%%%%%%%%%%%%%%%%%%%%%%%%%%%%%%%%%%
%%%%%%%%%%%%%%%%%%%%%%%%%%%%%%%%%%%%%
%%%%%%%%%%%%%%%%%%%%%%%%%%%%%%%%%%%%%

\section*{Funding}
The first author was supported by Postdoctoral Fellowship from Programa de Desarrollo de las Ciencias B\'asicas (PEDECIBA) - Ministerio de Educaci\'on y Cultura de la Rep\'ublica Oriental del Uruguay, and is currently sponsored by a postdoctoral fellowship from CONACHyT - Universidad Michoacana de San Nicol\'as de Hidalgo. The second author is supported by the Agencia Nacional de Investigaci\'on e Innovaci\'on (ANII) and PEDECIBA.

%%%%%%%%%%%%%%%%%%%%%%%%%%%%%%%%%%%%%
%%%%%%%%%%%%%%%%%%%%%%%%%%%%%%%%%%%%%
%%%%%%%%%%%%%%%%%%%%%%%%%%%%%%%%%%%%%
%%%%%%%%%%%%%%%%%%%%%%%%%%%%%%%%%%%%%

\bibliographystyle{plain}
\bibliography{biblio_relativeGF}

\end{document}